\documentclass{amsart}
\usepackage{tikz, hyperref, amssymb}

\newtheorem{thm}{Theorem}[section]
\newtheorem{cor}[thm]{Corollary}
\newtheorem{lem}[thm]{Lemma}
\newtheorem{prop}[thm]{Proposition}

\theoremstyle{definition}
\newtheorem{dfn}[thm]{Definition}

\newtheorem{ntn}[thm]{Notation}

\theoremstyle{remark}
\newtheorem{rmk}[thm]{Remark}

\newtheorem{example}[thm]{Example}

\newcommand{\extshrt}[3]{X_{#1}}
\newcommand{\Extshrt}[3]{\Xx_{#1}}

\numberwithin{equation}{section}

\newcommand{\CC}{\mathbb{C}}
\newcommand{\FF}{\mathbb{F}}

\newcommand{\NN}{\mathbb{N}}

\newcommand{\TT}{\mathbb{T}}
\newcommand{\ZZ}{\mathbb{Z}}

\newcommand{\Gg}{\mathcal{G}}

\newcommand{\Kk}{\mathcal{K}}

\newcommand{\Pp}{\mathcal{P}}
\newcommand{\Qq}{\mathcal{Q}}
\newcommand{\Uu}{\mathcal{U}}
\newcommand{\Vv}{\mathcal{V}}
\newcommand{\Ww}{\mathcal{W}}
\newcommand{\Xx}{\mathcal{X}}
\newcommand{\Yy}{\mathcal{Y}}
\newcommand{\Zz}{\mathcal{Z}}

\newcommand{\id}{\operatorname{id}}

\newcommand{\image}{\operatorname{Im}}
\newcommand{\Aut}{\operatorname{Aut}}

\newcommand{\coker}{\operatorname{coker}}
\newcommand{\Hom}{\operatorname{Hom}}
\newcommand{\Ext}{\operatorname{Ext}}
\newcommand{\MCE}{\operatorname{MCE}}
\newcommand{\Obj}{\operatorname{Obj}}
\newcommand{\lsp}{\operatorname{span}}
\newcommand{\clsp}{\overline{\lsp}}
\newcommand{\dom}{\operatorname{dom}}
\newcommand{\cod}{\operatorname{cod}}

\newcommand{\Bcub}[1]{B^{#1}}
\newcommand{\Ccub}[1]{C^{#1}}
\newcommand{\Hcub}[1]{H^{#1}}
\newcommand{\Zcub}[1]{Z^{#1}}
\newcommand{\dcub}[1]{\delta^{#1}}
\newcommand{\Bcat}[1]{\underline{B}^{#1}}
\newcommand{\Ccat}[1]{\underline{C}^{#1}}
\newcommand{\Hcat}[1]{\underline{H}^{#1}}
\newcommand{\Zcat}[1]{\underline{Z}^{#1}}
\newcommand{\dcat}[1]{\underline{\delta}^{#1}}
\newcommand{\Bgpd}[1]{\tilde{B}^{#1}}
\newcommand{\Cgpd}[1]{\tilde{C}^{#1}}
\newcommand{\Hgpd}[1]{\tilde{H}^{#1}}
\newcommand{\Zgpd}[1]{\tilde{Z}^{#1}}
\newcommand{\dgpd}[1]{\tilde{\delta}^{#1}}

\title[Twisted $k$-graph algebras]{On twisted higher-rank graph $C^*$-algebras}

\author{Alex Kumjian}
\address{Alex Kumjian\\ Department of Mathematics (084)\\ University
of Nevada\\ Reno NV 89557-0084\\ USA} \email{alex@unr.edu}

\author{David Pask}
\address{David Pask, Aidan Sims\\ School of Mathematics and
Applied Statistics  \\
University of Wollongong\\
NSW  2522\\
AUSTRALIA} \email{dpask@uow.edu.au, asims@uow.edu.au}
\author{Aidan Sims}

\thanks{This research was supported by the ARC. Part of the work was completed while the first author was 
employed at the University of Wollongong on the ARC grants DP0984339 and DP0984360.}



\date{\today}

\dedicatory{Dedicated to Marc A. Rieffel on the occasion of his 75th birthday}

\begin{document}

\begin{abstract}
We define the categorical cohomology of a $k$-graph $\Lambda$ and show that the first three terms in this cohomology are isomorphic to the corresponding terms in the cohomology defined in our previous paper. This leads to an alternative characterisation of the twisted $k$-graph $C^*$-algebras introduced there. We prove a gauge-invariant uniqueness theorem and use it to show that every twisted $k$-graph $C^*$-algebra is isomorphic to a twisted groupoid $C^*$-algebra. We deduce criteria for simplicity, prove a Cuntz-Krieger uniqueness theorem and establish that all twisted $k$-graph $C^*$-algebras are nuclear and belong to the bootstrap class.
\end{abstract}

\maketitle

\section{Introduction}

Higher-rank graphs, or $k$-graphs, are $k$-dimensional analogues of directed graphs which were
introduced by the first two authors \cite{KP2000} to provide combinatorial models for the
higher-rank Cuntz-Krieger algebras investigated by Robertson and Steger in \cite{RS1999a}. The
structure theory of $k$-graph $C^*$-algebras is becoming quite well understood
\cite{DavidsonYang:NYJM09, Evans2008, FarthingMuhlyEtAl:sf05, KangPask:xx11, RobertsonSims:blms07},
and the class of $k$-graph algebras has been shown to contain many interesting examples \cite{kps,
PRRS2006}.

In \cite{kps3} we introduced a homology theory $H_* ( \Lambda )$ for each $k$-graph $\Lambda$ and
the corresponding cohomology $H^* ( \Lambda , A )$ with coefficients in an abelian group $A$. We
proved a number of fundamental results providing tools for calculating homology, and showed that
the homology of a $k$-graph is naturally isomorphic to that of its topological realisation. Of most
interest to us was to show how, given a $k$-graph and a $\TT$-valued $2$-cocycle $\phi$, one may
construct a twisted $k$-graph $C^*$-algebra $C^*_\phi(\Lambda)$. Up to isomorphism,
$C^*_\phi(\Lambda)$ only depends on the cohomology class of $\phi$.  Examples of this construction
include all noncommutative tori, and also the Heegaard-type quantum 3-spheres of
\cite{BaumHajacEtAl:K-th05}.

The purpose of this paper is to begin to analyse the structure of twisted $k$-graph $C^*$-algebras.
In particular, we provide a groupoid model for twisted $k$-graph $C^*$-algebras, and establish
versions of the standard uniqueness theorems. The path groupoid of a $k$-graph was the basis for
the description of $k$-graph $C^*$-algebras in \cite{KP2000}, and many key theorems about $k$-graph
$C^*$-algebras follow from this description and Renault's structure theory for groupoid
$C^*$-algebras \cite{Renault1980}. We therefore set out to show that each twisted $k$-graph
$C^*$-algebra is also isomorphic to the twisted groupoid $C^*$-algebra $C^*(\Gg_\Lambda, \sigma)$
associated to the path groupoid $\Gg_\Lambda$ and an appropriate continuous $\TT$-valued
$2$-cocycle on $\Gg_\Lambda$. It is not immediately clear how to manufacture a groupoid cocycle
from a $k$-graph cocycle. Part of the difficulty lies in that continuous groupoid cocycle
cohomology is based on the simplicial structure of groupoids while the $k$-graph cohomology of
\cite{kps3} is based on the cubical structure of $k$-graphs.

To solve this difficulty we introduce another cohomology theory $\Hcat*(\Lambda , A)$ for
$k$-graphs, defined by analogy with continuous groupoid cocycle cohomology using the simplicial
structure of the $k$-graph as a small category.  We call this the categorical cohomology of
$\Lambda$ (it is no doubt closely related to the standard notion of the cohomology of a small
category, see \cite{BW}), and refer to the theory developed in \cite{kps3} simply as the cohomology
of $\Lambda$ or, if we wish to emphasise the distinction between the two theories, as the cubical
cohomology of $\Lambda$. It is relatively straightforward to see (Remark~\ref{rmk:0-hom} and
Theorem~\ref{thm:cocycles}) that the cohomology groups $\Hcub0(\Lambda,A)$ and $\Hcub1(\Lambda,A)$
of \cite{kps3} are isomorphic to the corresponding categorical cohomology groups
$\Hcat0(\Lambda,A)$ and $\Hcat1(\Lambda,A)$.

Of most interest to us, because of its role in the definition of twisted $C^*$-algebras, is second
cohomology. We show in Theorem~\ref{thm:2-cocycle map} and Theorem~~\ref{thm:2-cohomology iso} that
there is a map between (cubical) $2$-cocycles and categorical $2$-cocycles on a $k$-graph $\Lambda$
that induces an isomorphism $\Hcub2(\Lambda, A) \cong \Hcat2(\Lambda, A)$. However, this result
requires substantially more argument than those discussed in the preceding paragraph. The proof
occupies the greater part of Section~\ref{sec:cohomology} and all of Section~\ref{sec:extensions}.
Our approach is inspired by the classification of central extensions of groups  by second
cohomology (see \cite[\S{IV.3}]{Brown}). We first construct by hand a map $\phi \mapsto c_\phi$
from cubical cocycles to categorical cocycles which determines a homomorphism $\psi : \Hcub2 (
\Lambda , A ) \to \Hcat2 ( \Lambda , A)$. We then define the notion of a central extension of a
$k$-graph by an abelian group, and show that each categorical $A$-valued $2$-cocycle $c$ on
$\Lambda$ determines a central extension $\Xx_c$ of $\Lambda$ by $A$. We show that isomorphism
classes of central extensions of $\Lambda$ by $A$ form a group $\Ext(\Lambda,A)$, and that the
assignment $c \mapsto \Xx_c$ determines an isomorphism $\Hcat{2}(\Lambda,A ) \cong \Ext(\Lambda,
A)$ (cf. \cite[Theorem~2.3]{BW} and \cite[Proposition~I.1.14]{Renault1980}). We show that for $c
\in \Zcat2 ( \Lambda , A )$ there is a section $\sigma : \Lambda \to \Xx_c$ which gives rise to a
cubical cocycle $\phi_c$ such that $[ c_{\phi_c} ] = [c]$ and $[ \phi_{c_\phi} ] = [ \phi ]$. This
shows that $\psi: \Hcub2 ( \Lambda , A ) \to \Hcat2 ( \Lambda , A)$ is an isomorphism.

It is, of course, natural to ask whether $\Hcub{n} ( \Lambda, A ) \cong \Hcat{n} ( \Lambda , A)$
for all $n$. We suspect this is so, but have not found a proof as yet, and the methods we use to
prove isomorphism of the first three cohomology groups do not seem likely to extend readily to a
general proof. In any case, we expect that the central extensions of $k$-graphs introduced here are
of interest in their own right. For example, we believe that extensions of $k$-graphs can be used
to adapt Elliott's argument \cite[proof of Theorem 2.2]{Elliott84} --- which shows that the
$K$-groups of a noncommutative torus are isomorphic to those of the corresponding classical torus
--- to show that the $K$-groups of a twisted $k$-graph $C^*$-algebra are identical to those of the
untwisted algebra whenever the twisting cocycle is obtained from exponentiation of a real-valued
cocycle.

In the second half of the paper we turn to the relationship between categorical cohomology and
twisted $C^*$-algebras of $k$-graphs. We define the twisted $C^*$-algebra $C^*(\Lambda, c)$
associated to a categorical $\TT$-valued $2$-cocycle $c$ on a row-finite $k$-graph $\Lambda$ with
no sources, and show that $C^*_\phi(\Lambda) \cong C^*(\Lambda, c_\phi)$ for each cubical
$\TT$-valued $2$-cocycle $\phi$. The advantage of the description of twisted $k$-graph
$C^*$-algebras in terms of categorical cocycles is that it closely mirrors the usual definition of
the $C^*$-algebra of a $k$-graph. This allows us to commence a study of the structure theory of
twisted $k$-graph $C^*$-algebras.  We prove that there is map $c \mapsto \sigma_c$ which induces a
homomorphism from the second categorical cohomology of a $k$-graph to the second continuous
cohomology of the associated path groupoid.  We then prove that for a categorical $\TT$-valued
$2$-cocycle $c$ on $\Lambda$, there is a homomorphism from the twisted $k$-graph $C^*$-algebra
associated to $c$ to Renault's twisted groupoid $C^*$-algebra $C^*(\Gg_\Lambda, \sigma_c)$; this
shows in particular, that all the generators of every twisted $k$-graph $C^*$-algebra are nonzero.
We then prove a version of an Huef and Raeburn's gauge-invariant uniqueness theorem for twisted
$k$-graph $C^*$-algebras, and use it to prove that $C^* ( \Lambda , c ) \cong C^* ( \Gg_\Lambda ,
\sigma_c)$.

We finish up in Section~\ref{sec:structure} by using the results of the previous sections to
establish some fundamental structure results. We use the realisation of each twisted $k$-graph
$C^*$-algebra as a twisted groupoid $C^*$-algebra, together with Renault's theory of groupoid
$C^*$-algebras \cite{Renault1980} to prove a version of the Cuntz-Krieger uniqueness theorem. We
also indicate how groupoid technology applies to describe twisted $C^*$-algebras of pullback and
cartesian-product $k$-graphs, and to show that every twisted $k$-graph $C^*$-algebra is nuclear and
belongs to the bootstrap class $\mathcal{N}$.

\subsection*{Acknowledgement.}
We thank the anonymous referee for a careful reading of the paper.
The first author thanks his coauthors for their hospitality.

\section{Preliminaries}

\subsection{Higher-rank graphs}
We adopt the conventions of \cite{kps2, pqr1} for $k$-graphs. Given a nonnegative integer $k$, a
\emph{$k$-graph} is a nonempty countable small category $\Lambda$ equipped with a functor $d
:\Lambda \to \NN^k$ satisfying the \emph{factorisation property}: for all $\lambda \in \Lambda$ and
$m,n \in \NN^k$ such that $d( \lambda )=m+n$ there exist unique $\mu ,\nu \in \Lambda$ such that
$d(\mu)=m$, $d(\nu)=n$, and $\lambda=\mu \nu$. When $d(\lambda )=n$ we say $\lambda$ has
\emph{degree} $n$. We will typically use $d$ to denote the degree functor in any $k$-graph in this
paper.

For $k \ge 1$, the standard generators of $\NN^k$ are denoted $e_1, \dots, e_k$, and for $n \in
\NN^k$ and $1 \le i \le k$ we write $n_i$ for the $i^{\rm th}$ coordinate of $n$. For $n = (n_1 ,
\ldots , n_k ) \in \NN^k$ let $| n | := \sum_{i=1}^k n_i$;  for $\lambda \in \Lambda$ we define $|
\lambda | := | d ( \lambda ) |$. For $m,n \in \NN^k$, we write $m \le n$ if $m_i \le n_i$ for all
$i \le k$, and we write $m \vee n$ for the coordinatewise maximum of $m$ and $n$.

For $n \in \NN^k$, we write $\Lambda^n$ for $d^{-1}(n)$. The \emph{vertices} of $\Lambda$ are the
elements of $\Lambda^0$. The factorisation property implies that $o \mapsto \id_o$ is a bijection
from the objects of $\Lambda$ to $\Lambda^0$. We will use this bijection to identify
$\Obj(\Lambda)$ with $\Lambda^0$ without further comment. The domain and codomain maps in the
category $\Lambda$ then become maps $s,r : \Lambda \to \Lambda^0$. More precisely, for $\alpha
\in\Lambda$, the \emph{source} $s(\alpha)$ is the identity morphism associated with the object
$\dom(\alpha)$ and similarly, $r(\alpha) = \id_{\cod(\alpha)}$. An \textit{edge} is a morphism $f$
with $d(f) = e_i$ for some $i \in \{1, \ldots , k\}$.

Let $\lambda$ be an element of a $k$-graph $\Lambda$ and suppose $m,n \in \NN^k$ satisfy $0 \le m
\le n \le d(\lambda)$. By the factorisation property there exist unique elements $\alpha, \beta,
\gamma \in \Lambda$ such that
\[
\lambda = \alpha\beta\gamma, \quad d(\alpha) = m,
\quad d(\beta) = n - m,\quad\text{and}\quad d(\gamma) = d(\lambda) - n.
\]
We define $\lambda(m,n) := \beta$. Observe that $\alpha = \lambda(0, m)$ and $\gamma = \lambda(n,
d(\lambda))$.

For $\alpha,\beta\in\Lambda$ and $E \subset \Lambda$, we write $\alpha E$ for $\{\alpha\lambda :
\lambda \in E, r(\lambda) = s(\alpha)\}$ and $E\beta$ for $\{\lambda\beta : \lambda \in E,
s(\lambda) = r(\beta)\}$. So for $u,v \in \Lambda^0$, we have $uE = E \cap r^{-1}(u)$ and $E v = E
\cap s^{-1}(v)$.

Recall from \cite{RSY2003} that for $\mu,\nu \in \Lambda$, the set $\mu\Lambda \cap \nu\Lambda \cap
\Lambda^{d(\mu) \vee d(\nu)}$ of minimal common extensions of $\mu$ and $\nu$ is denoted
$\MCE(\mu,\nu)$.

We allow $0$-graphs with the convention that $\NN^0 = \{0\}$. A $0$-graph consists only of identity
morphisms, and we regard it as a countable nonempty collection of isolated vertices.

If $E = (E^0, E^1, r, s)$ is a directed graph as in \cite{KPRR1997}, then its
path category is a $1$-graph, and conversely, every $1$-graph $\Lambda$ is the path category of the
directed graph with vertices $\Lambda^0$, edges $\Lambda^1$ and range and source maps inherited
from $\Lambda$.  In this paper we shall treat directed graphs and $1$-graphs interchangeably. That
is, if $E$ is a directed graph $(E^0, E^1, r, s)$, then we shall also use $E$ to denote its path
category regarded as a $1$-graph.

\subsection{Cohomology of \texorpdfstring{$k$}{k}-graphs}

We recall the (cubical) cohomology of a $k$-graph described in \cite{kps3}.
For $k \ge 0$ define
\[
\mathbf{1}_k :=
\begin{cases}
(1, \dots, 1) & \text{if  } k > 0, \\
0 & \text{if  } k = 0.
\end{cases}
\]

\noindent
Let $\Lambda$ be a $k$-graph. For $0 \le r \le k$ let
\[
Q_r ( \Lambda ) := \{ \lambda \in \Lambda : d ( \lambda ) \le \mathbf{1}_k , |\lambda| = r \}.
\]
For $r > k$ let $Q_r ( \Lambda ) := \emptyset$.

Fix  $0 < r \le k$. The set $Q_r(\Lambda)$ consists of the morphisms of $\Lambda$ which may be
expressed as the composition of a sequence of $r$ edges whose degrees are distinct generators of
$\NN^k$. The factorisation property implies that each element of $Q_r ( \Lambda )$ determines a
commuting diagram in $\Lambda$ shaped like an $r$-cube. For example if $\lambda \in Q_3 ( \Lambda
)$ with $d(\lambda) = e_i + e_j + e_l$ with $i < j < l$, then multiple applications of the
factorisation property yield factorisations
\[
\lambda = f_0 g_0 h_0 = f_0 h_1 g_1 = h_2 f_1 g_1 = h_2 g_2 f_2 = g_3 h_3 f_2 = g_3 f_3 h_0
\]
such that $d(f_n) = e_i$, $d(g_n) = e_j$ and $d(h_n) = e_l$ for all $n$. So $\lambda$ determines
the following commuting diagram in which edges of degree $e_i$ are blue and solid, edges of degree
$e_j$ are red and dashed and edges of degree $e_l$ are green and dotted:
\begin{equation} \label{eq:commutingcube}
\begin{array}{lcr}
&\begin{tikzpicture}[scale=2.5]
    \node[circle,inner sep=1.5pt,fill=black] (000) at (0,0,0) {};
    \node[circle,inner sep=1.5pt,fill=black] (001) at (0,0,1) {};
    \node[circle,inner sep=1.5pt,fill=black] (010) at (0,1,0) {};
    \node[circle,inner sep=1.5pt,fill=black] (011) at (0,1,1) {};
    \node[circle,inner sep=1.5pt,fill=black] (100) at (1,0,0) {};
    \node[circle,inner sep=1.5pt,fill=black] (101) at (1,0,1) {};
    \node[circle,inner sep=1.5pt,fill=black] (110) at (1,1,0) {};
    \node[circle,inner sep=1.5pt,fill=black] (111) at (1,1,1) {};
    \draw[thick, blue, -stealth, preaction={draw,line width=1.6pt,white,-stealth}] (100)--(000)
        node[below, pos=0.6, inner sep=0.5pt, black] {\small$f_1$};
    \draw[thick, blue, -stealth, preaction={draw,line width=1.6pt,white,-stealth}] (110)--(010)
        node[above, pos=0.5, inner sep=0.5pt, black] {\small$f_2$};
    \draw[thick, red, dashed, -stealth, preaction={draw,line width=1.6pt,white,-stealth}] (010)--(000)
        node[left, pos=0.6, inner sep=0.5pt, black] {\small$g_2$};
    \draw[thick, red, dashed, -stealth, preaction={draw,line width=1.6pt,white,-stealth}] (110)--(100)
        node[left, pos=0.5, inner sep=0.5pt, black] {\small$g_1$};
    \draw[thick, green!50!black, dotted, -stealth, preaction={draw,line width=1.6pt,white,-stealth}] (000)--(001)
        node[anchor=south east, pos=0.5, inner sep=0.5pt, black] {\small$h_2$};
    \draw[thick, green!50!black, dotted, -stealth, preaction={draw,line width=1.6pt,white,-stealth}] (100)--(101)
        node[anchor=south east, pos=0.5, inner sep=0.5pt, black] {\small$h_1$};
    \draw[thick, green!50!black, dotted, -stealth, preaction={draw,line width=1.6pt,white,-stealth}] (010)--(011)
        node[anchor=south east, pos=0.5, inner sep=0.5pt, black] {\small$h_3$};
    \draw[thick, green!50!black, dotted, -stealth, preaction={draw,line width=1.6pt,white,-stealth}] (110)--(111)
        node[anchor=south east, pos=0.5, inner sep=0.5pt, black] {\small$h_0$};
    \draw[thick, blue, -stealth, preaction={draw,line width=1.6pt,white,-stealth}] (101)--(001)
        node[below, pos=0.5, inner sep=0.5pt, black] {\small$f_0$};
    \draw[thick, blue, -stealth, preaction={draw,line width=1.6pt,white,-stealth}] (111)--(011)
        node[above, pos=0.4, inner sep=0.5pt, black] {\small$f_3$};
    \draw[thick, red, dashed, -stealth, preaction={draw,line width=1.6pt,white,-stealth}] (011)--(001)
        node[left, pos=0.5, inner sep=0.5pt, black] {\small$g_3$};
    \draw[thick, red, dashed, -stealth, preaction={draw,line width=1.6pt,white,-stealth}] (111)--(101)
        node[left, pos=0.4, inner sep=0.5pt, black] {\small$g_0$};
\end{tikzpicture}&
\end{array}
\end{equation}

Each $\lambda \in Q_r(\Lambda)$ determines $2r$ elements of $Q_{r-1}(\Lambda)$ which we regard as
faces of $\lambda$. Fix $\lambda \in Q_r(\Lambda)$ and express $d(\lambda) = e_{i_1} + \cdots +
e_{i_r}$ where $i_1 < \cdots < i_r$. For $1 \le j \le r$ define $F_j^0(\lambda)$ and
$F_j^1(\lambda)$ to be the unique elements of $\Lambda^{d(\lambda) - e_{i_j}}$ such that $\lambda =
\alpha F_j^1(\lambda) = F_j^0(\lambda)\beta$ for some $\alpha,\beta \in \Lambda^{e_{i_j}}$.
Equivalently,
\[
F_j^0(\lambda) = \lambda(0, d(\lambda) - e_{i_j}) \quad\text{ and }\quad
F_j^1(\lambda) = \lambda(e_{i_j}, d(\lambda)).
\]
In example \eqref{eq:commutingcube}, $F^1_1(\lambda) = g_0h_0 = h_1g_1$, $F^0_2(\lambda) = f_0 h_1
= h_2 f_1$ and so on.

For $r \in \NN$ let $C_r(\Lambda) = \ZZ Q_r(\Lambda)$. For $r \ge 1$, define $\partial_r :
C_r(\Lambda) \to C_{r-1} ( \Lambda )$ to be the unique homomorphism such that
\[
\partial_r( \lambda )
    = \sum_{i=1}^r \sum_{\ell=0}^1 (-1)^{i+\ell}  F_i^\ell ( \lambda ) \quad\text{ for all $\lambda \in Q_r(\Lambda)$.}
\]
We write $\partial_0$ for the zero homomorphism $C_0 ( \Lambda ) \to \{0\}$. By
\cite[Lemma~3.3]{kps3} $( C_* ( \Lambda ) , \partial_* )$ is a chain complex.

As in \cite{kps3}, for $r \in \NN$ we denote the group $H_r(\Lambda) =
\ker(\partial_r) / \image(\partial_{r+1})$ by $H_r(\Lambda)$. We call $H_r(\Lambda)$ the 
$r^\text{th}$ homology group
of $\Lambda$.

Recall that a morphism $\phi : \Lambda \to \Gamma$ of $k$-graphs is a functor $\phi : \Lambda \to
\Gamma$ such that $d_\Gamma(\phi(\lambda)) = d_\Lambda(\lambda)$ for all $\lambda \in \Lambda$. By \cite[Lemma 3.5]{kps3} the assignment $\Lambda \mapsto H_*( \Lambda )$ is a covariant functor
from the category of $k$-graphs with $k$-graph morphisms to the category of abelian groups with
homomorphisms.

\begin{ntn}
Let $\Lambda$ be a $k$-graph and let $A$ be an abelian group. For $r \in \NN$, we write
$\Ccub{r}(\Lambda, A)$ for the collection of all functions $f : Q_r(\Lambda) \to A$. We identify
$\Ccub{r}(\Lambda, A)$ with $\Hom(C_r(\Lambda), A)$ in the usual way. Define maps $\dcub{r} :
\Ccub{r}(\Lambda,A) \to \Ccub{r+1}(\Lambda,A)$ by
\[
(\dcub{r} f)(\lambda) := f(\partial_{r+1}(\lambda)) =
\sum^{r+1}_{i=1} \sum^1_{\ell=0} (-1)^{i+\ell} f(F_i^\ell(\lambda)).
\]
Then $( \Ccub{*}(\Lambda, A), \dcub{*})$ is a cochain complex.
\end{ntn}

As in \cite{kps3}, we define the \emph{cohomology} $\Hcub*(\Lambda, A)$ of the $k$-graph $\Lambda$
with coefficients in $A$ to be the cohomology of the complex $\Ccub*(\Lambda,A)$; that is
$\Hcub{r}(\Lambda,A) := \ker(\dcub{r})/\image(\dcub{r-1})$. For $r \ge 0$, we write
$\Zcub{r}(\Lambda,A) := \ker(\dcub{r})$ for the group of $r$-cocycles, and for $r > 0$, we write
$\Bcub{r}(\Lambda,A) = \image(\dcub{r-1})$ for the group of $r$-coboundaries. We define
$\Bcub0(\Lambda, A) := \{0\}$. For each $r$, $\Hcub{r}(\Lambda,A)$ is a bifunctor, which is
contravariant in $\Lambda$ and covariant in $A$.

\begin{rmk}
As mentioned in the introduction, in the next section we introduce a new cohomology theory, called
``categorical cohomology" for $k$-graphs. When we wish to emphasise the distinction between the
two, we will refer to the version discussed here as ``cubical cohomology".
\end{rmk}

\section{Categorical cohomology}\label{sec:cohomology}

Here we introduce a second notion of cohomology for $k$-graphs,  obtained from the simplicial
structure of the category $\Lambda$ in a manner analogous to Renault's cohomology for groupoids
(see \cite[Definition I.1.11]{Renault1980}), which he attributes to Westman (see~\cite{West}). We
also follow his use of normalised cochains.

\begin{ntn}
Let $\Lambda$ be a $k$-graph, and let $A$ be an abelian group. For each integer $r \ge 1$, let
$\Lambda^{*r} := \big\{(\lambda_1, \dots, \lambda_r) \in \prod^r_{i=1}\Lambda : s(\lambda_i) =
r(\lambda_{i+1})\text{ for each }i\big\}$ be the collection of composable $r$-tuples in $\Lambda$,
and let $\Lambda^{*0} := \Lambda^0$. For $r \ge 0$, a function $f : \Lambda^{*r} \to A$ is said to
be an \emph{$r$-cochain} if $f(\lambda_1, \dots, \lambda_r) = 0$ whenever $\lambda_i \in \Lambda^0$
for some $0 < i \le r$. Observe that when $r = 0$ the cochain condition is vacuous, so every function
$f : \Lambda^0 \to A$ is a $0$-cochain. Let $\Ccat{r}(\Lambda,A)$ be the set of all $r$-cochains,
regarded as a group under pointwise addition.
\end{ntn}

\begin{dfn}\label{dfn:coboundary maps}
Fix $r \ge 1$. For $f \in \Ccat{r}(\Lambda, A)$ define $\dcat{r}f : \Lambda^{*(r+1)} \to A$ by
\begin{equation}\label{eq:dcat def}\begin{split}
(\dcat{r} f)(\lambda_0, \dots, \lambda_r)
    &= f(\lambda_1, \dots, \lambda_r) \\
    &\textstyle\qquad{} + \sum^r_{i=1} (-1)^i
        f\big(\lambda_0,\dots,\lambda_{i-2},(\lambda_{i-1}\lambda_i),\lambda_{i+1},\dots,\lambda_r\big)\\
    &\qquad\qquad{} + (-1)^{r+1} f(\lambda_0, \dots, \lambda_{r-1}).
\end{split}\end{equation}
For $f \in \Ccat0(\Lambda, A)$, define $\dcat0 f : \Lambda^{*1} \to A$ by
\begin{equation}\label{eq:dcat0 def}
(\dcat0 f)(\lambda) := f(s(\lambda)) - f(r(\lambda)).
\end{equation}
\end{dfn}

\begin{rmk}
It is routine to check that each $\dcat{r}$ maps $\Ccat{r}(\Lambda, A)$ to
$\Ccat{r+1}(\Lambda, A)$.
\end{rmk}

We sometimes emphasise the condition that $f(\lambda_1, \dots, \lambda_r) = 0$ whenever $\lambda_i
\in \Lambda^0$ for some $i$ by referring to such cochains as \emph{normalised} cochains. However,
since we will not consider any other sort of cochain in this paper, we usually eschew the
adjective.

\begin{lem}
The sequence
\[
0 \to \Ccat0(\Lambda,A)
    \stackrel{\dcat0}{\longrightarrow} \Ccat1(\Lambda,A)
    \stackrel{\dcat1}{\longrightarrow} \Ccat2(\Lambda,A)
    \stackrel{\dcat2}{\longrightarrow} \dots
\]
is a cochain complex.
\end{lem}
\begin{proof}
For $f \in \Ccat0(\Lambda,A)$ and $(\lambda_1,\lambda_2) \in \Lambda^{*2}$, we have
\begin{align*}
    (\dcat1\circ\dcat0 f)&(\lambda_1,\lambda_2)\\
        &= (\dcat0 f)(\lambda_1) - (\dcat0 f)(\lambda_1\lambda_2) + (\dcat0 f)(\lambda_2)\\
        &= f(s(\lambda_1)) - f(r(\lambda_1)) - (f(s(\lambda_1)) - f(r(\lambda_2))) + f(s(\lambda_2)) - f(r(\lambda_2)) \\
        &= 0,
\end{align*}
so $\dcat1 \circ \dcat0 = 0$.

To see that $\dcat{i+1}\circ\dcat{i} = 0$ for $i \ge 1$, we calculate:
\begin{align}
(\dcat{i+1}\circ\dcat{i} f)(\lambda_0, \dots, \lambda_{i+1})
    &= (\dcat{i} f)(\lambda_1, \dots, \lambda_{i+1} ) \label{eq:firstterm}\\
    &\qquad{} + \sum^{i+1}_{j=1} (-1)^j(\dcat{i} f)
        \big(\lambda_0, \dots, (\lambda_{j-1}\lambda_j), \dots, \lambda_{i+1}\big) \label{eq:midsum}\\
    &\qquad\qquad{} + (-1)^{i+2} (\dcat{i} f)(\lambda_0, \dots, \lambda_{i}). \label{eq:lastterm}
\end{align}
We must show that the right-hand side is equal to zero. Expand each term using~\eqref{eq:dcat def}.
For each $j$, the $j$\textsuperscript{th} term in the expansion of~\eqref{eq:firstterm} cancels the
first term in the expansion of the $j$\textsuperscript{th} summand of~\eqref{eq:midsum}. Likewise,
the $j$\textsuperscript{th} term in the expansion of~\eqref{eq:lastterm} cancels with the last term
in the expansion of the $j$\textsuperscript{th} summand of~\eqref{eq:midsum}. Finally, for $2 \le j
\le i$, the $i$\textsuperscript{th} term in the expansion of the $j$\textsuperscript{th} summand
of~\eqref{eq:midsum} cancels with the $j$\textsuperscript{th} term in the expansion of the
$(i+1)$\textsuperscript{st} summand.
\end{proof}

\begin{dfn}\label{dfn:categorical-cohomology}
The \emph{categorical cohomology} of $\Lambda$ with coefficients in $A$ is the cohomology
$\Hcat*(\Lambda,A)$ of the cochain complex described above. That is,
\[
\Hcat{r}(\Lambda,A) := \ker(\dcat{r})/\image(\dcat{r-1})\quad\text{ for each $r$.}
\]
We write $\Bcat{r}(\Lambda,A)$ for the group $\image(\dcat{r-1})$ of $r$-coboundaries, and
$\Zcat{r}(\Lambda,A)$ for the group $\ker(\dcat{r})$ of $r$-cocycles.
\end{dfn}

\begin{rmk}
For each $r$, $\Hcat{r}(\Lambda,A)$ is a bifunctor which is covariant in $A$ and contravariant in
$\Lambda$.
\end{rmk}

\begin{rmk}\label{rmk:groupoid cohomology}
Definitions \ref{dfn:coboundary maps}~and~\ref{dfn:categorical-cohomology} make sense for an
arbitrary small category $\Lambda$. If the category also carries a topology compatible with the
structure maps, and $A$ is a locally compact abelian group, it is natural to require $A$-valued
$n$-cochains on $\Lambda$ to be continuous. In this paper, we distinguish this \emph{continuous
cocycle cohomology} from its discrete cousin by denoting the cochain groups $\Cgpd*(\Lambda, A)$,
the coboundary groups $\Bgpd*(\Lambda, A)$, the cocycle groups $\Zgpd*(\Lambda, A)$ and the
cohomology groups $\Hgpd*(\Lambda, A)$. If $\Lambda$ is a topological groupoid $\Gg$ in the sense
of Renault, then we have simply replicated Renault's continuous cocycle cohomology of $\Gg$
introduced in \cite{Renault1980}.
\end{rmk}

A function from a $k$-graph $\Lambda$ into a group $G$ is called a \emph{functor} if it preserves
products. Such functors have sometimes been referred to informally as
cocycles; the following lemma justifies this informal usage.

\begin{lem}\label{lem:what r cocycles}
Let $(\Lambda,d)$ be a $k$-graph, and let $A$ be an abelian
group. Then a cochain $f_0 \in \Ccat{0} ( \Lambda , A )$ is a categorical
$0$-cocycle if and only if it is constant on connected
components; a cochain $f_1  \in \Ccat{1}(\Lambda,A)$ is a categorical
$1$-cocycle if and only if it is a functor; and a cochain
$f_2 \in \Ccat{2}(\Lambda,A)$ is a categorical $2$-cocycle if and
only if it satisfies the cocycle identity
\begin{equation} \label{eq:cocyleid}
f_2(\lambda_1,\lambda_2) + f_2(\lambda_1\lambda_2, \lambda_3) = f_2(\lambda_2, \lambda_3) + f_2(\lambda_1, \lambda_2\lambda_3)
\end{equation}
for all $(\lambda_1,\lambda_2,\lambda_3) \in \Lambda^{*3}$.
\end{lem}
\begin{proof}
For the first statement, note that $f_0$ is a 0-cocycle
if and only if $(\dcat0f_0)(\lambda) = 0$ for all $\lambda$,
which occurs if and only if $f_0(s(\lambda)) = f_0(r(\lambda))$
for all $\lambda$; that is, if and only if $f_0$ is constant on
connected components.

A $1$-cochain $f_1$ is a $1$-cocycle if and only if $(\dcat1f_1)(\lambda_1,\lambda_2) = 0$ for all
$(\lambda_1, \lambda_2) \in \Lambda^{*2}$; that is, if and only if
\[
f_1(\lambda_1) - f_1(\lambda_1\lambda_2) + f_1(\lambda_2) = 0\text{ for all $( \lambda_1,\lambda_2 ) \in \Lambda^{*2}$,}
\]
and this in turn is equivalent to the assertion that $f_1$ is a
functor.

Fix $f_2 \in \Ccat2(\Lambda, A)$. Then $f_2 \in \Zcat{2} ( \Lambda , A )$ if and only if for all
$(\lambda_1 , \lambda_2 , \lambda_3) \in \Lambda^{*3}$,
\[
0 = (\dcat2 f_2)(\lambda_1,\lambda_2,\lambda_3)
    = f_2(\lambda_2,\lambda_3) - f_2(\lambda_1\lambda_2, \lambda_3)
        + f_2(\lambda_1, \lambda_2\lambda_3) - f(\lambda_2,\lambda_3).
\]
Hence $f_2$ is a $2$-cocycle if and only if it satisfies \eqref{eq:cocyleid}.
\end{proof}

We now turn to the relationship between the cubical and the categorical cohomology of a $k$-graph
$\Lambda$. We will ultimately prove that $\Hcub{i} ( \Lambda,  A ) \cong \Hcat{i} ( \Lambda , A )$
for $i \le 2$, but sorting this out will take the remainder of this section and all of the next.

\begin{rmk} \label{rmk:0-hom}
By definition of the coboundary maps on cohomology from \cite{kps3}, an $A$-valued $0$-cocycle on a
$k$-graph $\Lambda$ is a function $c : \Lambda^0 \to A$ which is invariant for the equivalence
relation $\sim_{\mathrm{cub}}$ on vertices generated by $r(e) \sim_{\mathrm{cub}} s(e)$ for each
edge $e$. As in Lemma~\ref{lem:what r cocycles} an $A$-valued categorical $0$-cocycle on $\Lambda$
is a function $f_0 : \Lambda^0 \to A$ which is invariant for the equivalence relation
$\sim_{\mathrm{cat}}$ on vertices generated by $r(\lambda) \sim_{\mathrm{cat}} s(\lambda)$ for all
$\lambda \in \Lambda$. Since every path in $\Lambda$ can be factorised into edges,
$\sim_{\mathrm{cub}}$ and $\sim_{\mathrm{cat}}$ are identical. Hence
\[
\Hcub0(\Lambda,A) = \Hcat0(\Lambda,A) = \{f : \Lambda^0 \to A \mid f\text{ is constant on connected components}\}.
\]
\end{rmk}

We prove next that restriction of functions gives isomorphisms $\Zcat1(\Lambda, A)
\cong \Zcub1(\Lambda, A)$, $\Bcat1(\Lambda, A) \cong \Bcub1(\Lambda,A)$ and hence $\Hcat1(\Lambda,
A) \cong \Hcub1(\Lambda, A)$.

\begin{thm}\label{thm:cocycles}
Let $\Lambda$ be a $k$-graph and let $f \in C^1(\Lambda,A)$. If $f \in \Zcub1(\Lambda, A)$ then
there exists a unique element $\tilde{f} \in \Zcat{1} ( \Lambda , A )$ such that $\tilde{f}|_{Q_1 (
\Lambda )} = f$. Conversely, if $g \in \Zcat{1} ( \Lambda , A)$, then $g|_{Q_1 ( \Lambda )} \in
\Zcub{1} ( \Lambda , A )$. Finally, $f \in \Bcub1(\Lambda,A)$ if and only if $\tilde{f} \in
\Bcat1(\Lambda,A)$, and the map $f \mapsto \tilde{f}$ induces an isomorphism $\Hcub1(\Lambda, A)
\cong \Hcat1(\Lambda, A)$.
\end{thm}
\begin{proof}
Suppose first that $g \in \Zcat{1} (\Lambda , A)$ and let $g_0 = g|_{Q_1 ( \Lambda )}$. Then for
any $\lambda \in Q_2(\Lambda)$, we have
\begin{align*}
\dcub1(g_0)(\lambda)
    &= g_0 (F_1^1(\lambda)) - g_0 (F_1^0(\lambda)) - g_0 (F_2^1(\lambda)) + g_0 (F_2^0(\lambda)) \\
    &= \big(g_0(F_2^0(\lambda)) + g_0(F_1^1(\lambda))\big)
        - \big(g_0(F_1^0(\lambda)) + g_0(F_2^1(\lambda))\big).
\end{align*}
Since $F_2^0(\lambda)F_1^1(\lambda) = \lambda =
F_1^0(\lambda)F_2^1(\lambda)$, that $g$ is a functor
implies that $\dcub1(g_0) = 0$ so $g_0 \in \Zcub{1} ( \Lambda , A )$.

Now suppose that $f \in \Zcub1(\Lambda, A)$. We claim that there is a well-defined functor
$\tilde{f} : \Lambda \to A$ such that for any path $\lambda \in \Lambda$ and any factorisation
$\lambda = \lambda_1 \cdots \lambda_{|\lambda |}$ with each $\lambda_i \in Q_1(\Lambda)$,
\begin{equation}\label{eq:tildef formula}
\tilde{f}(\lambda) = \sum_{i=1}^{|\lambda|} f(\lambda_i).
\end{equation}
Given a path $\lambda \in \Lambda$, an edge-factorisation of $\lambda$ is a decomposition $\lambda
= \lambda_1 \cdots \lambda_{|\lambda|} \in \Lambda$ with each $\lambda_i \in Q_1(\Lambda)$. We say
that
\[
\lambda_1 \cdots \lambda_i\lambda_{i+1} \cdots \lambda_{| \lambda |}
    \to \lambda_1 \cdots \lambda'_i\lambda'_{i+1} \cdots \lambda_{| \lambda |}
\]
is an \emph{allowable transition of edge-factorisations of $\lambda$} if $d(\lambda_i) =
d(\lambda'_{i+1}) = e_j$ and $d(\lambda_{i+1}) = d(\lambda'_i) = e_l$ for some $1 \le l < j \le k$,
and $\lambda_i\lambda_{i+1} = \lambda'_i\lambda'_{i+1}$. Any edge-factorisation of a fixed path
$\lambda \in \Lambda$ can be transformed into any other by a sequence of such allowable transitions
and their inverses. Since $f$ is a cocycle, the formula~\eqref{eq:tildef formula} is invariant
under allowable transitions and so determines a well-defined function $\tilde{f}$ from $\Lambda$ to
$A$, which is a functor which extends $f$ by definition. Moreover, any functor $\tilde{f} : \Lambda
\to A$ which extends $f$ must satisfy~\eqref{eq:tildef formula}, and so must be equal to
$\tilde{f}$.

A function $f : Q_1(\Lambda) \to A$ belongs to $\Bcub1(\Lambda,A)$ if and only if there is a map $b
: \Lambda^0 \to A$ such that $f(\lambda) = b(s(\lambda)) - b(r(\lambda))$ for all $\lambda \in
Q_1(\Lambda)$. It follows that $f \in \Bcub1(\Lambda,A)$ if and only if there is a function $b :
\Lambda^0 \to A$ such that the unique extension $\tilde{f} : \Lambda \to A$ of the preceding paragraph satisfies $\tilde{f}(\lambda) = b(s(\lambda)) - b(r(\lambda)) =
(\dcat0b)(\lambda)$ for all $\lambda \in \Lambda$; that is, if and only if $\tilde{f} \in
\Bcat1(\Lambda,A)$.
\end{proof}

We now wish to show that each cubical $2$-cocycle determines a categorical $2$-cocycle, and deduce
that there is a homomorphism from $\Zcub2(\Lambda,A)$ to $\Zcat2(\Lambda,A)$ which descends to a
homomorphism $\psi : \Hcub2(\Lambda,A) \to \Hcat2(\Lambda,A)$. The set-up and proof of this result
will occupy the remainder of this section. In the next section, we will introduce central
extensions of $k$-graphs by abelian groups to show that $\psi$ is an isomorphism.

So for the remainder of the section, we fix a $k$-graph $\Lambda$ and an abelian group $A$. By
definition of $\dcub2$, for $\phi \in \Zcub2(\Lambda,A)$ and any $\lambda \in Q_3 ( \Lambda )$,
\begin{equation}\label{eq:cubical 2-cocycle}
\phi(F^0_3(\lambda)) + \phi(F^1_2(\lambda)) + \phi(F^0_1(\lambda))
    = \phi(F^1_1(\lambda)) + \phi(F^0_2(\lambda)) + \phi(F^1_3(\lambda)).
\end{equation}

To commence our construction of the homomorphism $\psi : \Hcub2(\Lambda,A) \to \Hcat2(\Lambda,A)$
we recall the notion of the skeleton, viewed as a $k$-coloured graph, of a $k$-graph $\Lambda$.

\begin{ntn}\label{ntn:coloured graph}
A $\emph{$k$-coloured graph}$ is a directed graph $E$ endowed with a map $C : E^1 \to \{1, \dots,
k\}$ which we regard as assigning a colour to each edge. Using our convention that the
path-category of $E$, regarded as a $1$-graph, is still denoted $E$, we extend $C$ to a functor,
also denoted $C$, from $E$ to the free semigroup $\FF^+_k = \langle 1, 2, \dots, k\rangle$ on $k$
generators.

Given a $k$-graph $\Lambda$ we write  $E_\Lambda$ for the $k$-coloured graph such that $E_\Lambda^0 =
\Lambda^0$, $E_\Lambda^1 = Q_1(\Lambda) = \bigcup^k_{i=1} \Lambda^{e_i}$, the maps $r,s :
E_\Lambda^1 \to E_\Lambda^0$ are inherited from $\Lambda$, and $d(\alpha) = e_{C(\alpha)}$ for all
$\alpha \in Q_1(\Lambda)$.

There is a surjective functor $\pi : E_\Lambda \to \Lambda$ such that $\pi(\alpha) = \alpha$ for
all $\alpha \in Q_1(\Lambda)$. Let $q : \FF^+_k \to \NN^k$ be the semigroup homomorphism such that
$q(i) = e_i$ for $1 \le i \le k$. Then $q \circ C = d \circ \pi$.

We define a preferred section for $\pi$ as follows. Given $\lambda \in \Lambda^n$, we denote by
$\overline{\lambda} \in E_\Lambda$ the unique path $\overline{\lambda}_1 \dots
\overline{\lambda}_{|n|}$ in $E_\Lambda$ such that $\pi(\overline{\lambda}) = \lambda$ and
$C(\overline{\lambda}_i) \le C(\overline{\lambda}_{i+1})$ for all $i$\footnote{The ordering on the
generators of $\FF^+_k$ is just the usual ordering of $\{1, \dots, k\}$.}.

An \emph{allowable transition} in $E_\Lambda$ is an ordered pair $(u,w) \in E_\Lambda \times
E_\Lambda$ such that $\pi(u) = \pi(w)$ and there is an $i$ such that $u_j = w_j$ for $j \not\in
\{i,i+1\}$ and $C(w_{i+1})= C(u_i) < C(u_{i+1}) = C(w_i)$. The factorisation property forces
$u_iu_{i+1} = w_iw_{i+1}$ because  $\pi(u) = \pi(w)$ in $\Lambda$. Informally, if $(u,w)$ is an
allowable transition, then the edges $w_i$ and $w_{i+1}$ are in reverse colour-order, and $u$ is
the path obtained by switching them around using the factorisation property in $\Lambda$. If $(u,
w)$ is an allowable transition we define $p(u,w) := \min\{j : u_j \not= w_j\}$.
\end{ntn}

\begin{dfn}\label{dfn:transition graph}
Given a $k$-graph $\Lambda$, the \emph{transition graph} of $\Lambda$ is the $1$-graph $F_\Lambda$
such that $F_\Lambda^0 := E_\Lambda$, $F_\Lambda^1 := \{(u,w) : (u,w)\text{ is an allowable
transition in }E_\Lambda\}$, and $r,s : F_\Lambda^1 \to F_\Lambda^0$ are defined by $r(u,w) := u$
and $s(u,w) := w$.
\end{dfn}

Let $\Lambda$ be a $k$-graph. Given $u \in F_\Lambda^0$, since $u$ is a path in $E_\Lambda$, we
will frequently write $\ell(u)$ for the number of edges in $u$ regarded as a path in $E_\Lambda$.
The connected components of the transition graph $F_\Lambda$ are in one-to-one correspondence with
elements of $\Lambda$. Specifically, given a path $\lambda \in \Lambda$, the set $\pi^{-1}(\lambda)
\subset F_\Lambda^0$ is the collection of vertices in a connected component $F_\lambda$ of
$F_\Lambda$. We have $\ell(u) = |\lambda|$ for all $u \in F^0_\lambda$.

Each $F_\lambda$ (and hence $F_\Lambda$) contains no directed cycles. Moreover, for each $\lambda
\in \Lambda$, the preferred factorisation $\overline{\lambda}$ is the unique terminal vertex of
$F_\lambda$.

Define $h: F_\Lambda^0 \to \NN$ by
\[
h(u) = \sum_{i=1}^{\ell(u)} |\{j < i : C(u_j) > C(u_i)\}|.
\]
An induction shows that $h(u)$ measures the distance from $u$ to the terminal vertex in its
connected component: that is, we have $h(u) = |\alpha|$ for any path $\alpha \in \overline{\pi(u)}
F_\Lambda u$. In particular, for $\lambda \in \Lambda$, we have $h(\overline{\lambda}) = 0$, and if
$u,w \in F^0_\Lambda$, and $\alpha \in u F_\Lambda w$, then $|\alpha| = h(w) - h(u)$.

\begin{ntn}\label{ntn:phitilde}
For $\phi \in \Zcub2(\Lambda, A)$, define $\tilde{\phi} : F^1_\Lambda \to A$ as follows: if $(u,w) \in F^1_\Lambda$ and $p(u,w) = i$, then $\tilde{\phi}(u,w) = \phi(\pi(u_i u_{i + 1}))$. That is, $\tilde{\phi}(u,w)$ is the value of $\phi$ on the element of $Q_2(\Lambda)$ which is flipped when passing from $w$ to $u$. We extend $\tilde\phi$ to a functor from $F_\Lambda$ to $A$ by $\tilde\phi(\alpha) := \sum^{|\alpha|}_{i=1} \tilde\phi(\alpha_i)$.
\end{ntn}

\begin{lem}\label{lem:meet-up}
Let $\tau, \rho \in F^1_\Lambda$ with $s(\tau) = s(\rho)$. Then there exist $\mu \in F_\Lambda
r(\tau)$ and $\nu \in F_\Lambda r(\rho)$ such that $r(\mu) = r(\nu)$ and $\tilde\phi(\mu\tau) =
\tilde\phi(\nu\rho)$.
\end{lem}
\begin{proof}
Let $w := s(\tau)$, and let $n = \ell(w)$ so that $w = w_1 \cdots w_n$ with each $w_i \in
E^1_\Lambda$. We assume without loss of generality that $p(\tau) \le p(\rho)$. We consider three
cases.

Case~1: $p(\tau) = p(\rho)$. Then $\tau = \rho$, and $\mu = \nu
= r(\tau)$ trivially have the desired properties.

Case~2: $p(\rho) \ge p(\tau) + 2$. Then
\begin{align*}
r(\tau) &= w_1 \cdots w_{p(\tau)-1} ef w_{p(\tau)+2}\cdots w_{p(\rho)-1} w_{p(\rho)} w_{p(\rho)+1} w_{p(\rho)+2} \cdots w_n\quad\text{ and}\\
r(\rho) &= w_1 \cdots w_{p(\tau)-1} w_{p(\tau)} w_{p(\tau)+1} w_{p(\tau)+2}\cdots w_{p(\rho)-1} gh w_{p(\rho)+2} \cdots w_n
\end{align*}
where $w_{p(\tau)} w_{p(\tau)+1} = ef$ and $w_{p(\rho)} w_{p(\rho)+1} = gh$ are $2$-cubes of
$\Lambda$. Let
\[
    v := w_1 \cdots w_{p(\tau)-1} ef w_{p(\tau)+2}\cdots w_{p(\rho)-1} gh w_{p(\rho)+2} \cdots w_n \in F^0_\Lambda,
\]
and let $\mu := (v, r(\tau))$ and $\nu := (v,r(\rho))$. Then
$\mu \in F^1_\Lambda r(\tau)$ and $\nu \in F^1_\Lambda r(\rho)$
with $r(\mu) = r(\nu) = v$, and $\tilde\phi(\mu\tau) = \phi(gh)
+ \phi(ef) = \tilde\phi(\nu\rho)$ as required.

Case~3: $p(\rho) = p(\tau) + 1$. Then $\lambda := \pi(w_{p(\tau)}w_{p(\tau)+1}w_{p(\tau)+2})$
belongs to $Q_3(\Lambda)$. Hence we may factorise $\lambda$ as in~\eqref{eq:commutingcube}. That
is,
\[
\lambda = f_0 g_0 h_0 = f_0 h_1 g_1 = h_2 f_1 g_1 = h_2 g_2 f_2 = g_3 h_3 f_2 = g_3 f_3 h_0
\]
where $C(f_i) = C(f_0) < C(g_j) = C(g_0) < C(h_l) = C(h_0)$ for all $i,j,l$. Since $\rho$ and
$\tau$ are allowable transitions, we have $w_{p(\tau)} = h_2$, $w_{p(\tau) + 1} = g_2$ and
$w_{p(\tau) + 2} = f_2$. We have
\begin{align*}
r(\tau) &= w_1\cdots w_{p(\tau)-1}g_3 h_3 f_2 w_{p(\tau)+3}\cdots w_n
    \quad\text{and}\\
r(\rho) &= w_1\cdots w_{p(\tau)-1}h_2 f_1 g_1 w_{p(\tau)+3}\cdots w_n.
\end{align*}
Define $\mu, \nu \in F_\Lambda^2$ by
\begin{align*}
\mu &:= (w_1\cdots f_0g_0h_0 \cdots w_n,\; w_1\cdots g_3f_3h_0 \cdots w_n)\\
    &\qquad\qquad (w_1\cdots g_3f_3h_0\cdots w_n,\; w_1\cdots g_3h_3f_2\cdots w_n),\\
\intertext{and}
\nu &:= (w_1\cdots f_0g_0h_0 \cdots w_n,\; w_1\cdots f_0h_1g_1 \cdots w_n))\\
    &\qquad\qquad (w_1\cdots f_0h_1g_1\cdots w_n,\; w_1\cdots h_2f_1g_1 \cdots w_n).
\end{align*}
Then $\mu \in F_\Lambda r(\tau)$ and $\nu \in F_\Lambda r(\beta)$ with $r(\mu) = r(\nu)$. Moreover,
\begin{align*}
\tilde\phi(\mu\tau) &= \phi(f_0g_0) + \phi(f_3 h_0) + \phi(g_3 h_3) = \phi(F^0_3(\lambda)) + \phi(F^1_2(\lambda)) + \phi(F^0_1(\lambda))\quad\text{ and}\\
\tilde\phi(\nu\rho) &= \phi(g_0h_0) + \phi(f_0 h_1) + \phi(f_1 g_1) = \phi(F^1_1(\lambda)) + \phi(F^0_2(\lambda)) + \phi(F^1_3(\lambda)),
\end{align*}
so $\tilde\phi(\mu\tau) = \tilde\phi(\nu\rho)$
by~\eqref{eq:cubical 2-cocycle}.
\end{proof}

\begin{lem}\label{lem:coboundary on transition graph}
Let $\Lambda$ be a $k$-graph. There is a well-defined function $S_\phi : F_\Lambda^0 \to A$ defined
by $S_\phi(w) = \tilde\phi(\alpha)$ for any $\alpha \in \overline{\pi(w)} F_\Lambda w$.
\end{lem}
\begin{proof}
Since each connected component of $F_\Lambda$ has a unique sink and is finite, it suffices to fix
$\lambda \in \Lambda$, a vertex $w \in F_\lambda^0$ and two paths $\alpha,\beta \in
\overline{\lambda} F_\lambda w$ and show that $\tilde\phi(\alpha) = \tilde\phi(\beta)$. We proceed
by induction on $h(w)$. If $h(w) = 0$ then $w = \overline{\lambda}$ and the result is trivial.

Now fix $n \in \NN$. Suppose as an inductive hypothesis that $\tilde\phi(\alpha) =
\tilde\phi(\beta)$ whenever $\alpha, \beta \in \overline{\lambda} F_\Lambda w$ with $h(w) \le n$.
Fix $w \in F_\Lambda$ with $h(w) = n+1$ and $\alpha,\beta \in \overline{\lambda} F_\Lambda w$.

We have $|\alpha| = |\beta| = n+1$. Write $\alpha = \alpha'\alpha_{n+1}$ and $\beta =
\beta'\beta_{n+1}$ where $\alpha_{n+1}, \beta_{n+1} \in F^1_\Lambda$. By Lemma~\ref{lem:meet-up}
applied to $\tau := \alpha_{n+1}$ and $\rho := \beta_{n+1}$ in $F^1_\Lambda w$, there exist $\mu
\in F_\Lambda r(\alpha_{n+1})$ and $\nu \in F_\Lambda r(\beta_{n+1})$ such that $r(\mu) = r(\nu)$
and $\tilde\phi(\mu\alpha_{n+1}) = \tilde\phi(\nu\beta_{n+1})$. Since $\overline{\lambda}$ is the
unique sink in $F_\lambda$, and since $F^0_\lambda$ is finite, there is a path $\eta$ from $r(\mu)$
to $\overline{\lambda}$. The situation is summarised in the following diagram.
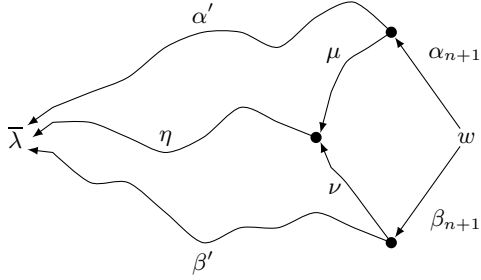
\[\begin{tikzpicture}[scale=2]
    \node[circle, inner sep=0pt] (w) at (3,0) {\small$w$};
    \node[circle, inner sep=1.5pt, fill=black] (ra) at (2.5,0.7) {};
    \node[circle, inner sep=1.5pt, fill=black] (rb) at (2.5,-0.7) {};
    \node[circle, inner sep=1.5pt, fill=black] (v) at (2, 0) {};
    \node[circle, inner sep=1pt] (lambdabar) at (0,0) {\small$\overline{\lambda}$};
    \draw[-latex] (w)--(ra) node[pos=0.5,anchor=south west, circle] {\small$\alpha_{n+1}$};
    \draw[-latex] plot[smooth] coordinates {(ra) (2.25, 0.9) (2, 0.8) (1.75,0.6) (1.5,0.7)
            (1.25,0.7) (1, 0.6) (0.75, 0.4) (0.5, 0.3) (0.25, 0.2)}
            -- (lambdabar.north east);
    \node[circle, inner sep=1pt, anchor=south] at (1.25, 0.7) {\small$\alpha'$};
    \draw[-latex] plot[smooth] coordinates {(ra) (2.35, 0.6) (2.2, 0.5) (2.1, 0.3)} -- (v.north east);
    \node[circle, inner sep=1pt, anchor=330] at (2.2,0.5) {\small$\mu$};
    \draw[-latex] (w)--(rb) node[pos=0.5,anchor=north west, circle] {\small$\beta_{n+1}$};
    \draw[-latex] plot[smooth] coordinates {(rb) (2.25, -0.6) (2, -0.5) (1.75, -0.6) (1.5, -0.6)
            (1.25, -0.7) (1, -0.5) (0.75, -0.3) (0.5, -0.3) (0.25, -0.1)} -- (lambdabar.south east);
    \node[circle, inner sep=1pt, anchor=north] at (1.25, -0.7) {\small$\beta'$};
    \draw[-latex] plot[smooth] coordinates {(rb) (2.35, -0.5) (2.2, -0.3) (2.1, -0.2)} -- (v.south east);
    \node[circle, inner sep=1pt, anchor=30] at (2.2,-0.3) {\small$\nu$};
    \draw[-latex] plot[smooth] coordinates{(v) (1.75, 0.1) (1.5, 0.2) (1.25,0) (1, -0.1) (0.75,0)
            (0.5,0.1) (0.25, 0.1)} -- (lambdabar.east);
    \node[circle, inner sep=1pt, anchor=south] at (1, -0.1) {\small$\eta$};
\end{tikzpicture}\]
Since $h(r(\alpha_{n+1})) = h(r(\beta_{n+1})) = n$, the inductive hypothesis gives
$\tilde\phi(\eta\mu) = \tilde\phi(\alpha')$ and $\tilde\phi(\eta\nu) = \tilde\phi(\beta')$. We then
have
\[
\tilde\phi(\alpha)
    = \tilde\phi(\alpha') + \tilde\phi(\alpha_{n+1})
    = \tilde\phi(\eta\mu) + \tilde\phi(\alpha_{n+1})
    = \tilde\phi(\eta) + \tilde\phi(\mu\alpha_{n+1}).
\]
A symmetric calculation shows that
\[
\tilde\phi(\beta) = \tilde\phi(\eta) + \tilde\phi(\nu\beta_{n+1}).
\]
Since $\tilde\phi(\mu\alpha_{n+1}) = \tilde\phi(\nu\beta_{n+1})$ by choice of $\mu$ and $\nu$, it
follows that $\tilde\phi(\alpha) = \tilde\phi(\beta)$.
\end{proof}

Lemma~\ref{lem:coboundary on transition graph} implies that $\tilde\phi$ is a $1$-coboundary of
$F_\Lambda$: specifically, $\tilde\phi = \dcat0(S_\phi)$. We call $S_\phi$ the \emph{shuffle
function} associated with $\phi$. We regard it as measuring the ``cost'' of shuffling the edges in
a coloured path into preferred order.

\begin{thm}\label{thm:2-cocycle map}
Let $(\Lambda,d)$ be a $k$-graph. For $\phi \in
\Zcub2(\Lambda,A)$, define
\[
c_\phi : \Lambda^{*2} \to A\quad\text{ by }
    c_\phi(\mu, \nu) := S_\phi(\overline{\mu}\, \overline{\nu}).
\]
Then $c_\phi \in \Zcat2(\Lambda, A)$ and $\phi \mapsto c_\phi$ is a homomorphism from
$\Zcub2(\Lambda,A)$ to $\Zcat2(\Lambda,A)$ satisfying
\begin{enumerate}
\item\label{it:on squares} $c_\phi(f,g) = \phi(fg)$ if
    $d(f) = e_i$ and $d(g) = e_j$ with $i > j$, and
    $c_\phi(f,g) = 0$ if $d(f) = e_i$ and $d(g) = e_j$ with
    $i \le j$; and
\item\label{it:already preferred} $c_\phi(\mu,\nu) = 0$
    whenever $\overline{\mu}\,\overline{\nu} =
    \overline{\mu\nu}$; in particular
    $c_\phi(r(\lambda),\lambda) =
    c_\phi(\lambda,s(\lambda)) = 0$
for all $\lambda \in \Lambda$.
\end{enumerate}
Moreover, if $\phi \in \Bcub2(\Lambda,A)$, then $c_\phi \in \Bcat2(\Lambda,A)$; hence, $[\phi]
\mapsto [c_\phi]$ defines a homomorphism $\psi: \Hcub2(\Lambda,A) \to \Hcat2(\Lambda,A)$.
\end{thm}

To prove the theorem, we need a further technical result.

\begin{ntn}
For $\phi \in \Zcub2(\Lambda, A)$, define $\tilde{c}_\phi : (E_\Lambda)^{*2} \to A$ by
\[
\tilde{c}_\phi(u,w) := S_\phi(uw) - S_\phi(u) - S_\phi(w).
\]
\end{ntn}

\begin{lem}\label{lem:csubphi}
For all $(u,w) \in (E_\Lambda)^{*2}$, we have
\[
S_\phi(uw) = S_\phi(\overline{\pi(u)}\, \overline{\pi(w)}) + S_\phi(u) +
S_\phi(w),
\]
and hence $\tilde{c}_\phi(u,w) = S_{\phi}(\overline{\pi(u)}\,\overline{\pi(w)})$. Moreover,
$\tilde{c}_\phi \in \Zcat2(E_\Lambda, A)$; that is,
\begin{equation}\label{eq:tildec cocycle id}
    \tilde{c}_\phi(u,w) + \tilde{c}_\phi(uw,x) = \tilde{c}_\phi(w,x) + \tilde{c}_\phi(u,wx)
\end{equation}
for all $(u,w,x) \in (E_\Lambda)^{*3}$.
\end{lem}
\begin{proof}
Fix a path $\alpha$ in $F_\Lambda$ from $u$ to $\overline{\pi(u)}$ and a path $\beta$ from $w$ to
$\overline{\pi(w)}$. There is a path $\alpha'$ from $uw$ to $\overline{\pi(u)}w$ with $|\alpha'| =
|\alpha|$ determined by $p(\alpha'_j) := p(\alpha_j)$ for all $j \le |\alpha|$. Likewise, there is
a path $\beta'$ from $\overline{\pi(u)}w$ to $\overline{\pi(u)}\,\overline{\pi(w)}$ with $|\beta'|
= |\beta|$ such that $p(\beta'_j) = \ell(u) + p(\beta_j)$ for all $j$. By definition of these
paths, we have
\begin{equation}\label{eq:tphi(alpha)=S(u)}
\tilde\phi(\alpha') = \tilde\phi(\alpha) = S_\phi(u)\quad\text{ and }\quad
\tilde\phi(\beta') = \tilde\phi(\beta) = S_\phi(w).
\end{equation}

Since $\overline{\pi(uw)}$ is the unique sink in $F_{\pi(uw)}$, there is a path $\gamma$ from
$\overline{\pi(u)}\,\overline{\pi(w)}$ to $\overline{\pi(uw)}$, and we have
\begin{equation}\label{eq:tphi(gamma)}
\tilde\phi(\gamma) = S_\phi(\overline{\pi(u)}\,\overline{\pi(w)}).
\end{equation}
Then $\gamma\beta'\alpha'$ is a path from $uw$ to $\overline{\pi(uw)}$, and hence
$\tilde\phi(\gamma\beta'\alpha') = S_\phi(uw)$. Using that $\tilde\phi$ is a functor, and then
equations \eqref{eq:tphi(alpha)=S(u)}~and~\eqref{eq:tphi(gamma)}, we now calculate:
\[
S_\phi(uw)
    = \tilde\phi(\gamma\beta'\alpha')
    = \tilde\phi(\gamma) + \tilde\phi(\beta') + \tilde\phi(\alpha')
    = S_{\phi}(\overline{\pi(u)}\,\overline{\pi(w)}) + S_\phi(w) + S_\phi(u),
\]
proving the first assertion of the lemma. The second assertion
follows immediately from the definition of $\tilde{c}_\phi$.

For the final assertion, we calculate
\begin{align*}
\tilde{c}_\phi(u,w) + \tilde{c}_\phi(uw,x)
    &= \big(S_{\phi}(uw) - S_\phi(u) - S_\phi(w)\big) + \big(S_\phi(uwx) - S_\phi(uw) - S_\phi(x)\big) \\
    &= S_\phi(uwx) - S_\phi(u) - S_\phi(w) - S_\phi(x).
\end{align*}
A similar calculation yields
\[
\tilde{c}_\phi(w,x) + \tilde{c}_\phi(u,wx) = S_\phi(uwx) - S_\phi(u) - S_\phi(w) - S_\phi(x)
\]
also.
\end{proof}

\begin{proof}[Proof of Theorem~\ref{thm:2-cocycle map}]
Fix $\phi \in \Zcub2(\Lambda, A)$. We show that $c_\phi$ satisfies (\ref{it:on
squares})~and~(\ref{it:already preferred}). For~(\ref{it:on squares}), suppose that $f \in
\Lambda^{e_i}$ and $g \in s(f)\Lambda^{e_j}$. If $i \le j$, then $fg = \overline{fg}$, so
$S_\phi(fg) = 0$ as required. If $i
> j$, we factorise $fg = g'f'$ where $d(g') = e_j$ and $d(f') = e_i$, and note that $\overline{fg}
= g'f'$, and $S_\phi(fg) = \phi(fg)$ by definition.

For~(\ref{it:already preferred}), suppose that
$\overline{\mu}\,\overline{\nu} = \overline{\mu\nu}$. Then
$c_\phi(\mu,\nu) = S_\phi(\overline{\mu\nu}) = 0$ by
definition. Since $\overline{r(\lambda)}\,\overline{\lambda} =
\overline{\lambda} = \overline{r(\lambda)\lambda}$ for all
$\lambda$, and similarly for $ \overline{\lambda}\,
\overline{s(\lambda)}$, (\ref{it:already preferred}) follows.

To see that $c_\phi$ is a cocycle, it remains to show that it
satisfies the cocycle identity
\[
c_\phi(\lambda_1, \lambda_2) + c_\phi(\lambda_1\lambda_2, \lambda_3)
    = c_\phi(\lambda_2, \lambda_3) + c_\phi(\lambda_1, \lambda_2\lambda_3)
\]
for $(\lambda_1, \lambda_2, \lambda_3) \in \Lambda^{*3}$. By Lemma~\ref{lem:csubphi}, we have
$c_\phi(\pi(u), \pi(w)) = \tilde{c}_\phi(u,w)$ for any $(u,w) \in E^{*2}_\Lambda$, and hence
\[
c_\phi(\lambda_1, \lambda_2) + c_\phi(\lambda_1\lambda_2, \lambda_3)
    = \tilde{c}_\phi(\overline{\lambda_1}, \overline{\lambda_2}) + \tilde{c}_\phi(\overline{\lambda_1}\,\overline{\lambda_2}, \overline{\lambda_3}),
\]
and
\[
c_\phi(\lambda_2, \lambda_3) + c_\phi(\lambda_1, \lambda_2\lambda_3)
    = \tilde{c}_\phi(\overline{\lambda_2}, \overline{\lambda_3}) + \tilde{c}_\phi(\overline{\lambda_1}, \overline{\lambda_2}\,\overline{\lambda_3}),
\]
so the cocycle identity for $c_\phi$ follows from the cocycle
identity~\eqref{eq:tildec cocycle id} for $\tilde{c}_\phi$.

To see that $\phi \mapsto c_\phi$ is a homomorphism, observe
that if $\phi_1, \phi_2 \in \Zcub{2}(\Lambda, A)$, then
$S_{\phi_1 + \phi_2} = S_{\phi_1} + S_{\phi_2}$, and hence
$\tilde{c}_{\phi_1 + \phi_2} = \tilde{c}_{\phi_1} + \tilde{c}_{\phi_2}$. It then
follows that $c_{\phi_1 + \phi_2} = c_{\phi_1} + c_{\phi_2}$
also.

Finally, we must show that the assignment $\phi \mapsto c_\phi$ carries coboundaries to
coboundaries. Fix $\phi \in \Bcub2(\Lambda,A)$ and $f \in \Ccub1(\Lambda, A)$ such that $\phi =
\dcub1 f$. By definition of $\dcub1 : \Ccub1(\Lambda,A) \to \Ccub2(\Lambda,A)$,
\[
\phi(\lambda) = f(F^0_2(\lambda)) + f(F^1_1(\lambda)) -
f(F^0_1(\lambda)) - f(F^1_2(\lambda))
\]
for all $\lambda \in Q_2(\Lambda)$.
In  particular, if $d(\alpha) = e_i$ and $d(\beta) = e_j$ with $i < j$ and
if $\alpha\beta = \eta\zeta$ with $d(\eta) = e_j$ and $d(\zeta) = e_i$,
then $\phi(\alpha\beta) = f(\alpha) + f(\beta) - f(\eta) - f(\zeta)$.

Define $b : E_\Lambda \to A$ by $b(w) := \sum^{\ell(w)}_{i=1} f(w_i)$. We show by induction on
$h(w)$ that $S_\phi(w) = b(\overline{\pi(w)}) -b(w)$ for all $w \in F^0_\Lambda$. This is trivial
when $h(w) = 0$. Now suppose that $S_\phi(w) = b(\overline{\pi(w)}) - b(w)$ whenever $h(w) \le n$,
and fix $w \in F^0_\Lambda$ with $h(w) = n+1$. Fix $\alpha \in \overline{\pi(w)} F_\Lambda w$, and
let $w' := r(\alpha_{n+1})$. Then Lemma~\ref{lem:coboundary on transition graph} implies that
\begin{equation}\label{eq:apply ind hyp}
\begin{split}
S_\phi(w) = \sum^{n+1}_{i=1} \tilde\phi(\alpha_i)
    &= \tilde{\phi}(\alpha_{n+1}) + \sum^{n}_{i=1} \tilde\phi(\alpha_i) \\
    &= S_\phi(\alpha_1, \dots, \alpha_n) + \tilde{\phi}(\alpha_{n+1})
    = b(\overline{\pi(w)}) - b(w') + \tilde{\phi}(\alpha_{n+1}),
\end{split}
\end{equation}
where the last equality follows from the inductive hypothesis.

Let $j := p(\alpha_{n+1})$, and let $\lambda :=
\pi(w_jw_{j+1})$. We have $\tilde{\phi}(\alpha_{n+1}) =
\phi(\lambda)$ by definition of $\tilde\phi$. Since
$\alpha_{n+1}$ is an allowed transition, we have $C(w_j) >
C(w_{j+1})$, and hence
\[
w_j = F^0_1(\lambda), \quad w_{j+1} = F^1_2(\lambda),\quad w'_j = F^0_2(\lambda),\quad\text{ and }\quad w'_{j+1} = F^1_1(\lambda).
\]
Hence $\tilde{\phi}(\alpha_{n+1}) = \phi(\lambda) = f(w'_j) +
f(w'_{j+1}) - f(w_j) - f(w_{j+1})$. Combining this
with~\eqref{eq:apply ind hyp}, we have
\begin{align*}
S_\phi(w)
    &= b(\overline{\pi(w)}) - b(w') + f(w'_j) + f(w'_{j+1}) - f(w_j) - f(w_{j+1}) \\
    &\textstyle= b(\overline{\pi(w)}) - \Big(\sum^{\ell(w)}_{i=1} f(w'_i)\Big) + f(w'_j) + f(w'_{j+1}) - f(w_j) - f(w_{j+1}).
\end{align*}
Since $w'_i = w_i$ for $i \not\in \{j,j+1\}$, it follows that
\[\textstyle
S_\phi(w) = b(\overline{\pi(w)}) - \sum^{\ell(w)}_{i=1} f(w_i) = b(\overline{\pi(w)}) - b(w).
\]

Define $g \in \Ccat1(\Lambda, A)$ by $g(\lambda) = - b(\overline{\lambda})$
for $\lambda \in \Lambda$.
We prove that $c_\phi =  \dcat1 g$.  Fix $(\mu,\nu) \in \Lambda^{*2}$.
Then
\begin{align*}
c_\phi(\mu,\nu)
    = S_\phi(\overline{\mu}\,\overline{\nu})
    = b(\overline{\mu\nu}) - b(\overline{\mu}\,\overline{\nu})
    &= b(\overline{\mu\nu}) - b(\overline{\mu}) - b(\overline{\nu}) \\
    &= - g(\mu\nu)  + g(\mu) + g(\nu)
    = (\dcat1 g)(\mu,\nu).
\end{align*}
Hence $c_\phi \in  \Bcat2(\Lambda, A)$, so $[\phi] \mapsto [c_\phi]$ is a homomorphism from
$\Hcub2(\Lambda,A)$ to $\Hcat2(\Lambda,A)$.
\end{proof}

\section{Central extensions of \texorpdfstring{$k$}{k}-graphs}\label{sec:extensions}

In this section we prove that the map $[\phi] \mapsto [c_\phi]$ of Theorem~\ref{thm:2-cocycle map}
is an isomorphism (see Theorem~\ref{thm:2-cohomology iso}). To prove this, we introduce the notion
of a central extension of a $k$-graph $\Lambda$ by an abelian group $A$. We show that the
collection of isomorphism classes of central extensions forms a group $\Ext(\Lambda,A)$ which is
isomorphic to $\Hcat2(\Lambda,A)$ (cf. \cite[Theorem 2.3]{BW} and
\cite[Proposition~I.1.14]{Renault1980}).  We show that each central extension is isomorphic to one
obtained from a cubical $2$-cocycle, which is unique modulo coboundaries.

Given a $k$-graph $\Lambda$ and an abelian group $A$, the set $\Lambda^0 \times A$ becomes a category with $r(v,a) = s(v,a) = v$ and $(v,a)(v,b) = (v, a+b)$.

\begin{dfn}\label{dfn:extension}
Let $A$ be an abelian group, and let $\Lambda$ be a $k$-graph.
An \emph{extension of $\Lambda$ by $A$} is a sequence
\[
\Xx  :  \Lambda^0 \times A \stackrel{\iota}{\longrightarrow}
X \stackrel{q}{\longrightarrow} \Lambda
\]
consisting of a small category $X$, a functor $\iota : \Lambda^0 \times A \to X$, and a surjective
functor $q : X \to \Lambda$ such that $q(\iota(v,a)) = v$ for all $v \in \Lambda^0$ and $a \in A$,
and such that whenever $q(x) = q(y)$, there exists a unique $a(x,y) \in A$ such that $x =
\iota(r(q(x)), a(x,y))y$. We say that $\Xx$ is a \emph{central extension} if it satisfies
$\iota(r(q(x)),a) x = x \iota(s(q(x)),a)$ for all $x \in X$ and $a \in A$.

As we do for $k$-graphs, for $x \in X$ we write $r(x)$ for $\id_{\cod(x)}$ and $s(x)$ for
$\id_{\dom(x)}$.
\end{dfn}

\begin{rmk}\label{rmk:extension manipulations}
Let $\Xx$ be an extension of a $k$-graph $\Lambda$ by an abelian group $A$. Then 
$\iota$ is injective and induces a bijection between $\Lambda^0$ and 
$\Obj(X)$. Since $q(\iota(v,a)) = v$ for all $(v,a) \in \Lambda^0 \times A$, we have 
$q(x) = q(y)$ if \emph{and only if} there exists $a \in A$ such that $\iota(r(x), a) x = 
y$. We then have $q(r(x)) = q(\iota(r(x), a)) q(r(x)) = q(\iota(r(x), a)r(x)) = 
q(r(y))$. Since $q$ is injective on objects, it follows that $q(x) = q(y)$ implies $r(x) 
= r(y)$ (and similarly $s(x) = s(y)$) for all $x,y \in X$.
\end{rmk}

\begin{ntn}
Given an extension $\Xx$ of $\Lambda$ by $A$, it is unambiguous, and frequently convenient, to
write $a \cdot x$ for $\iota_X(r(x),a)x$ and $x \cdot a$ for $x \iota_X(s(x),a)$. In this notation,
$q(x) = q(y)$ if and only if $x = a(x,y) \cdot y$, and $\Xx$ is a central extension precisely if $a
\cdot x = x \cdot a$ for all $x \in X$ and $a \in A$. We implicitly identify $\Lambda^0$ with the
identity morphisms in $X$ via the bijection $v \mapsto \iota_X(v, 0)$. This allows us to write $a
\cdot v$ or $v \cdot a$ (as appropriate) for $\iota_X(v,a)$. With this convention, we also have
$q_X(v) = v$, and for $x \in X$, we regard $r(x)$ and $s(x)$ as elements of $\Lambda^0$.

\end{ntn}

That $\iota$ is a functor implies that $a \cdot (b\cdot x) = (a+b)\cdot x$ for all $a,b \in A$ and
$x \in X$. Since composition in $X$ is also associative, we have identities like $x(a \cdot y) =
(x\cdot a)y = (a\cdot x) y = a\cdot(xy)$. In particular, the expression $a\cdot xy$ is unambiguous.

\begin{lem}\label{lem:a(x,y)=a(xz,yz)}
Let $\Lambda$ be a $k$-graph, $A$ an abelian group, and
\[
\Xx : \Lambda^0 \times A \to X \to \Lambda
\]
a central extension of $\Lambda$ by $A$. If $(w,x,z), (w,y,z) \in X^{*3}$ and $q(x) = q(y)$, then
$a(x,y) = a(wxz,wyz)$. In particular, $a(xz,yz) = a(x,y) = a(wx,wy)$.

Moreover given elements $x_1, \dots, x_n \in X$ such that
$q(x_1) = q(x_2) = \cdots = q(x_n)$, we have $a(x_1, x_n) =
\sum^{n-1}_{i=1} a(x_i, x_{i+1})$.
\end{lem}
\begin{proof}
Using that $\Xx$ is a central extension, we calculate:
\[
a(x,y) \cdot wyz = w(a(x,y)\cdot y)z = wxz.
\]
The first assertion of the lemma therefore follows from uniqueness of $a(wxz,wyz)$. The second
assertion follows from the first applied with $w = r(x)$ and with $z = s(x)$.

The final assertion follows from a straightforward induction: it is trivial when $n = 2$. Suppose
as an inductive hypothesis that it holds for $n \le N$, and fix $x_1, \dots x_{N+1}$ with $q(x_i) =
q(x_j)$ for all $i,j$. Then
\begin{align*}
\Big(\sum^{N}_{i=1} a(x_i, x_{i+1})\Big) \cdot x_{N+1}
    &= \Big(\sum^{N-1}_{i=1} a(x_i, x_{i+1})\Big) \cdot a(x_N, x_{N+1})\cdot x_{N+1} \\
    &= \Big(\sum^{N-1}_{i=1} a(x_i, x_{i+1})\Big) \cdot x_N
    = x_1
\end{align*}
by the inductive hypothesis.
\end{proof}

\begin{ntn}\label{ntn:twist product}
Let $\Lambda$ be a $k$-graph, let $A$ be an abelian group, and
let
\[
\Xx : \Lambda^0 \times A \stackrel{\iota_X}{\longrightarrow} X \stackrel{q_X}{\longrightarrow} \Lambda
    \quad\text{ and }\quad
\Yy : \Lambda^0 \times A \stackrel{\iota_Y}{\longrightarrow} Y \stackrel{q_Y}{\longrightarrow} \Lambda
\]
be central extensions of $\Lambda$ by $A$. Let $X *_\Lambda Y
:=\{(x,y) \in X \times Y : q_X(x) = q_Y(y)\}$. Define a
relation $\sim_A$ on $X *_\Lambda Y$ by
$(x,y) \sim_A (-a\cdot x, a\cdot y)$ for all
$(x,y) \in X *_\Lambda Y$ and $a \in A$.
\end{ntn}

\begin{lem}\label{lem:er on product}
With $\Lambda$, $A$, $\Xx$ and $\Yy$ as in
Notation~\ref{ntn:twist product}, the relation $\sim_A$ is an
equivalence relation, and satisfies
\[
    (x \cdot a,y) = (a\cdot x, y) \sim_A (x, a\cdot y) = (x,y\cdot a)
\]
for all $(x,y) \in X *_\Lambda Y$ and all $a \in A$.
\end{lem}
\begin{proof}
The relation $\sim_A$ is clearly reflexive. For symmetry,
observe that
\[
(-a\cdot x, a\cdot y) \sim_A \big(a\cdot(-a\cdot x), -a\cdot(a \cdot y) \big) = (x,y).
\]
For transitivity, observe that
\[
(-b \cdot (-a \cdot x), b \cdot (a \cdot y))
    = (-(a+b)\cdot x, (a+b)\cdot y) \sim_A (x,y).
\]
For the final assertion, we first establish the middle equality
by calculating
\[
(a\cdot x, y) \sim (-a \cdot (a\cdot x), a\cdot y) = (x, a \cdot y).
\]
the other equalities follow because $\Xx$ and $\Yy$ are central
extensions.
\end{proof}

\begin{lem}\label{lem:extension ops}
With the hypotheses of Lemma~\ref{lem:er on product}, let $Z(X,Y) := X *_\Lambda Y / \sim_A$, and
for $(x,y) \in X *_\Lambda Y$, let $[x,y]$ denote its equivalence class in $Z(X,Y)$. There are
well-defined maps $r,s : Z(X,Y) \to Z(X,Y)$ such that
\[
r([x,y]) = \big[r(x), r(y)\big]  \quad\text{ and }\quad  s([x,y]) = \big[s(x), s(y)\big]
        \quad\text{ for all $(x,y) \in X *_\Lambda Y$.}
\]
There is also a well-defined composition determined by $[x_1, y_1][x_2, y_2] = [x_1x_2, y_1y_2]$
whenever $s([x_1, y_1]) = r([x_2, y_2])$.
\end{lem}
\begin{proof}
Suppose that $(x,y) \sim_A (x',y')$, say $(x,y) = (-a\cdot x', a\cdot y')$. Then $q_X(x) = q_X(x')$
and $q_Y(y) = q_Y(y')$, so $(r(x), r(y)) = (r(x'), r(y'))$ by Remark~\ref{rmk:extension
manipulations}. So $r : Z(X,Y) \to Z(X,Y)$ is well defined. A similar argument shows that $s$ is
well defined.

To see that composition is well defined, fix $a,b \in A$. Since $X$ is a central extension, we
calculate:
\[
(-a\cdot x_1) (-b\cdot x_2)
    = -a\cdot (-b\cdot x_1) x_2
    = -(a+b)\cdot x_1x_2.
\]
Similarly $(a \cdot y_1)(b\cdot y_2) = (a+b)\cdot y_1y_2$. In
particular,
\[
((-a \cdot x_1)(-b\cdot x_2), (a\cdot y_1)(b\cdot y_2)) \sim_A (x_1x_2,y_1y_2). \qedhere
\]
\end{proof}

If $[x,y] = [x',y']$ in $Z(X,Y)$, then $x' = -a\cdot x$ for some $a \in A$, so $q(x') = q(x)$. So
we may define $\iota = \iota_{Z(X,Y)} : \Lambda^0 \times A \to Z(X,Y)$ and $q = q_{Z(X,Y)} : Z(X,Y)
\to \Lambda$ by
\begin{equation}\label{eq:ext sum maps}
\iota(v,a) := [\iota_X(v,a), \iota_Y(v, 0)]\qquad\text{ and }\qquad
q([x,y]) := q_X(x).
\end{equation}
Lemma~\ref{lem:er on product} implies that
\begin{equation}\label{eq:a on either side}
\begin{split}
\iota(v,a) = [\iota_X(v,a), \iota_Y(v, 0)]
    &= [\iota_X(v,a), \iota_Y(v,-a)\iota_Y(v,a)]\\
    &= [\iota_X(v,-a)\iota_X(v,a),\iota_Y(v,a)] = [\iota_X(v,0), \iota_Y(v,a)].
\end{split}
\end{equation}

\begin{lem}\label{lem:ext addition}
Let $\Lambda$ be a $k$-graph, let $A$ be an abelian group, and let $X$ and $Y$ be central
extensions of $\Lambda$ by $A$ as in Notation~\ref{ntn:twist product}. Then $Z(X,Y)$ is a small
category under the operations described in Lemma~\ref{lem:extension ops} and with the identity
morphism corresponding to an object $v \in \Lambda^0$ given by $\id_v = [v, v]$.

Let $\iota := \iota_{Z(X,Y)}$ and $q := q_{Z(X,Y)}$ be as in~\eqref{eq:ext sum maps}. Then
$q([x,y]) = q_Y(y)$ for all $[x,y] \in Z(X,Y)$, and
\[
\Xx + \Yy : \Lambda^0 \times A \stackrel{\iota}{\longrightarrow} Z(X,Y) \stackrel{q}{\longrightarrow} \Lambda
\]
is a central extension of $\Lambda$ by $A$ and satisfies
\begin{equation}\label{eq:a-maps add}
a([x,y],[x',y']) = a(x,x') + a(y,y')\text{ whenever }q([x,y]) =
q([x',y']).
\end{equation}
Finally, $\iota_{Z(X,Y)}(v,a) = [\iota_X(v,b), \iota_X(v,a - b)]$ for all $v \in \Lambda^0$ and
$a,b \in A$.
\end{lem}

\begin{rmk}
The rather suggestive notation $\Xx + \Yy$ is justified by
Example~\ref{eg:cocycle extension} and Proposition~\ref{prp:Ext
group} below.
\end{rmk}

\begin{proof}[Proof of Lemma~\ref{lem:ext addition}]
Routine checks show that $Z(X,Y)$ is a category. It is small because $X$ and $Y$ are.

That $q(x,y) = q_Y(y)$ for all $(x,y) \in X *_\Lambda Y$ is
just a combination of the definitions of the map $q$ and the
space $X *_\Lambda Y$. Using this it is routine to see that
$\iota$ and $q$ are functors (the operations in $X *_\Lambda Y$
being coordinate-wise). For $v \in \Lambda^0$, we have
\[
q(\iota(v,0)) = q([\iota_X(v,0), \iota_Y(v,0)]) = q_X(\iota_X(v,0)) = v
\]
since $\Xx$ is an extension. Moreover, if $q([x,y]) = q([x',y'])$, then $q_X(x) = q_X(x') = q_Y(y')
= q_Y(y)$, so there exists a unique element $a = a(x,x') \in A$ such that $x = a\cdot x'$, and a
unique $b = b(y,y')$ such that $y = b\cdot y'$. Applying~\eqref{eq:a on either side} in the second equality, we calculate:
\begin{align*}
\iota(r(x),a+b)[x',y']
    &= \iota(r(x), a)\iota(r(x),b)[x',y'] \\
    &= [\iota_X(r(x),a),\iota_Y(r(x),0)] [\iota_X(r(x),0), \iota_Y(r(x),b)][x',y']\\
    &= [(a+0)\cdot x', (0+b)\cdot y']\\
    &= [x,y].
\end{align*}
For uniqueness of $a+b$,
suppose that $\iota(r(x), c)[x',y'] = [x,y]$. Then $(c\cdot x',y')
\sim_A (x,y)$, so there exists $d \in A$ such that $c \cdot x'
= -d \cdot x$ and $y' = d \cdot y$. Hence $y = -d\cdot y'$ and
uniqueness of $b(y,y')$ forces $b = b(y,y') = -d$; and then
\[
x = d\cdot c\cdot x' = (c - b)\cdot x',
\]
and uniqueness of $a(x,x')$ forces $a = a(x,x') = c - b$. Hence $c = a(x,x') + b(y,y')$ as
required. Thus $\Xx + \Yy$ is an extension of $\Lambda$ by $A$. It is central because each of $\Xx$
and $\Yy$ is central.

The final assertion follows from~\eqref{eq:a on either side}.
\end{proof}

We next construct from each $c \in \Zcat2(\Lambda, A)$ a central extension of $\Lambda$ by $A$ by
twisting the composition in $\Lambda \times A$.

\begin{ntn}
Let $\Lambda$ be a $k$-graph, let $A$ be an abelian group, and fix $c \in \Zcat2(\Lambda,A)$. Let
$X_c ( \Lambda, A)$ be the small category with underlying set and structure maps identical to the
cartesian-product category $\Lambda \times A$ and with composition defined by
\[
(\mu,a)(\nu,b) := (\mu\nu, c(\mu,\nu) + a + b).
\]
We will usually suppress the $\Lambda$ and $A$ in our notation, and write $X_c$ for $X_c(\Lambda,
A)$.
\end{ntn}

\begin{example}\label{eg:cocycle extension}
Let $\Lambda$ be a $k$-graph, let $A$ be an abelian group, and fix $c \in \Zcat2(\Lambda,A)$.
Define $\iota : \Lambda^0 \times A \to \extshrt{c}{\Lambda}{A}$ by inclusion of sets, and define $q
: \extshrt{c}{\Lambda}{A} \to \Lambda$ by $q(\lambda,a) := \lambda$. Then
\[
\Extshrt{c}{\Lambda}{A} : \Lambda^0 \times A \stackrel{\iota}{\longrightarrow} \extshrt{c}{\Lambda}{A} \stackrel{q}{\longrightarrow} \Lambda
\]
is a central extension of $\Lambda$ by $A$.

In particular, the trivial cocycle $0 : \Lambda^{*2} \to A$ given by $0(\mu,\nu) = 0$ for all
$\mu,\nu$ gives rise to the \emph{trivial extension} $\Extshrt{0}{\Lambda}{A} : \Lambda^0 \times A
\to \extshrt{0}{\Lambda}{A} \to \Lambda$, where $X_0=\Lambda \times A$ is the cartesian-product
category (with un-twisted composition).
\end{example}

\begin{dfn}\label{dfn:ext isomorphism}
Let $\Lambda$ be a $k$-graph, and let $A$ be an abelian group. We say that two extensions
\[
\Xx : \Lambda^0 \times A \stackrel{\iota_X}{\longrightarrow} X \stackrel{q_X}{\longrightarrow} \Lambda\quad\text{ and }\quad
\Yy : \Lambda^0 \times A \stackrel{\iota_Y}{\longrightarrow} Y \stackrel{q_Y}{\longrightarrow} \Lambda
\]
of $\Lambda$ by $A$ are \emph{isomorphic} if there is a bijective functor $f : X \to Y$ such that
the following diagram commutes.
\[\begin{tikzpicture}
    \node[rectangle, inner sep=1pt] (A) at (-1.5,0) {$\Lambda^0 \times A$};
    \node (X) at (0,1) {$X$};
    \node (Y) at (0,-1) {$Y$};
    \node[rectangle, inner sep=1pt] (L) at (1.5,0) {$\Lambda$};
    \draw[-latex] (A) -- (X) node[pos=0.5,anchor=south east,inner sep=1pt] {$\iota_X$};
    \draw[-latex] (A) -- (Y) node[pos=0.5,anchor=north east,inner sep=1pt] {$\iota_Y$};
    \draw[-latex] (X) -- (L) node[pos=0.5,anchor=south west,inner sep=1pt] {$q_X$};
    \draw[-latex] (Y) -- (L) node[pos=0.5,anchor=north west,inner sep=1pt] {$q_Y$};
    \draw[-latex] (X.south)--(Y.north) node[pos=0.5,anchor=east,inner sep=1.5pt] {$f$};
\end{tikzpicture}\]
We call $f$ an isomorphism of $\Xx$ with $\Yy$. We write $\Ext(\Lambda,A)$ for the set of
isomorphism classes of extensions of $\Lambda$ by $A$.
\end{dfn}

Fix a central extension
\[
\Xx : \Lambda^0 \times A \stackrel{\iota}{\longrightarrow} X \stackrel{q}{\longrightarrow} \Lambda.
\]
Let $\overline{X} := \{\overline{x} : x \in X\}$ be a copy of the category $X$. Define
$\overline\iota : \Lambda^0 \times A \to \overline{X}$ by $\overline\iota(v,a) :=
\overline{\iota(v,-a)}$. Define $\overline{q} : \overline{X} \to \Lambda$ by
$\overline{q}(\overline{x}) = q(x)$. Then
\[
-\Xx : \Lambda^0 \times A \stackrel{\overline\iota}{\longrightarrow} \overline{X}
    \stackrel{\overline{q}}{\longrightarrow} \Lambda
\]
is also a central extension of $\Lambda$ by $A$. Observe that $a \cdot \overline{x} = \overline{-a
\cdot x}$ for $a \in A$ and $x \in X$.

\begin{prop}\label{prp:Ext group}
Let $\Lambda$ be a $k$-graph, and let $A$ be an abelian group. Then the formula $[\Xx] + [\Yy] :=
[\Xx + \Yy]$ determines a well-defined operation under which $\Ext(\Lambda,A)$ is an abelian group
with identity element $[\Extshrt{0}{\Lambda}{A}]$, the class of the trivial extension. Moreover,
$-[\Xx] = [-\Xx]$ for each extension $\Xx$ of $\Lambda$ by $A$.
\end{prop}
\begin{proof}
We must first check that $[\Xx] + [\Yy]$ is well-defined.

Suppose that $\Xx$, $\Xx'$, $\Yy$ and $\Yy'$ are central extensions of $\Lambda$ by $A$, and that
$f_X$ is an isomorphism of $\Xx$ with $\Xx'$ and $f_Y$ is an isomorphism of $\Yy$ with $\Yy'$.

Then $(f_X,f_Y) : X *_\Lambda Y \to X' *_\Lambda Y'$ given by $(f_X,f_Y)(x,y) = (f_X(x),
f_Y(y))$ is bijective. For $a \in A$ and $x \in X$,
\begin{align*}
f_X(a \cdot x) = f_X(\iota_X(r(x),a)x) &= f_X(\iota_X(r(x),a))f_X(x)\\ 
	&= \iota_{X'}(r(x),a)\cdot f_X(x) = a\cdot f_X(x)
\end{align*}
and similarly $f_Y(a\cdot y) = a\cdot f_Y(y)$ for each $y \in Y$. Hence
\begin{align*}
(f_X,f_Y)(-a\cdot x, a\cdot y)
    &= \big(f_X(-a\cdot x), f_Y(a\cdot y)\big)\\
    &= (-a\cdot f_X(x), a\cdot f_Y(y))
    \sim_A (f_X(x), f_Y(y)).
\end{align*}
It follows that there is a well-defined map $(f_X,f_Y)^\sim : Z(X,Y) \to Z(X',Y')$ determined by
$(f_X,f_Y)^\sim([x,y]) = [f_X(x), f_Y(y)]$. This map is bijective because $[x',y'] \mapsto
[f_X^{-1}(x), f_Y^{-1}(y)]$ is a well-defined inverse. One checks that $(f_X,f_Y)^\sim$
is an isomorphism between $\Xx + \Yy$ and $\Xx' + \Yy'$, so $[\Xx] + [\Yy]$ is well-defined.

It is routine to check that $[[x,y],z] \mapsto [x,[y,z]]$ determines an isomorphism $([\Xx] +
[\Yy])+[\Zz] \cong [\Xx] + ([\Yy]+[\Zz])$ so addition in $\Ext(\Lambda, A)$ is associative.
Likewise $[x,y] \mapsto [y,x]$ determines an isomorphism $[\Xx]+[\Yy] \cong [\Yy]+[\Xx]$, so the
operation is commutative.

To see that $[\Xx] + [\Xx_0] = [\Xx]$ for all $[\Xx] \in \Ext(\Lambda,A)$, we show that
$[x,(\lambda, a)] \mapsto a\cdot x$ determines an isomorphism of $Z(X, X_0)$ onto $X$ with inverse
given by $x \mapsto [x,(q(x),0)]$. We must show first that the formula $[x,(\lambda, a)] \mapsto
a\cdot x$ is well-defined. If $[x, (\lambda,a)] = [y, (\mu,b)]$ then there exists $c \in A$ such
that $y = -c\cdot x$ and $(\mu,b) = c\cdot(\lambda,a) = (\lambda, a+c)$. In particular, $q(x) =
\lambda = \mu = q(y)$, and $c = b-a$ is then the unique element $a(x,y)$ of $A$ such that $x =
a(x,y)\cdot y$. Hence
\[
a\cdot x = a\cdot((b-a)\cdot y) = b\cdot y,
\]
so the formula $[x,(\lambda, a)] \mapsto a\cdot x$ is well-defined. The map $x \mapsto [x,
(q(x),0)]$ is an inverse, and these maps determine an isomorphism of $\Xx + \Xx_0$ with $\Xx$.

Finally, we must show that $[\Xx] + [-\Xx] = [\Extshrt{0}{\Lambda}{A}]$. For this we show that the
map $[x,\overline{y}] \mapsto (q(x), a(x,y))$ is an isomorphism with inverse $(\lambda,a) \mapsto
[x, -a\cdot \overline{x}]$ for any $x$ such that $q(x) = \lambda$. We first check that $[x,
\overline{y}] \mapsto (q(x), a(x,y))$ is well-defined. Since $[x, \overline{y}] = [x',
\overline{y'}]$ in $Z(X,\overline{X})$ implies that $q(x) = q(x') = q(y) = q(y')$, it suffices to
show that $a(x,y) = a(x',y')$. To see this, observe that since $[x, \overline{y}] = [x',
\overline{y'}]$, there exists a unique $b \in A$ such that $x = -b\cdot x'$ and $\overline{y} =
b\cdot \overline{y'}$, so $y = -b \cdot y'$. Hence
\[
a(x',y')\cdot y
    = \big(a(x',y') - b\big)\cdot y'
    = -b \cdot x'
    = x.
\]
So uniqueness of $a(x,y)$ forces $a(x',y') = a(x,y)$. Thus $[x,\overline{y}] \mapsto (q(x),
a(x,y))$ is well-defined. We claim that the formula $(\lambda,a) \mapsto [x, -a\cdot \overline{x}]$
does not depend on the choice of $x$ such that $q(x) = \lambda$. To see this, suppose $q(y) =
\lambda$ also. Then $x = a(x,y)\cdot y$. Hence, using once again that $-a(x,y) \cdot \overline{y} =
\overline{a(x,y) \cdot y} = \overline{x}$, we see that
\[
[x, -a\cdot \overline{x}]
    = [a(x,y)\cdot y, -a \cdot (-a(x,y) \cdot \overline{y})]
    = [a(x,y)\cdot y, -a(x,y) \cdot (-a\cdot \overline{y})]
    = [y, -a\cdot \overline{y}].
\]
It is now routine to see that $[x,\overline{y}] \mapsto [q(x), a(x,y)]$ determines an isomorphism
from $\Xx + (-\Xx)$ to $\Extshrt{0}{\Lambda}{A}$.
\end{proof}

Our next result shows that every central extension of $\Lambda$ by $A$ is isomorphic to one of the
form $\Xx_c$ described in Example~\ref{eg:cocycle extension}, and that the assignment $c \mapsto
\Xx_c$ determines an isomorphism from $\Hcat2(\Lambda,A)$ to $\Ext(\Lambda,A)$.

Let $\Lambda$ be a $k$-graph, and let $A$ be an abelian group. Let
\[
\Xx : \Lambda^0 \times A \stackrel{\iota}{\longrightarrow} X \stackrel{q}{\longrightarrow} \Lambda
\]
be a central extension of $\Lambda$ by $A$. A \emph{normalised section for $q$} is a function
$\sigma : \Lambda \to X$ such that $q \circ \sigma$ is the identity map on $\Lambda$ and such that
$\sigma(v) = \iota(v,0)$ for all $v \in \Lambda^0$. A normalised section for $q$ is typically not
multiplicative.

\begin{thm}\label{thm:H2=Ext}
Let $\Lambda$ be a $k$-graph, let $A$ be an abelian group, and let $\Xx$ be an extension of
$\Lambda$ by $A$. For each normalised section $\sigma$ for $q : X \to \Lambda$, define $c_\sigma :
\Lambda^{*2} \to A$ by $c_\sigma(\mu,\nu)\cdot\sigma(\mu\nu) = \sigma(\mu)\sigma(\nu)$; that is,
$c_\sigma(\mu,\nu) = a(\sigma(\mu)\sigma(\nu), \sigma(\mu\nu))$. Then $c_\sigma$ is a $2$-cocycle.
If $\sigma'$ is any other normalised section for $q$, then $c_\sigma$ and $c_{\sigma'}$ are
cohomologous. Finally, the assignment $[\Xx] \mapsto [c_\sigma]$ for any normalised section
$\sigma$ for $q$ is an isomorphism $\theta : \Ext(\Lambda,A) \cong \Hcat2(\Lambda,A)$ with inverse
given by $\theta^{-1}([c]) = [\Extshrt{c}{\Lambda}{A}]$.
\end{thm}
\begin{proof}
We check that $c_\sigma$ is a $2$-cocycle. Fix $(\mu, \nu) \in \Lambda^{*2}$. Then
$\sigma(\mu)\sigma(\nu) = c_\sigma(\mu,\nu)\cdot\sigma(\mu\nu)$. If $\mu$ or $\nu$ is in
$\Lambda^0$, then $\sigma(\mu)\sigma(\nu) = \sigma(\mu\nu)$ so $c_\sigma(\mu,\nu) = 0$.  Hence,
$c_\sigma  \in \Ccat2(\Lambda,A)$. Fix $(\lambda,\mu,\nu) \in \Lambda^{*3}$. By uniqueness of
$a(\sigma(\lambda)\sigma(\mu)\sigma(\nu),\sigma(\lambda\mu\nu))$ (see
Lemma~\ref{lem:a(x,y)=a(xz,yz)}) it suffices to show that
\[
\big(c_\sigma(\lambda,\mu) + c_\sigma(\lambda\mu,\nu)\big) \cdot \sigma(\lambda\mu\nu)
    = \sigma(\lambda)\sigma(\mu)\sigma(\nu)
    = \big(c_\sigma(\mu,\nu)+c_\sigma(\lambda,\mu\nu)\big)\cdot\sigma(\lambda\mu\nu).
\]
We just verify the first equality; the second follows from similar considerations. We calculate
\[
\sigma(\lambda)\sigma(\mu)\sigma(\nu)
    = c_\sigma(\lambda,\mu)\cdot \sigma(\lambda\mu)\sigma(\nu)
    = \big(c_\sigma(\lambda,\mu) + c_\sigma(\lambda\mu,\nu)\big) \cdot \sigma(\lambda\mu\nu).
\]
Hence, $c_\sigma  \in \Zcat2(\Lambda,A)$.

Now suppose that $\sigma'$ is another normalised section for $q$. For each $\lambda \in \Lambda$,
we have $q(\sigma(\lambda)) = q(\sigma'(\lambda))$, so there is a unique $b(\lambda) :=
a(\sigma(\lambda), \sigma'(\lambda)) \in A$ such that $\sigma(\lambda) =
b(\lambda)\cdot\sigma'(\lambda)$. Since $\sigma$ and $\sigma'$ are normalised, $b(v) = 0$ for all
$v \in \Lambda^0$, so $b \in \Ccub1(\Lambda, A)$.

 If $s(\mu) = r(\nu)$ in $\Lambda$, then
\begin{align*}
(b(\mu\nu) - b(\mu) - b(\nu))\cdot\sigma(\mu)\sigma(\nu)
    &= b(\mu\nu) \cdot\sigma'(\mu)\sigma'(\nu)\\
    &= (b(\mu\nu) + c_{\sigma'}(\mu,\nu))\cdot\sigma'(\mu\nu) \\
    &= c_{\sigma'}(\mu,\nu)\cdot\sigma(\mu\nu) \\
    &= (c_{\sigma'}(\mu,\nu) - c_{\sigma}(\mu,\nu))\cdot \sigma(\mu)\sigma(\nu).
\end{align*}
Hence $c_{\sigma'}(\mu,\nu) - c_{\sigma}(\mu,\nu) = (\dcat1b)(\mu,\nu)$, so $c_\sigma$ and
$c_{\sigma'}$ are cohomologous.

For the final assertion, we must first check that $[\Xx] \mapsto [c_\sigma]$ is well-defined. If
$f$ is an isomorphism of extensions $\Xx$ and $\Yy$, and if $\sigma$ is a section for $q_{{X}}$,
then $\sigma' := {f} \circ \sigma$ is a section for $q_{{Y}}$. Since
\[
a(x,x')\cdot f(x') = f(a(x,x')\cdot x') = f(x)
\]
for all $x,x' \in X$ with $q(x) = q(x')$, we have $a(f(x),f(x')) = a(x,x')$ for all $x,x' \in X$.
In particular, since $f$ is a functor,
\begin{align*}
c_{\sigma'}(\mu,\nu)
    &= a(\sigma'(\mu)\sigma'(\nu), \sigma'(\mu\nu)) \\
    &= a(f(\sigma(\mu)\sigma(\nu)), f(\sigma(\mu\nu)))
    = a(\sigma(\mu)\sigma(\nu), \sigma(\mu\nu))
    = c_\sigma(\mu,\nu).
\end{align*}
Thus $c_\sigma = c_{\sigma'}$. Since we already proved that distinct normalised sections for the
same central extension yield cohomologous categorical $2$-cocycles, it follows that for any pair of
sections $\sigma$ for $q_X$ and $\rho$ for $q_Y$ we have $[c_\sigma] = [c_\rho]$ in
$\Hcat2(\Lambda,A)$. Hence $[\Xx] \mapsto [c_\sigma]$ (for any section $\sigma$ for $q$) is well
defined. This map is additive by~\eqref{eq:a-maps add}, and hence a homomorphism.

To see that it is an isomorphism, it suffices to show that the map $c \mapsto
[\Extshrt{c}{\Lambda}{A}]$ from $\Zcat2(\Lambda,A)$ to $\Ext(\Lambda,A)$ determines a well-defined
map $[c] \mapsto [\Extshrt{c}{\Lambda}{A}]$ from $\Hcat2(\Lambda,A)$ to $\Ext(\Lambda,A)$, and that
this map is an inverse for $\theta$. Fix $c \in \Zcat2(\Lambda,A)$ and $b  \in \Ccat1(\Lambda, A)$
and let $c' = c - (\dcat1b)$ so that $[c] = [c'] \in \Hcat2(\Lambda, A)$. Define $f : X_c \to
X_{c'}$ by $f(\lambda,a) := (\lambda, a + b(\lambda))$. To see that $f$ is a functor, we calculate:
\[
f((\lambda,a)(\mu,a'))
    = f(\lambda\mu, c(\lambda,\mu) + a + a')
    = (\lambda\mu, c(\lambda,\mu) + b(\lambda\mu) + a + a'),
\]
and
\[
f(\lambda,a)f(\mu,a')
    = (\lambda, a + b(\lambda))(\mu, a' + b(\mu))
    = (\lambda\mu, c'(\lambda,\mu) + b(\lambda) + b(\mu) + a + a').
\]
Since $c - c' = \dcat1b$, we have
\[
c(\lambda,\mu) = c'(\lambda,\mu) + b(\lambda) + b(\mu) - b(\lambda\mu),
\]
so $c'(\lambda,\mu) + b(\lambda) + b(\mu) = c(\lambda,\mu) + b(\lambda\mu)$, giving
$f((\lambda,a)(\mu,a')) = f(\lambda,a)f(\mu,a')$. The functor $f$ is bijective because $(\lambda,a)
\mapsto (\lambda, a - b(\lambda))$ is an inverse. Hence $[c] \mapsto [\Extshrt{c}{\Lambda}{A}]$ is
a well-defined map from $\Hcat2(\Lambda,A)$ to $\Ext(\Lambda,A)$.

To see that $[c] \mapsto [\Extshrt{c}{\Lambda}{A}]$ is an inverse for $\theta$, fix an extension
\[
\Xx : \Lambda^0 \times A \stackrel{\iota}{\longrightarrow} X \stackrel{q}{\longrightarrow} \Lambda
\]
and a section $\sigma$ for $q$. We must show that $\Extshrt{c_\sigma}{\Lambda}{A}$ is isomorphic to
$\Xx$. We define $f : X \to \extshrt{c_\sigma}{\Lambda}{A}$ by $f(x) := \big(q(x), a(x,
\sigma(q(x)))\big)$ and $g : \extshrt{c_\sigma}{\Lambda}{A} \to X$ by $g(\lambda,a) :=
a\cdot\sigma(\lambda)$. Then
\begin{align*}
f\big(g(\lambda,a)\big)
    &= f\big(a\cdot\sigma(\lambda)\big)\\
    &= \big(q(a\cdot\sigma(\lambda)), a(a\cdot\sigma(\lambda), 
    \sigma(q(a\cdot\sigma(\lambda))))\big)
    = \big(\lambda, a(a\cdot\sigma(\lambda), \sigma(\lambda))\big).
\end{align*}
For all $x \in X$ and $b \in A$, we have $a(b\cdot x, x) = b$ by definition. Thus $f \circ g =
\id_{X_{c_\sigma}}$. Likewise, for $x \in X$, we have
\[
g\circ f(x)
    = g\big(q(x), a(x, \sigma(q(x)))\big)
    = a(x, \sigma(q(x))) \cdot \sigma(q(x))
    = x.
\]
So $f$ and $g$ are mutually inverse, and we just need to show that $f$ preserves composition. We
fix $(x,y) \in X^{*2}$ and calculate:
\[
f(xy)
    = \big(q(xy), a(xy, \sigma(q(xy)))\big)
    = \big(q(x)q(y), a(xy, \sigma(q(x)q(y)))\big),
\]
and
\begin{align*}
f&(x)f(y)\\
    &= \big(q(x), a(x, \sigma(q(x)))\big)\big(q(y), a(y, \sigma(q(y)))\big) \\
    &= \big(q(x)q(y), c_\sigma(q(x), q(y)) + a(x, \sigma(q(x))) + a(y, \sigma(q(y)))\big)\\
    &= \big(q(x)q(y), a\big(\sigma(q(x))\sigma(q(y)), \sigma(q(x)q(y))\big) + a(x, \sigma(q(x))) + a(y, \sigma(q(y)))\big).
\end{align*}
Since $\Xx$ is a central extension, we have
\begin{align*}
\big(a\big(\sigma(q(x))\sigma(q(y))&, \sigma(q(x)q(y))\big) + a(x, \sigma(q(x))) + a(y, \sigma(q(y)))\big)\cdot\sigma(q(x)q(y))\\
    &=\big(a(x, \sigma(q(x))) + a(y, \sigma(q(y)))\big)\sigma(q(x))\sigma(q(y))\\
    &= \big(a(x,\sigma(q(x)))\cdot \sigma(q(x))\big) \big(a(y, \sigma(q(y)))\cdot\sigma(q(y))\big)\\
    &= xy.
\end{align*}
Hence 
\[
a\big(\sigma(q(x))\sigma(q(y)), \sigma(q(x)q(y))\big) 
	+ a(x, \sigma(q(x))) + a(y, \sigma(q(y))) 
		= a(xy,\sigma(q(x)q(y))),
\] 
giving $f(xy) = f(x)f(y)$ as claimed. This shows that $f$
is a functor, and hence an isomorphism of extensions. So $[c] \mapsto [\Extshrt{c}{\Lambda}{A}]$ is
a left inverse for $\theta$. To see that it is a right inverse also, fix a cocycle $c \in
\Zcat2(\Lambda,A)$. Then the normalised section $\sigma_c : \lambda \mapsto (\lambda, 0)$ for $q :
X_c \to \Lambda$ satisfies $c_{\sigma_c} = c$.
\end{proof}

\begin{thm}\label{thm:2-cohomology iso}
The map  $\psi : \Hcub2(\Lambda,A) \to \Hcat2(\Lambda,A)$ of Theorem~\ref{thm:2-cocycle map} given by $[\phi] \mapsto [c_\phi]$ is an isomorphism.
\end{thm}
\begin{proof}
Fix a central extension
\[
\Xx : \Lambda^0 \times A \stackrel{\iota}{\longrightarrow} X
\stackrel{q}{\longrightarrow} \Lambda
\]
of $\Lambda$ by $A$.  For each $v \in \Lambda^0$, let $\sigma(v) \in X$ be the unique identity
morphism such that $q(\sigma(v)) = v$; that is, $\sigma(v) = \iota(v, 0)$. For each edge $e \in E^1_\Lambda$
in the skeleton of $\Lambda$, fix an element $\sigma(e) \in X$ such that $q(\sigma(e)) = e$. Extend
$\sigma$ to a section for $q$ by setting $\sigma(\lambda) :=
\sigma(\overline{\lambda}_1)\sigma(\overline{\lambda}_2) \cdots
\sigma(\overline{\lambda}_{|\lambda|})$ where $\lambda \mapsto \overline\lambda$ is the preferred
section for the quotient map $\pi : E_\Lambda \to \Lambda$ as in Notation~\ref{ntn:coloured graph}.

Define $\phi : Q_2(\Lambda) \to A$ by
\[
\phi(\lambda) := a(\sigma(\lambda), \sigma(F^0_1(\lambda))\sigma(F^1_2(\lambda)));
\]
that is, if $d(\lambda) = e_i + e_j$ with $i < j$, and if
$\lambda = fg = g'f'$ where $f,f' \in \Lambda^{e_i}$ and $g,g'
\in \Lambda^{e_j}$, then
\[
\phi(\lambda)\cdot \big(\sigma(g')\sigma(f')\big) = \sigma(f)\sigma(g).
\]

We check that $\phi$ is a cubical $2$-cocycle. Fix $\lambda \in Q_3(\Lambda)$, say
$d(\lambda) = e_i + e_j + e_l$ where $i < j < l$, and factorise
\[
\lambda = f_0 g_0 h_0 = f_0 h_1 g_1 = h_2 f_1 g_1 = h_2 g_2 f_2 = g_3 h_3 f_2 = g_3 f_3 h_0
\]
as in~\eqref{eq:commutingcube}, so $d(f_n) = e_i$, $d(g_n) = e_j$ and $d(h_n) = e_l$ for all $n$.

By definition of $\phi$, we have
\[
\phi(F^0_3(\lambda)) + \phi(F^1_2(\lambda)) + \phi(F^0_1(\lambda))
    = \phi(f_0g_0) + \phi(f_3h_0) + \phi(g_3h_3).
\]
Moreover,
\begin{gather*}
\phi(f_0g_0)\cdot\sigma(g_3)\sigma(f_3) = \sigma(f_0)\sigma(g_0),\quad
    \phi(f_3h_0)\cdot\sigma(h_3)\sigma(f_2)=\sigma(f_3)\sigma(h_0),\quad\text{ and}\\
    \phi(g_3h_3)\cdot\sigma(h_2)\sigma(g_2) = \sigma(g_3)\sigma(h_3).
\end{gather*}
Using this and three applications of
Lemma~\ref{lem:a(x,y)=a(xz,yz)}, we deduce that
\[
\big(\phi(f_0g_0) + \phi(f_3h_0) + \phi(g_3h_3)\big)\cdot \sigma(h_2)\sigma(g_2)\sigma(f_2)
    = \sigma(f_0)\sigma(g_0)\sigma(h_0) = \sigma(\lambda),
\]
and hence
\begin{align*}
\phi(F^0_3(\lambda)) + \phi(F^1_2(\lambda)) + \phi(F^0_1(\lambda)) &=
a(\sigma(\lambda), \sigma(h_2)\sigma(g_2)\sigma(f_2)).
\intertext{Symmetric reasoning shows that }
\phi(F^1_1(\lambda))+ \phi(F^0_2(\lambda)) + \phi(F^1_3(\lambda))  &=
a(\sigma(\lambda), \sigma(h_2)\sigma(g_2)\sigma(f_2)).
\end{align*}
In particular, $\phi$ satisfies~\eqref{eq:cubical 2-cocycle}, so it is a cubical $2$-cocycle.

Next, we claim that the cohomology class $[\phi]$ is independent of the choices made (on the
skeleton). Let $\tilde{\sigma}$ be another section of $q$ constructed as above. We show that the
resulting cubical $2$-cocycle is cohomologous to $\phi$. Let $b(e) = a(\tilde{\sigma}(e),
\sigma(e))$ for each $e \in Q_1(\Lambda) = E^1_\Lambda$.  Then the cubical $2$-cocycle
$\tilde{\phi}$ built from $\tilde{\sigma}$ is defined so that if $d(\lambda) = e_i + e_j$ with $i <
j$, and if $\lambda = fg = g'f'$ where $f,f' \in \Lambda^{e_i}$ and $g,g' \in \Lambda^{e_j}$, then
\[
\tilde\phi(\lambda) = a(\tilde{\sigma}(f)\tilde{\sigma}(g), \tilde{\sigma}(g')\tilde{\sigma}(f')).
\]
A routine computation then shows
\[
\phi(\lambda)  - \tilde\phi(\lambda) = b(g') + b(f') - b(f) - b(g) = (\dcub1 b)(\lambda).
\]
Hence $[\phi] = [\tilde\phi]$.

We now claim that the cocycle $c_\sigma$ obtained from the section $\sigma$ via
Theorem~\ref{thm:H2=Ext} is equal to $-c_\phi$. To see this, fix $(\mu,\nu) \in \Lambda^{*2}$. By
definition, $c_\phi(\mu,\nu) = S_\phi(\overline{\mu}\,\overline{\nu})$. Fix a sequence of allowable
transitions from $\overline{\mu}\,\overline{\nu}$ to $\overline{\mu\nu}$; that is, a path
$\alpha_1\cdots\alpha_n$ in $F_\Lambda$ with $r(\alpha_1) = \overline{\mu\nu}$ and $s(\alpha_n) =
\overline{\mu}\,\overline{\nu}$. Fix $i \le n$ and let $u = r(\alpha_i)$ and $w = s(\alpha_i)$.
Recall from Notation~\ref{ntn:phitilde} the definition of $\tilde\phi : F^1_\Lambda \to A$. For
each $i$, we have $\tilde{\phi}(\alpha_i) = \phi(w_{p(\alpha_i)} w_{p(\alpha_i)+1}) =
a(u_{p(\alpha_i)} u_{p(\alpha_i)+1}, w_{p(\alpha_i)} w_{p(\alpha_i)+1})$. Let $\ell := |\mu\nu|$.
Using the first assertion of Lemma~\ref{lem:a(x,y)=a(xz,yz)}, we see that
\[
\tilde{\phi}(\alpha_i) =  a\big(\sigma(u_1)\sigma(u_2)\cdots\sigma(u_\ell),
\sigma(w_1)\sigma(w_2)\cdots\sigma(w_\ell)\big).
\]
Hence, by definition of $S_\phi$, we have
\[
S_\phi(\overline{\mu}\,\overline{\nu})  = \sum^n_{i=1}
a\big(\sigma(r(\alpha_i)_1)\cdots\sigma(r(\alpha_i)_\ell),
\sigma(s(\alpha_i)_1)\cdots\sigma(s(\alpha_i)_\ell)\big).
\]
Now using the second assertion of
Lemma~\ref{lem:a(x,y)=a(xz,yz)}, we deduce that
\begin{align*}
S_\phi(\overline{\mu}\,\overline{\nu})
    &= a\big(\sigma((\overline{\mu\nu})_1\cdots\sigma(\overline{\mu\nu})_\ell,
        \sigma(\overline{\mu}_1)\cdots\sigma(\overline{\mu}_{|\mu|})
        \sigma(\overline{\nu}_1)\cdots\sigma(\overline{\nu}_{|\nu|})\big) \\
    &= a(\sigma(\mu\nu), \sigma(\mu)\sigma(\nu)).
\end{align*}
Hence $c_\phi(\mu,\nu)  = a(\sigma(\mu\nu), \sigma(\mu)\sigma(\nu))= -c_\sigma(\mu,\nu)$ as
claimed.

Now fix $c \in \Zcat2(\Lambda, A)$. By Theorem~\ref{thm:H2=Ext}, we have $[c] = [c_\sigma]$ for any
section $\sigma$ for $q : \extshrt{c}{\Lambda}{A} \to \Lambda$. By the preceding paragraphs, there
exists a section $\sigma$ for $q : \extshrt{c}{\Lambda}{A} \to \Lambda$, and a cubical $2$-cocycle
$\phi$ on $\Lambda$ such that $c_\sigma = -c_\phi$. In particular, we have $[c] = [c_\sigma] =
[-c_\phi]$, and it follows that the map $[\phi] \mapsto [c_\phi]$ is surjective from
$\Hcub2(\Lambda,A)$ to $\Hcat2(\Lambda,A)$.  Since the class $[\phi]$ does not depend on the choice
of section $\sigma$, the map is also injective.
\end{proof}

\section{Twisted \texorpdfstring{$k$}{k}-graph \texorpdfstring{$C^*$}{C*}-algebras}\label{sec:twisted
algebras}

In this section, unless otherwise noted, we restrict attention to row-finite $k$-graphs with no
sources and consider twisted $k$-graph $C^*$-algebras. We recall the definition of a twisted
$k$-graph $C^*$-algebra from \cite{kps3}, and then introduce the notion of a twisted Cuntz-Krieger
$(\Lambda, c)$-family associated to a categorical cocycle $c \in \Zcat{2} ( \Lambda , \TT )$. We
show that given a cubical cocycle $\phi \in \Zcub{2} ( \Lambda , \TT )$, if $c_\phi \in \Zcat{2} (
\Lambda , \TT )$ is the 2-cocycle of Theorem~\ref{thm:2-cocycle map} then the twisted $C^*$-algebra
$C^*_\phi(\Lambda)$ introduced in \cite[Section~7]{kps3} is universal for twisted Cuntz-Krieger
$c_\phi$-families for $\Lambda$.

For the abelian group $\TT$ we break with our conventions earlier in the paper and write the group
operation multiplicatively, write $\overline{z}$ for the inverse of $z \in \TT$ and write $1$ for
the identity element.

For the following, recall from \cite{RSY2003} that a $k$-graph $\Lambda$ is said to be
\emph{locally convex} if, whenever $e_i, e_j$ are distinct generators of $\NN^k$ and $\mu \in
\Lambda^{e_i}$ and $\nu \in \Lambda^{e_j}$ satisfy $r(\mu) = r(\nu)$, both $s(\mu)\Lambda^{e_j}$
and $s(\nu) \Lambda^{e_i}$ are nonempty.

\begin{dfn}[{see \cite[Definitions 7.4, 7.5]{kps3}}]\label{def:twisted algebra}
Let $\Lambda$ be a row-finite locally convex $k$-graph and fix $\phi \in \Zcub2(\Lambda,\TT)$. A
Cuntz-Krieger $\phi$-representation of $\Lambda$ in a $C^*$-algebra $A$ is a set $\{p_v : v \in
\Lambda^0\} \subseteq A$ of mutually orthogonal projections and a set $\{s_\lambda : \lambda \in
\bigcup^k_{i=1} \Lambda^{e_i}\} \subseteq A$ satisfying
\begin{enumerate}
    \item\label{it:new CK1} for all $i \le k$ and $\lambda \in \Lambda^{e_i}$,
        $s_\lambda^*s_\lambda = p_{s(\lambda)}$;
    \item\label{it:new CK2} for all $1 \le i < j \le k$ and $\mu,\mu' \in \Lambda^{e_i}$,
        $\nu,\nu' \in \Lambda^{e_j}$ such that $\mu\nu = \nu'\mu'$,
    \[
        s_{\nu'} s_{\mu'} = \phi(\mu\nu)s_\mu s_\nu;\text{ and}
    \]
    \item\label{it:new CK3} for all $v \in \Lambda^0$ and all $i \in \{1, \dots, k\}$ such that
        $v\Lambda^{e_i} \not= \emptyset$,
    \[
        p_v =  \sum_{\lambda \in v\Lambda^{e_i}} s_\lambda s_\lambda^*.
    \]
\end{enumerate}
We write $C^*_\phi(\Lambda)$ for the universal $C^*$-algebra generated by a Cuntz-Krieger
$\phi$-rep\-re\-sen\-ta\-tion of~$\Lambda$.
\end{dfn}

The following is much closer to the usual definition of a Cuntz-Krieger $\Lambda$-family. Notice,
however, that we now restrict attention to $k$-graphs with no sources: that is, $v\Lambda^n \not=
\emptyset$ for all $v \in \Lambda^0$ and $n \in \NN^k$. Every $k$-graph with no sources is locally
convex. Versions of the following definition for row-finite locally convex $k$-graphs or for
finitely aligned $k$-graphs incorporating the ideas of \cite{RSY2003} or \cite{RSY2004} seem likely
to produce reasonable notions of twisted $k$-graph $C^*$-algebras but we do not pursue this level
of generality here.

\begin{dfn}\label{dfn:twisted family}
Let $(\Lambda,d)$ be a row-finite $k$-graph with no sources,
and fix $c \in \Zcat2(\Lambda,\TT)$. A Cuntz-Krieger
$(\Lambda,c)$-family in a $C^*$-algebra $B$ is a function $t :
\lambda \mapsto t_\lambda$ from $\Lambda$ to $B$ such that
\begin{itemize}
\item[(CK1)] $\{t_v : v \in \Lambda^0\}$ is a collection of
    mutually orthogonal projections;
\item[(CK2)] $t_\mu t_\nu = c(\mu,\nu)t_{\mu\nu}$ whenever
    $s(\mu) = r(\nu)$;
\item[(CK3)] $t^*_\lambda t_\lambda = t_{s(\lambda)}$ for all $\lambda \in \Lambda$; and
\item[(CK4)] $t_v = \sum_{\lambda \in v\Lambda^n} t_\lambda t^*_\lambda$ for all $v \in
    \Lambda^0$ and $n \in \NN^k$.
\end{itemize}
\end{dfn}

We first show that given a $\TT$-valued $2$-cocycle $\phi$, the universal $C^*$-algebra
$C^*_\phi(\Lambda)$ of Definition~\ref{def:twisted algebra} is universal for Cuntz-Krieger
$(\Lambda, c_\phi)$-families.

Recall from Notation~\ref{ntn:coloured graph} that for $\lambda \in \Lambda$, we write
$\overline{\lambda}$ for the path in $E_\Lambda$ corresponding to the factorisation $\lambda =
\overline{\lambda}_1 \cdots \overline{\lambda}_n$ of $\lambda$ in which edges of degree $e_1$
appear leftmost, then those of degree $e_2$ and so on.

\begin{prop}\label{prp:old=new}
Let $\Lambda$ be a row-finite $k$-graph with no sources, and let $\phi \in \Zcub{2}(\Lambda, \TT)$.
Let $c_\phi \in \Zcat2(\Lambda,\TT)$ be the categorical $2$-cocycle obtained from
Theorem~\ref{thm:2-cocycle map}. Let $\{p_v : v \in \Lambda^0\}$ and $\{s_{\lambda} : \lambda \in
\bigsqcup^k_{i=1} \Lambda^{e_i} \}$ be the universal generating Cuntz-Krieger $\phi$-representation
of $\Lambda$ in $C^*_\phi(\Lambda)$. For $v \in \Lambda^0$, let $t_v := p_v$ and for $\lambda \in
\Lambda \setminus \Lambda^0$, set $t_\lambda = s_{\overline{\lambda}_1}\cdots
s_{\overline{\lambda}_{|\lambda|}}$. Then $t : \lambda \mapsto t_\lambda$ constitutes a
Cuntz-Krieger $(\Lambda,c_\phi)$-family in $C^*_\phi(\Lambda)$. Moreover, this family is universal
in the sense that given any  Cuntz-Krieger $(\Lambda,c_\phi)$-family $\{t'_\lambda : \lambda \in
\Lambda\}$  in a $C^*$-algebra $B$, there is a homomorphism $\pi: C^*_\phi(\Lambda) \to B$ such
that $\pi(t_\lambda) = t'_\lambda$.
\end{prop}
\begin{proof}
Recall from Definition~\ref{dfn:transition graph} that $F_\Lambda$ denotes the transition graph
associated to $\Lambda$, and that for each $\lambda \in \Lambda$, the preferred factorisation
$\overline{\lambda}$ of $\lambda$ is the terminal vertex in the component $F_\lambda$ of
$F_\Lambda$ corresponding to $\lambda$. Since each $u \in F^0_\Lambda$ is a path $u = u_1 \cdots
u_n \in E_\Lambda$, we may define partial isometries $\{ \tau_u : u \in F^0_\Lambda, \ell(u) \ge
1\}$ by $\tau_u := s_{u_1} \cdots s_{u_n}$. Thus for $\lambda \in \Lambda$ with $d(\lambda) \ne 0$,
the definition of $t_\lambda$ given in the statement of the proposition can be restated as
$t_\lambda = \tau_{\overline{\lambda}}$. For $v \in \Lambda^0$, we define $t_v := p_v$.

Suppose that $(u,v)$ is an allowable transition in $\lambda$, say $u = u_1 \cdots u_{i-1}u_i
 \cdots u_n$ and $v = u_1 \cdots u_{i-2}v_{i-1}v_i u_{i+1} \cdots u_n$ with $d(u_i) =
d(v_{i-1}) = e_j$ and $d(u_{i-1}) = d(v_i) = e_l$ with $j > l$. Then relation~(\ref{it:new CK2}) of
Definition~\ref{def:twisted algebra} gives
\[
s_{u_1} \cdots s_{v_{i-1}} s_{v_i} \cdots s_{u_n}
    = \phi(u_{i-1} u_i) s_{u_1} \cdots s_{u_{i-1}} s_{u_i} \cdots s_{u_n}
\]
Hence, using the map $\tilde{\phi} : F^1_\Lambda \to \TT$ defined as in
Notation~\ref{ntn:phitilde}, we have
\[
\tau_v = \tilde{\phi}(u,v) \tau_u.
\]
So if $\alpha$ is a path in $F_\Lambda$, then $\tau_{s(\alpha)} = \tilde{\phi}(\alpha)
\tau_{r(\alpha)}$.

Fix $(\mu,\nu ) \in \Lambda^{*2}$ and $\alpha \in F_\Lambda$ with $r(\alpha) = \overline{\mu\nu}$
and $s(\alpha) = \overline{\mu}\,\overline{\nu}$. By the above
\[
\tilde{\phi}(\alpha) t_{\mu\nu}
    = \tilde{\phi}(\alpha) \tau_{\overline{\mu\nu}}
    = \tau_{\overline{\mu}\,\overline{\nu}}
    = \tau_{\overline{\mu}} \tau_{\overline{\nu}}
    = t_\mu t_\nu.
\]
Lemma~\ref{lem:coboundary on transition graph} implies that $\tilde{\phi}(\alpha) =
S_\phi(\overline{\mu}\,\overline{\nu})$, so, by the definition of $c_\phi$ in
Theorem~\ref{thm:2-cocycle map}, $c_\phi(\mu,\nu) t_{\mu\nu} = t_\mu t_\nu$. Thus $t : \Lambda \to
C^*_\phi(\Lambda)$ satisfies relation~(CK2) for the cocycle $c_\phi$. It trivially satisfies~(CK1),
and standard induction arguments establish (CK3)~and~(CK4). Hence $t$ is a Cuntz-Krieger
$(\Lambda,c_\phi)$-family.

Let  $\{t'_\lambda : \lambda \in \Lambda\}$  be a Cuntz-Krieger $(\Lambda,c_\phi)$-family in a
$C^*$-algebra $B$. We claim that $\{t'_v : v \in \Lambda^0\}$ and $\{t'_{\lambda} : \lambda \in
\bigsqcup^k_{i=1} \Lambda^{e_i} \}$ constitute a Cuntz-Krieger $\phi$-representation of $\Lambda$
in $B$. Relations (\ref{it:new CK1})~and~(\ref{it:new CK3}) are special cases of (CK3)~and~(CK4)
respectively. Let $\mu,\nu, \mu', \nu'$ be as in Definition~\ref{def:twisted algebra}(\ref{it:new
CK2}). By definition of $c_\phi$, we have
\[
    c_\phi(\nu',\mu') = \phi(\mu\nu)\quad\text{ and }\quad c_\phi(\mu,\nu) = 1.
\]
Hence
\[
t'_{\nu'} t'_{\mu'}
    = c_\phi(\nu', \mu') t'_{\nu'\mu'}
    = c_\phi(\nu',\mu') \overline{c_\phi(\mu,\nu)} t'_\mu t'_\nu
    = \phi(\mu\nu) t'_\mu t'_\nu.
\]
So the elements $\{t'_v : v \in \Lambda^0\}$ and $\{t'_\lambda : \lambda \in \bigsqcup^k_{i=1}
\Lambda^{e_i}\})$ satisfy the relations (\ref{it:new CK1})--(\ref{it:new CK3}).  By the universal
property of $C_\phi^* ( \Lambda )$ there is a homomorphism  $\pi: C^*_\phi(\Lambda) \to B$ such
that $\pi(t_\lambda) = t'_\lambda$ for $\lambda \in \Lambda^0 \cup \bigsqcup^k_{i=1}
\Lambda^{e_i}$. An induction shows that $\pi(t_\lambda) = t'_\lambda$ for all $\lambda \in
\Lambda$.
\end{proof}

Proposition~\ref{prp:old=new} shows that the twisted $C^*$-algebras associated to $\TT$-valued
$2$-cocycles in \cite{kps3} can be regarded as twisted $C^*$-algebras associated to the
corresponding categorical cocycles. But, while every categorical $2$-cocycle $c$ is cohomologous to
$c_\phi$ for some $\phi \in \Zcub{2}(\Lambda, \TT)$, it is not clear that every categorical
$2$-cocycle is equal to $c_\phi$ for some $\phi$.

\begin{ntn}\label{rmk:old twisted algebra}
Let $(\Lambda,d)$ be a row-finite $k$-graph with no sources, and fix $c \in \Zcat2(\Lambda,\TT)$.
Relations (CK1)~and~(CK3) imply that the images of elements of a Cuntz-Krieger $(\Lambda,c)$-family
under any $*$-homomorphism are partial isometries and hence have norm $1$ (or $0$). A standard
argument (see, for example, \cite[Propositions 1.20~and~1.21]{Raeburn:Graphalgebras05}) then shows
that there is a $C^*$-algebra $C^*(\Lambda,c)$ generated by a Cuntz-Krieger $(\Lambda,c)$-family $s
: \Lambda \to C^* (\Lambda ,c)$ which is universal in the sense that given any other Cuntz-Krieger
$(\Lambda,c)$-family $t$, there is a homomorphism $\pi_t : C^*(\Lambda,c) \to C^*(t)$ such that
$\pi_t \circ s = t$. This universal property determines $C^*(\Lambda,c)$ up to canonical
isomorphism.
\end{ntn}

The following remark reconciles the use of $s$ to denote the universal family in $C^*(\Lambda,
c_\phi)$ with the use of the same symbol to denote the universal family in $C^*_\phi(\Lambda)$.

\begin{rmk}\label{rmk:same old}
Let $\Lambda$ be a row-finite $k$-graph with no sources, and fix $\phi \in \Zcub{2}(\Lambda, \TT)$.
Let $c_\phi \in \Zcat2(\Lambda,\TT)$ be the corresponding categorical $2$-cocycle.
Proposition~\ref{prp:old=new} and that $C^*(\Lambda , c)$ is determined by its universal property
imply that there is an isomorphism $C^*(\Lambda, c_\phi) \to C^*_\phi(\Lambda)$ which carries each
generator of $C^*(\Lambda, c_\phi)$ associated to a vertex or an edge to the corresponding
generator of $C^*_\phi(\Lambda)$. We will henceforth identify $C^*(\Lambda, c_\phi)$ and
$C^*_\phi(\Lambda)$ via this isomorphism without comment.
\end{rmk}

\begin{prop}\label{prp:cohomologous isomorphic}
Let $\Lambda$ be a row-finite $k$-graph with no sources, fix $c_1, c_2 \in \Zcat2(\Lambda, \TT)$
and suppose that $c_1$ and $c_2$ are cohomologous, say $b \in \Ccat1(\Lambda,\TT)$ satisfies  $c_1 = (\dcat1b) c_2$. Denote by $s^{c_i}$ the universal Cuntz-Krieger $(\Lambda, c_i)$-family in
$C^*(\Lambda,c_i)$ for $i = 1,2$. Then there is an isomorphism $\pi : C^*(\Lambda, c_1) \to
C^*(\Lambda, c_2)$ satisfying $\pi(s^{c_1}_\lambda) = b(\lambda) s^{c_2}_\lambda$ for all $\lambda
\in \Lambda$. In particular, if $c \in \Bcat2(\Lambda,\TT)$, then $C^*(\Lambda,c) \cong
C^*(\Lambda)$.
\end{prop}
\begin{proof}
For $(\mu,\nu ) \in \Lambda^{*2}$ we have
$c_1(\mu,\nu) = b(\mu)b(\nu)\overline{b(\mu\nu)}c_2(\mu,\nu)$,
and hence
\[
    b(\mu)b(\nu)c_2(\mu,\nu) = b(\mu\nu)c_1(\mu,\nu).
\]
Hence relation~(CK2) gives
\[
b(\mu) s^{c_2}_\mu b(\nu) s^{c_2}_\nu
    = b(\mu)b(\nu) \big( c_2(\mu,\nu) s^{c_2}_{\mu\nu} \big)
    = c_1(\mu,\nu) \big(b(\mu\nu) s^{c_2}_{\mu\nu}\big).
\]
Since relations (CK1), (CK3)~and~(CK4) do not depend on the cocycle $c_1$, it follows that the
function $t : \Lambda \to C^* ( \Lambda , c_2 )$ defined by  $\lambda \mapsto
b(\lambda)s^{c_2}_\lambda$ is a Cuntz-Krieger $(\Lambda,c_1)$-family, so the universal property of
$C^*(\Lambda, c_1)$ yields a homomorphism $\pi : C^*(\Lambda, c_1) \to C^*(\Lambda,c_2)$ satisfying
$\pi(s^{c_1}_\lambda) = b(\lambda) s^{c_2}_\lambda$ for all $\lambda \in \Lambda$. The symmetric
argument yields a homomorphism $\phi : C^*(\Lambda, c_2) \to C^*(\Lambda,c_1)$ which is an inverse
for $\pi$ on generators. Hence $\pi$ is an isomorphism as claimed.

For the final assertion, note that if $c$ is a coboundary, then
it is cohomologous to the trivial cocycle $1 \in
\Zcat2(\Lambda,\TT)$ given by $1(\mu,\nu) = 1$ for all
$(\mu,\nu) \in \Lambda^{*2}$. Since $C^*(\Lambda,1)$ is universal for the same
relations as the $k$-graph $C^*$-algebra $C^*(\Lambda)$ of
\cite{KP2000}, we have $C^*(\Lambda,1) \cong C^*(\Lambda)$. Hence
$C^*(\Lambda,c) \cong C^*(\Lambda)$ by the preceding paragraph.
\end{proof}
By Theorem \ref{thm:2-cohomology iso}, the map $[\phi] \mapsto [c_\phi]$ is an
isomorphism $\Hcub2(\Lambda,A) \cong \Hcat2(\Lambda,A)$. Combining Proposition \ref{prp:old=new}
and Proposition \ref{prp:cohomologous isomorphic} we obtain the following.
\begin{cor}
Let $\Lambda$ be a row-finite $k$-graph with no sources, and fix $c \in \Zcat2(\Lambda,\TT)$. Let
$\phi \in \Zcub2(\Lambda,\TT)$ be a $2$-cocycle such that $c_\phi$ is cohomologous to $c$.
Then $C^*(\Lambda,c)  \cong C^*_\phi(\Lambda)$.
\end{cor}

The preceding corollary gives another proof that if $\phi, \psi \in \Zcub{2}(\Lambda, \TT)$ are
cohomologous, then $C^*_\psi(\Lambda) \cong C^*_\phi(\Lambda)$ (see \cite[Proposition~7.6]{kps3}).

\section{Twisted groupoid \texorpdfstring{$C^*$}{C*}-algebras}

Let $\Lambda$ be a row-finite $k$-graph with no sources. Let $\Gg_\Lambda$ be the $k$-graph
groupoid of \cite{KP2000} (see Definition~\ref{dfn:path-groupoid} below) and let
$\Hgpd2(\Gg_\Lambda, \cdot )$ denote the continuous cocycle cohomology used in \cite{Renault1980}
(see Remark~\ref{rmk:groupoid cohomology}). Given a  categorical $2$-cocycle $c$  on $\Lambda$ we
construct a $2$-cocycle $\sigma_c$ on the groupoid $\Gg_\Lambda$. Given a locally compact abelian
group $A$ we show that $c \mapsto \sigma_c$ determines a homomorphism $\Hcat2(\Lambda, A) \to
\Hgpd2(\Gg_\Lambda, A)$. If $c$ is $\TT$-valued, we show that there is a canonical homomorphism
from the twisted $k$-graph $C^*$-algebra $C^*(\Lambda,c)$ to Renault's twisted groupoid algebra
$C^*(\Gg_\Lambda, \sigma_c)$ (see \cite{Renault1980}). We show in \S \ref{sec:giut} that this
homomorphism is an isomorphism.

We denote by $\Lambda \mathbin{{_s*_s}} \Lambda$ the set $\{(\mu,\nu) \in \Lambda \times \Lambda :
s(\mu) = s(\nu)\}$. Recall the definition of $\Lambda^\infty$ given in \cite[Definition
2.1]{KP2000}: we write $\Omega_k$ for the $k$-graph $\{(m,n) \in \NN^k : m \le n\}$ with $r(m,n) =
m$, $s(m,n) = n$, $(m,n)(n,p) = (m,p)$ and $d(m,n) = n - m$, and we define $\Lambda^\infty$ to be
the collection of all $k$-graph morphisms $x : \Omega_k \to \Lambda$. For $p \in \NN^k$, we define
$\sigma^p : \Lambda^\infty \to \Lambda^\infty$ by $(\sigma^p x)(m,n) := x(m+p, n+p)$ for all $(m,n)
\in \Omega_k$. For $x \in \Lambda^\infty$ we denote $x(0)$ by $r(x)$.

\begin{dfn}[{\cite[Definition~2.7]{KP2000}}]\label{dfn:path-groupoid}
Let $\Lambda$ be a row-finite $k$-graph with no sources. Let
\[
\Gg_\Lambda :=
    \{ (x, \ell - m, y) \in \Lambda^\infty \times \ZZ^k \times \Lambda^\infty : \ell, m \in \NN^k, \sigma^\ell x = \sigma^m y\}.
\]
For $\mu, \nu \in \Lambda$ with $s(\mu) = s(\nu)$ define
$Z(\mu, \nu) \subset \Gg_\Lambda$ by
\[
Z(\mu, \nu) := \{ (\mu x, d(\mu) - d(\nu), \nu x) : x \in \Lambda^\infty, r(x) = s(\mu)\}.
\]
For $\lambda \in \Lambda$, we define $Z(\lambda) := Z(\lambda, \lambda)$.
\end{dfn}

The sets $Z(\mu, \nu)$ form a basis of compact open sets for a locally compact Hausdorff topology
on $\Gg_\Lambda$ under which it is an \'{e}tale groupoid \noindent with structure maps $r
(x,\ell-m,y) = (x,0,x)$, $s ( x , \ell-m,y)= (y,0,y)$, and $(x,\ell-m,y) (y,p-q,z) = (x ,
\ell-m+p-q,z)$. (see \cite[Proposition 2.8]{KP2000}). The $Z(\lambda)$ are then a basis for the
relative topology on $\Gg_\Lambda^{(0)}$. We shall identify $\Gg_\Lambda^{(0)} = \{ (x,0,x) : x \in
\Lambda^\infty \}$ with $\Lambda^\infty$.

\begin{ntn}\label{ntn:dtilde}
We write $\tilde{d}$ for the continuous $\ZZ^k$-valued $1$-cocycle on $\Gg_\Lambda$ induced by the
degree map on $\Lambda$. That is, $\tilde{d} (x,m,y) = m$.
\end{ntn}

Our next two results show how to use an appropriate partition of $\Gg_\Lambda$ to construct a
continuous $2$-cocycle on $\Gg_\Lambda$ (see Remark~\ref{rmk:groupoid cohomology}) from a
categorical $2$-cocycle on $\Lambda$.

\begin{lem}\label{lem:groupoid-2-cocycle}
Let $\Lambda$ be a row-finite $k$-graph with no sources, let $A$ be an abelian group and let $c \in
\Zcat2(\Lambda,A)$. Suppose that $\Pp$ is a subset of $\Lambda \mathbin{{_s*_s}} \Lambda$
such that $\{Z(\mu,\nu) : (\mu,\nu) \in \Pp\}$ is a partition of $\Gg_\Lambda$. For each $g \in
\Gg_\Lambda$, let $(\mu_g,\nu_g)$ be the unique element of $\Pp$ such that $g \in Z(\mu_g, \nu_g)$.
\begin{enumerate}\renewcommand{\theenumi}{\roman{enumi}}
\item\label{it:alphas} For each $(g, h) \in \Gg_\Lambda^{(2)}$, there exist $\alpha \in
    s(\mu_g)\Lambda$, $\beta \in s(\mu_h)\Lambda$, $\gamma \in s(\mu_{gh})\Lambda$ and $y \in
    \Lambda^\infty$ such that all of the following hold: $r(y) = s(\alpha) =  s(\beta) = s(\gamma)$; $\mu_g\alpha = \mu_{gh}\gamma$, $\nu_{h}\beta = \nu_{gh}\gamma$ and $\nu_{g}\alpha = \mu_{h}\beta$; and
    \begin{equation}\label{eq:alphbetgam}
    \begin{split}
        g = (\mu_g\alpha y, d ( \mu_g ) &{}- d ( \nu_g ) , \nu_g\alpha y),\quad
        h = (\mu_h\beta y, d( \mu_h) - d ( \nu_h ) , \nu_h\beta y),\quad\text{and} \\
        &gh = (\mu_{gh}\gamma y, d(\mu_{gh}) - d(\nu_{gh}),
        \nu_{gh}\gamma y).
    \end{split}
    \end{equation}
\item\label{it:G-cocycle formula} Fix $(g,h) \in \Gg_\Lambda^{(2)}$ and $\alpha,\beta,\gamma$
    satisfying~\eqref{eq:alphbetgam}. The formula
    \begin{equation}\label{eq:groupoid-cocycle-formula}
    \big(c(\mu_{g}, \alpha) - c(\nu_{g}, \alpha)\big) + \big(c(\mu_{h}, \beta) - c(\nu_{h}, \beta)\big)
        - \big(c(\mu_{gh}, \gamma) - c(\nu_{gh}, \gamma)\big)
    \end{equation}
    does not depend on the choice of $\alpha$, $\beta$, $\gamma$.
\item\label{it:gpd cocycle} For $(g,h) \in \Gg_\Lambda^{(2)}$, define $\sigma_c(g,h)$ to be the
    value of~\ref{eq:groupoid-cocycle-formula} for any choice of $\alpha, \beta, \gamma$
    satisfying~\eqref{eq:alphbetgam}. Then $\sigma_c$ is a continuous groupoid $2$-cocycle.
\end{enumerate}
\end{lem}
\begin{proof}
Recall that $\tilde{d} : \Gg_\Lambda \to \ZZ^k$ is given by $\tilde{d}(x,m,y) = m$. Let $N
:= d ( \mu_g ) \vee d ( \mu_{gh} ) \vee ( d ( \mu_h ) - \tilde{d} (g) )$. Then routine calculations
show that $\alpha := r ( g) ( d ( \mu_g ) , N )$, $\gamma := r (gh) ( d ( \mu_{gh} ) , N )$, $\beta
:= r(h) ( d ( \mu_h ) , N + \tilde{d} (g) )$ and $y = \sigma^N ( r(g) )$ have the desired
properties.

Fix $\alpha$, $\beta$, $\gamma$ and $\alpha'$, $\beta'$, $\gamma'$
satisfying~\eqref{eq:alphbetgam}. Let
\[
    \delta := r(g)(d(\mu_g\alpha), d(\mu_g\alpha) \vee d(\mu_g\alpha'))
    \quad\text{ and }\quad
    \varepsilon := r(g)(d(\mu_g\alpha'), d(\mu_g\alpha) \vee d(\mu_g\alpha'))
\]
Then $\alpha\delta = \alpha'\varepsilon$, $\beta\delta = \beta'\varepsilon$ and $\gamma\delta =
\gamma'\varepsilon$ satisfy~\eqref{eq:alphbetgam}. So by symmetry, it suffices to show that
replacing $\alpha$ with $\alpha\delta$, $\beta$ with $\beta\delta$ and $\gamma$ with $\gamma\delta$
in~\eqref{eq:groupoid-cocycle-formula} yields the same value. Since $c$ is a categorical
$2$-cocycle,
\begin{align*}
c(\mu_g, \alpha \delta) - c(\nu_{g}, \alpha \delta)
    &= c(\mu_g\alpha , \delta) - c(\nu_{g}\alpha , \delta) + c(\mu_g, \alpha ) - c(\nu_{g}, \alpha ), \\
c(\mu_h, \beta\delta) - c(\nu_{h}, \beta\delta)
    &= c(\mu_h\beta, \delta) - c(\nu_{h}\beta, \delta) + c(\mu_h, \beta) - c(\nu_{h}, \beta),\quad\text{ and} \\
c(\mu_{gh}, \gamma\delta) - c(\nu_{gh}, \gamma\delta)
    &= c(\mu_{gh}\gamma, \delta) - c(\nu_{gh}\gamma, \delta) + c(\mu_{gh}, \gamma) - c(\nu_{gh}, \gamma).
\end{align*}
Hence
\begin{align*}
\big(c(\mu_{g}, \alpha \delta)&{} - c(\nu_{g}, \alpha \delta)\big)  +
        \big(c(\mu_{h}, \beta \delta) - c(\nu_{h}, \beta \delta)\big)  -
        \big(c(\mu_{gh}, \gamma \delta) - c(\nu_{gh}, \gamma \delta)\big) \\
    &= c(\mu_{g}\alpha , \delta) - c(\nu_{g}\alpha , \delta) + c(\mu_{g}, \alpha ) - c(\nu_{g}, \alpha ) \\
    &\qquad {} + c(\mu_{h}\beta , \delta) - c(\nu_{h}\beta , \delta) + c(\mu_{h}, \beta ) - c(\nu_{h}, \beta ) \\
    &\qquad\qquad {} - c(\mu_{gh}\gamma , \delta) + c(\nu_{gh}\gamma , \delta) - c(\mu_{gh}, \gamma ) + c(\nu_{gh}, \gamma).
\end{align*}
Since $\mu_{g}\alpha = \mu_{gh}\gamma $, $\nu_{h}\beta  = \nu_{gh}\gamma $ and $\nu_{g}\alpha =
\mu_{h}\beta$, we obtain
\begin{align*}
\big(c(\mu_{g}, \alpha \delta)&{} - c(\nu_{g}, \alpha \delta)\big)
        + \big(c(\mu_{h}, \beta \delta) - c(\nu_{h}, \beta \delta)\big)
        - \big(c(\mu_{gh}, \gamma \delta) - c(\nu_{gh}, \gamma \delta)\big) \\
    &= \big(c(\mu_{g}, \alpha ) - c(\nu_{g}, \alpha )\big) + \big(c(\mu_{h}, \beta) - c(\nu_{h}, \beta)\big)
        - \big( c(\mu_{gh}, \gamma ) - c(\nu_{gh}, \gamma )\big).
\end{align*}
So~\eqref{eq:groupoid-cocycle-formula} does not depend on the choice of $\alpha, \beta, \gamma$.

To prove that $\sigma_c$ is a $2$-cocycle, fix $(g_1, g_2, g_3) \in \Gg_\Lambda^{(3)}$. Let
$(\mu_i, \nu_i) = (\mu_{g_i}, \nu_{g_i})$ for $i = 1,2,3$. Let $(\mu_{ij},\nu_{ij}) =
(\mu_{g_ig_j}, \nu_{g_ig_j})$ for $ij = 12, 23$ and let $(\mu_{123}, \nu_{123}) = (\mu_{g_1g_2g_3},
\nu_{g_1g_2g_3})$. Fix $z \in \Lambda^\infty$, and for each symbol $\star \in \{1, 2, 3, 12, 23,
123\}$, fix $\alpha_\star \in \Lambda$ such that
\[
    g_\star = (\mu_\star\alpha_\star z, d(\mu_\star) - d(\nu_\star), \nu_\star\alpha_\star z).
\]
Then~\eqref{eq:groupoid-cocycle-formula} yields
\begin{align*}
\sigma_c(g_1, g_2) &{}+ \sigma_c(g_1g_2, g_3)\\
    &= \Big(\big(c(\mu_{1}, \alpha_1) - c(\nu_{1}, \alpha_1)\big) + \big(c(\mu_{2}, \alpha_2) - c(\nu_{2}, \alpha_2)\big)\\
    &\qquad{} - \big( c(\mu_{{12}}, \alpha_{12}) - c(\nu_{{12}}, \alpha_{12})\big)\Big) + \Big(\big(c(\mu_{{12}}, \alpha_{12}) - c(\nu_{{12}}, \alpha_{12})\big) \\
    &\qquad{} + \big(c(\mu_{3}, \alpha_3) - c(\nu_{3}, \alpha_3)\big) - \big( c(\mu_{{123}}, \alpha_{123}) - c(\nu_{{123}}, \alpha_{123})\big)\Big)
\end{align*}
and
\begin{align*}
\sigma_c(g_2, g_3) &{}+ \sigma_c(g_1, g_2 g_3)\\
    &= \Big(\big(c(\mu_{2}, \alpha_2) - c(\nu_{2}, \alpha_2)\big) + \big(c(\mu_{3}, \alpha_3) - c(\nu_{3}, \alpha_3)\big)\\
    & \qquad{} - \big( c(\mu_{23}, \alpha_{23}) - c(\nu_{23}, \alpha_{23})\big)\Big) + \Big(\big(c(\mu_{1}, \alpha_1) - c(\nu_{1}, \alpha_1)\big)\\
    & \qquad{} + \big(c(\mu_{{23}}, \alpha_{23}) - c(\nu_{{23}}, \alpha_{23})\big) -
        \big( c(\mu_{{123}}, \alpha_{123}) - c(\nu_{{123}}, \alpha_{123})\big)\Big).
\end{align*}
Cancelling and comparing gives
\[
\sigma_c(g_1, g_2) + \sigma_c(g_1g_2, g_3) =
\sigma_c(g_2, g_3) + \sigma_c(g_1, g_2 g_3).
\]
Hence, $\sigma_c$ satisfies the groupoid $2$-cocycle identity. A straightforward calculation shows
that $c(g, h) = 0$ if either $g$ or $h$ is a unit. Since  $\sigma_c$ is locally constant, it is
continuous, so $\sigma_c \in \Zgpd2(\Gg_\Lambda, A)$.
\end{proof}

\begin{rmk}
The cocycle $\sigma_c$ constructed in Lemma~\ref{lem:groupoid-2-cocycle}(\ref{it:gpd cocycle})
depends both on $c$ and on the collection $\Pp$.
\end{rmk}

\begin{thm}\label{thm:graph->gpd cohomology}
Let $\Lambda$ be a row-finite $k$-graph with no sources and let $A$ be an abelian group. Suppose
that $\Pp$ is a subset of $\Lambda \mathbin{{_s*_s}} \Lambda$ such that $\{Z(\mu,\nu) :
(\mu,\nu) \in \Pp\}$ is a partition of $\Gg_\Lambda$. Fix $c \in \Zcat2(\Lambda,A)$, and let
$\sigma_c \in \Zgpd2(\Gg_\Lambda, A)$ be the continuous cocycle of
Lemma~\ref{lem:groupoid-2-cocycle}. The cohomology class $[\sigma_c]$ is independent of the choice
of $\Pp$ and depends only on the cohomology class of $c$. Moreover, $[c] \mapsto [\sigma_c]$ is a
homomorphism $\Hcat2(\Lambda,A) \to \Hgpd2(\Gg_\Lambda, A)$.
\end{thm}
\begin{proof}
Suppose that $c$ is a categorical $2$-coboundary on $\Lambda$. We show that  $\sigma_c$ is a
groupoid $2$-coboundary on $\Gg_\Lambda$.  Since $c$ is a categorical $2$-coboundary, there is a
cochain $b \in \Ccat{1}(\Lambda, A)$, such that $c = \dcat1b$.  Hence, for $(\lambda_1, \lambda_2)
\in \Lambda^{*2}$ we have
\[
c(\lambda_1, \lambda_2)  =  (\dcat1b)(\lambda_1, \lambda_2)
                        = b(\lambda_1) - b(\lambda_1\lambda_2) + b(\lambda_2).
\]
Define $a : \Gg_\Lambda \to A$ by $a(g) = b(\mu_g) - b(\nu_g)$. Then $a$ is continuous because it
is locally constant. For $g \in \Gg_\Lambda^{(0)}$ we have $\mu_g = \nu_g$, so $a(g) = 0$. Hence $a
\in \Cgpd1(\Gg_\Lambda, A)$.

Fix $(g,h)\in \Gg^{(2)}_\Lambda$. With notation as in
Lemma~\ref{lem:groupoid-2-cocycle}(\ref{it:alphas}), equation~\eqref{eq:groupoid-cocycle-formula}
gives
\begin{align*}
\sigma_c(g, h) &= \big(c(\mu_{g}, \alpha) - c(\nu_{g}, \alpha)\big) +
\big(c(\mu_{h}, \beta) - c(\nu_{h}, \beta)\big) -
\big(c(\mu_{gh}, \gamma) - c(\nu_{gh}, \gamma)\big)\\
&=  (b(\mu_{g}) -  b(\mu_{g}\alpha) + b(\alpha)) -  (b(\nu_{g}) -  b(\nu_{g}\alpha) + b(\alpha))  \\
&\qquad{}+ \big(b(\mu_{h}) -  b(\mu_{h}\beta) + b(\beta)\big) -  \big(b(\nu_{h}) -  b(\nu_{h}\beta) + b(\beta)\big) \\
&\qquad{} - \big(b(\mu_{gh}) -  b(\mu_{gh}\gamma) + b(\gamma)\big) +
 \big(b(\nu_{gh}) -  b(\nu_{gh}\gamma) + b(\gamma)\big).
\intertext{Since $\mu_{g}\alpha= \mu_{gh}\gamma$, $\nu_{h}\beta = \nu_{gh}\gamma$ and
    $\nu_{g}\alpha= \mu_{h}\beta$, we obtain}
\sigma_c(g,h) &= b(\mu_{g}) -  b(\nu_{g})  + b(\mu_{h}) -  b(\nu_{h})  - b(\mu_{gh}) +  b(\nu_{gh}) \\
&= a(g) + a(h) - a(gh) = \dgpd1(a)(g, h).
\end{align*}
Hence, $\sigma_c = \dgpd1(a)$ is a coboundary.
Since the map $c \mapsto \sigma_c$ is a homomorphism
(see formula ~\eqref{eq:groupoid-cocycle-formula}) which maps coboundaries to coboundaries,
the map $[c] \mapsto [\sigma_c]$ is a well-defined homomorphism
$\Hcat2(\Lambda,A) \to \Hgpd2(\Gg_\Lambda, A)$.

It remains to verify that $[\sigma_c]$ does not depend on the choice of $\Pp$. Fix subsets $\Pp$ and $\Qq$ of $\Lambda \mathbin{{_s*_s}} \Lambda$ yielding partitions of
$\Gg_\Lambda$. For $(\mu,\nu), (\sigma,\tau) \in \Lambda \mathbin{{_s*_s}} \Lambda$, if $d(\mu) -
d(\nu) \not= d(\sigma) - d(\tau)$, then $Z(\mu,\nu) \cap Z(\sigma,\tau) = \emptyset$. Otherwise,
setting
\[
(\mu,\nu) \wedge (\sigma,\tau) := \{(\mu\alpha,\nu\alpha) : \mu\alpha \in \MCE(\mu,\sigma)\text{ and } \nu\alpha \in \MCE(\nu,\tau)\},
\]
we have
\begin{equation}\label{eq:wedge implements intersection}
    Z(\mu,\nu) \cap Z(\sigma,\tau) =
        \bigsqcup_{(\eta,\zeta) \in (\mu,\nu) \wedge (\sigma,\tau)} Z(\eta,\zeta).
\end{equation}
Let $\Pp \vee \Qq := \bigcup_{(\mu, \nu) \in \Pp} \bigcup_{(\sigma,\tau) \in \Qq} (\mu,\nu) \wedge
(\sigma, \tau)$. Then $\{Z(\eta,\zeta) : (\eta,\zeta) \in \Pp \vee \Qq\}$ is a common refinement of
$\{Z(\mu,\nu) : (\mu,\nu) \in \Pp\}$ and $\{Z(\sigma,\tau) : (\sigma,\tau) \in \Qq\}$ such that if
$g \in \Gg_\Lambda$ satisfies $g \in Z(\mu,\nu)$ for $(\mu, \nu) \in \Pp$ and $g \in Z(\eta,\zeta)$
for $(\eta, \zeta) \in (\Pp \vee \Qq)$, then $\eta = \mu\lambda$ and $\zeta = \nu\lambda$ for some
$\lambda$, and similarly for $\Qq$. So by replacing $\Qq$ with $\Pp \vee \Qq$ we may assume that
$\{Z(\eta,\zeta) : (\eta,\zeta) \in \Qq\}$ is a refinement of $\{Z(\mu,\nu) : (\mu,\nu) \in \Pp\}$
and that for each element $(\eta,\zeta)$ of $\Qq$, there is a unique element $(\mu,\nu) \in \Pp$
and a unique $\lambda \in \Lambda$ such that $\eta = \mu\lambda$ and $\zeta = \nu\lambda$.

For $g \in \Gg_\Lambda$, let $(\mu_g, \nu_g) \in \Pp$ and $(\eta_g, \zeta_g) \in \Qq$ be the unique
elements such that $g \in Z(\eta_g, \zeta_g) \subseteq Z(\mu_g, \nu_g)$ and let $\lambda_g$ be the
unique path such that $(\eta_g, \zeta_g) = (\mu_g\lambda_g, \nu_g\lambda_g)$.

Fix $(g, h) \in \Gg_\Lambda^{(2)}$. By Lemma~\ref{lem:groupoid-2-cocycle}(\ref{it:alphas}), we may
fix $\alpha', \beta', \gamma'$ and $y$ satisfying
\[
\begin{split}
    g = (\eta_g\alpha' y, d(\eta_g ) &{}- d(\zeta_g ) ,\zeta_g\alpha' y),\quad
    h = (\eta_h\beta' y, d(\eta_h) - d (\zeta_h ) , \zeta_h\beta' y),\quad\text{and} \\
    &gh = (\eta_{gh}\gamma' y, d(\eta_{gh}) - d(\zeta_{gh}), \zeta_{gh}\gamma' y).
\end{split}
\]
The triple $\alpha = \lambda_g\alpha'$, $\beta = \lambda_h\beta'$, $\gamma = \lambda_{gh}\gamma'$
then satisfies~\eqref{eq:alphbetgam}. So
\begin{align*}
    g &\in Z(\eta_g\alpha', \zeta_g\alpha') = Z(\mu_g\alpha, \nu_g\alpha)\\
    h &\in Z(\eta_h\beta', \zeta_h\beta') = Z(\mu_h\beta, \nu_h\beta),\text{ and}\\
    gh &\in Z(\eta_{gh}\gamma', \zeta_{gh}\gamma') =  Z(\mu_{gh}\gamma, \nu_{gh}\gamma).
\end{align*}

Fix $c \in \Zcat2(\Lambda, A)$. Let $\sigma^\Pp_c$ be the groupoid $2$-cocycle obtained 
from
Lemma~\ref{lem:groupoid-2-cocycle} applied to $c$ and $\Pp$ and let $\sigma^\Qq_c$ be the
groupoid $2$-cocycle obtained in the same way from $c$ and $\Qq$. By
Lemma~\ref{lem:groupoid-2-cocycle}(\ref{it:G-cocycle formula}), we have
\begin{align*}
\sigma^\Pp_c(g, h)
    &=   \big(c(\mu_g,\alpha) - c(\nu_g,\alpha)\big)
                     + \big(c(\mu_h,\beta) - c(\nu_h,\beta)\big)
                     - \big(c(\mu_{gh},\gamma) - c(\nu_{gh},\gamma)\big),\\
    &=   \big(c(\mu_g,\lambda_g\alpha') - c(\nu_g,\lambda_g\alpha')\big)
                     + \big(c(\mu_h,\lambda_h\beta') - c(\nu_h,\lambda_h\beta')\big) \\
                     & \hskip15em{} - \big(c(\mu_{gh},\lambda_{gh}\gamma') - c(\nu_{gh},\lambda_{gh}\gamma')\big),
                     \text{ \;and}\\
\sigma^\Qq_c(g, h)
    &=  \big(c(\eta_g,\alpha') - c(\zeta_g,\alpha')\big)
                   + \big(c(\eta_h,\beta') - c(\zeta_h,\beta')\big)
                   - \big(c(\eta_{gh}, \gamma') - c(\zeta_{gh},\gamma')\big) \\
    &=  \big(c(\mu_g\lambda_g,\alpha') - c(\nu_g\lambda_g,\alpha')\big)
                   + \big(c(\mu_h\lambda_h,\beta') - c(\nu_h\lambda_h,\beta')\big) \\
                   & \hskip15em{} - \big(c(\mu_{gh}\lambda_{gh}, \gamma') - c(\nu_{gh}\lambda_{gh},\gamma')\big).
\end{align*}
Define $b : \Gg_\Lambda \to A$ by $b(g) = c(\mu_g, \lambda_g) - c(\nu_g, \lambda_g)$. Then $b$ is
continuous because it is locally constant. If $g \in \Gg_\Lambda^{(0)}$, then $\mu_g = \nu_g$, and
hence $b(g) = 0$. So $b \in \Cgpd1(\Gg_\Lambda, A)$. The categorical $2$-cocycle identity for $c$
implies that
\begin{align*}
\big(c(\mu_g,\lambda_g\alpha') &{}- c(\nu_g,\lambda_g\alpha')\big)
    - \big(c(\mu_g\lambda_g,\alpha') - c(\nu_g\lambda_g,\alpha')\big) \\
    &= \big(c(\mu_g, \lambda_g) - c(\lambda_g, \alpha')\big) - \big(c(\nu_g, \lambda_g) - c(\lambda_g, \alpha')\big)
    = b(g).
\end{align*}
This and a symmetric calculation yield
\[
\sigma^\Pp_c(g, h) - \sigma^\Qq_c(g, h)
    = b(g) + b(h) - b(gh) = (\dgpd1b)(g,h).
\]
Hence $\sigma_c$ and $\tilde\sigma_c$ are cohomologous.
\end{proof}

We now show that there always exists a set $\Pp$ producing a partition of $\Gg_\Lambda$ as
hypothesised in the preceding results.

\begin{lem}\label{lem:twist-section}
Let $\Lambda$ be a row-finite $k$-graph with no sources.  Then there exists a subset $\Pp
\subset \Lambda \mathbin{{_s*_s}} \Lambda$ such that $(\lambda, s(\lambda)) \in \Pp$ for all
$\lambda \in \Lambda$ and $\{Z(\mu,\nu) : (\mu,\nu) \in \Pp\}$ is a partition of $\Gg_\Lambda$.
\end{lem}
\begin{proof}
First observe that $Z(\lambda, s(\lambda)) \cap Z(\mu,\nu)$ is nonempty if and only if $\mu =
\lambda\nu$, in which case, $Z(\mu,\nu) \subset Z(\lambda, s(\lambda))$. Hence, basic open sets of
the form $Z(\lambda, s(\lambda))$ are pairwise disjoint.  A sequential argument shows that
\[
X: = \bigsqcup_{\lambda \in \Lambda} Z(\lambda, s(\lambda))
\]
is clopen in $\Gg_\Lambda$. So there is a countable collection $\Uu$ of basic open
sets of the form $Z(\mu,\nu)$ whose union is $\Gg_\Lambda \setminus X$.

It remains to show that there is a collection of pairwise disjoint basic open sets $\mathcal{V}$
such that $\bigcup \mathcal{V} = \bigcup \mathcal{U}$. We first show that, given two basic open
sets $Z(\mu_1, \nu_1)$, $Z(\mu_2, \nu_2)$, both $Z(\mu_1, \nu_1) \cap Z(\mu_2, \nu_2)$ and
$Z(\mu_1, \nu_1) \setminus Z(\mu_2, \nu_2)$ may be expressed as finite disjoint unions of such
basic open sets. We may assume that $d(\mu_1)-d(\nu_1) = d(\mu_2)-d(\nu_2)$, since $Z(\mu_1, \nu_1)
\cap Z(\mu_2, \nu_2)$ is empty otherwise. Recall from~\eqref{eq:wedge implements intersection} that
\[
Z(\mu_1, \nu_1) \cap Z(\mu_2, \nu_2) =
\bigsqcup_{(\eta,\zeta) \in (\mu_1, \nu_1) \wedge (\mu_2,\nu_2)} Z(\eta, \zeta).
\]
Since $Z(\mu_1, \nu_1) = \bigsqcup_{d(\mu\alpha) = d(\mu_1) \vee d(\nu_1)} Z(\mu_1\alpha,
\nu_1\alpha)$, it follows that
\begin{equation}
\begin{split}
Z(\mu_1, \nu_1) \setminus Z(\mu_2, \nu_2) =
\bigsqcup\big\{Z(\mu_1\alpha, \nu_1\alpha) :{} &d(\mu_1\alpha) = d(\mu_1) \vee d(\nu_1),\\
                                           & (\mu_1\alpha,\nu_1\alpha) \not\in (\mu_1,\nu_1)\wedge(\mu_2,\nu_2)\big\}.
\end{split}\label{eq:setminus}
\end{equation}
Now by the standard inclusion-exclusion decomposition, $Z(\mu_1, \nu_1) \cup Z(\mu_2, \nu_2)$ is
also a finite disjoint union of basic open sets. The collection $\mathcal{U} = \{ U_1, U_2, \dots
\}$ may now be replaced by a pairwise disjoint collection $\mathcal{V}$ recursively. Set $\Vv_1 =
\{U_1\}$. Now suppose that $\Vv_i$ is a collection of mutually disjoint basic open sets such that
$\bigcup^i_{j=1} U_j = \bigsqcup \Vv_i$. Use~\eqref{eq:setminus} to write $U_{i+1} \setminus
\bigcup \Vv_i = \bigsqcup \Ww_{i+1}$ where $\Ww_{i+1}$ is a finite collection of mutually disjoint
sets of the form $Z(\eta,\zeta)$, and let $\Vv_{i+1} := \Vv_i \cup \Ww_{i+1}$. Then
$\bigcup^{i+1}_{j=1} U_j = \bigsqcup \Vv_{i+1}$. By induction we obtain the desired family $\Vv$ of
mutually disjoint basic open sets such that $\bigsqcup \Vv = \bigcup \Uu$.
\end{proof}

We can now prove that every twisted $k$-graph algebra admits a homomorphism into a twisted
$C^*$-algebra, in the sense of Renault (see \cite{Renault1980}), of the path  groupoid of the
corresponding $k$-graph. It follows that all of the generators of $C^*(\Lambda, c)$ are nonzero. We
will use our gauge-invariant uniqueness theorem in the next section to see that this homomorphism
is an isomorphism.

Recall from \cite{Renault1980} that involution and convolution in
$C_c(\Gg_\Lambda, \sigma_c)$ of $C^*(\Gg_\Lambda, \sigma_c)$ are given by
\[
\big(f * g\big)(\gamma) = \sum_{\alpha\beta = \gamma} \sigma_c(\alpha,\beta) f(\alpha)g(\beta)
    \quad\text{ and }\quad
f^*(\gamma) = \overline{\sigma_c(\gamma^{-1}, \gamma)} \overline{f(\gamma^{-1})}.
\]

\begin{thm}\label{thm: C*(Lambda,c) cong C*(G,omega)}
Let $\Lambda$ be a row-finite $k$-graph with no sources. Let $\Pp$ be a subset of
$\Lambda \mathbin{{_s*_s}} \Lambda$ containing $\{(\lambda, s(\lambda)) : \lambda \in \Lambda\}$
such that $\{Z(\mu,\nu) : (\mu,\nu) \in \Pp\}$ is a partition of $\Gg_\Lambda$. Fix $c \in
\Zcat2(\Lambda,\TT)$, and let $\sigma_c \in \tilde{Z}^2(\Gg_\Lambda, \TT)$ be the cocycle
constructed from $\Pp$ as in Lemma~\ref{lem:groupoid-2-cocycle}. Then there is a surjective
homomorphism $\pi : C^*(\Lambda,c) \to C^*(\Gg_\Lambda, \sigma_c)$ such that $\pi(s_\lambda) =
1_{Z(\lambda, s(\lambda))}$ for all $\lambda \in \Lambda$. Moreover, for each $(\mu,\nu) \in
\Lambda \mathbin{{_s*_s}} \Lambda$, there is a finite subset $F \subseteq s(\mu)\Lambda$ such that
$Z(\mu,\nu) = \bigsqcup_{\tau \in F} Z(\mu\tau, \nu\tau)$ and a function $a : F \to \TT$ such that
$\pi(s_\mu s^*_\nu) = \sum_{\tau \in F} a_\tau 1_{Z(\mu\tau, \nu\tau)}$.
\end{thm}
\begin{proof}
By the universal property of $C^*(\Lambda, c)$, to prove the first statement it suffices to show
that $t : \lambda \mapsto 1_{Z(\lambda, s(\lambda))}$ is a Cuntz-Krieger $(\Lambda,c)$-family in
$C^*(\Gg_\Lambda , \sigma_c)$. Calculations like those of \cite[Lemma~4.3]{KPRR1997} verify (CK1),
(CK3)~and~(CK4). It remains only to verify that $t_\mu t_\nu = c(\mu,\nu) t_{\mu\nu}$ whenever
$r(\nu) = s(\mu)$. Fix $h \in \Gg_\Lambda$. We have
\begin{equation}\label{eq:gpd product}
(t_\mu t_\nu)(h)
    = 1_{Z(\mu,s(\mu))} * 1_{Z(\nu, s(\nu))}(h)
    = \sum_{g \in \Gg_\Lambda} \sigma_c(g,g^{-1}h) 1_{Z(\mu,s(\mu))}(g) 1_{Z(\nu,s(\nu))}(g^{-1}h).
\end{equation}
For fixed $g$, putting $h' = g^{-1}h$, the product $1_{Z(\mu,s(\mu))}(g) 1_{Z(\nu,s(\nu))}(h')$ is
equal to $1$ if $h' = (\nu z, d(\nu), z)$ and $g = (\mu\nu z, d(\mu), \nu z)$ for some $z \in
\Lambda^\infty$, and to $0$ otherwise. So: there is at most one nonzero term in the sum on the
right-hand side of~\eqref{eq:gpd product}; there is such a term precisely when $h = (\mu\nu z,
d(\mu\nu), z)$; and then the nonzero term occurs with $g = (\mu\nu z, d(\mu), \nu z)$ and $h' =
(\nu z, d(\nu), z)$. Setting $z := s(h)$, $g := (\mu\nu z, d(\mu), \nu z)$ and $h' := (\nu z,
d(\nu), z)$, we have
\[
(t_\mu t_\nu)(h) = \sigma_c(g, h') 1_{Z(\mu\nu, s(\nu))}(h).
\]
We have $g \in Z(\mu, s(\mu))$, $h \in Z(\nu, s(\nu))$, and $gh' \in Z(\mu\nu, s(\mu\nu))$. Since
$(  \mu , s ( \mu ) )$, $( \nu , s ( \nu ) )$, $( \mu \nu , s ( \mu \nu ) )$ belong to $\Pp$, we have $\mu_g = \mu$, $\nu_g = s(\mu)$, $\mu_{h'} =
\nu$, $\nu_{h'} = s(\nu)$, $\mu_{gh'} = \mu\nu$, and $\nu_{gh'} = s(\mu\nu)$. Furthermore $\alpha
:= \nu$ and $\beta = \gamma := s(\nu)$ satisfy~\eqref{eq:alphbetgam}, and
hence~\eqref{eq:groupoid-cocycle-formula} yields
\[
\sigma_c(g,h') = c(\mu,\nu) - c(\nu, s(\nu))
                    + c(\nu, s(\nu)) - c(s(\nu), s(\nu))
                    - \big(c(\mu\nu, s(\nu)) - c(s(\nu), s(\nu)\big),
\]
which, since $c$ is normalised, collapses to $c(\mu,\nu)$. Hence
\[
(t_\mu t_\nu)(x,m,y) = c(\mu, \nu) 1_{Z(\mu\nu, s(\nu))}(x, m, y) = c(\mu,\nu)t_{\mu\nu}(x,m,y),
\]
establishing~(CK2). Hence there is a homomorphism $\pi : C^* ( \Lambda , c ) \to C^* ( \Gg_\Lambda
, \sigma_c )$ such that $\pi ( s_\lambda ) = 1_{Z ( \lambda , s ( \lambda ) )}$. We postpone the
proof that this map is surjective until we have established the final statement.

Fix $\mu,\nu$ with $s(\mu) = s(\nu)$. For $g \in \Gg_\Lambda$, we have
\begin{align*}
\pi(s_\mu s^*_\nu)(g)
    = \pi(s_\mu) \pi(s_\nu)^* (g)
    &= 1_{\mu, s(\mu)} * 1_{\nu, s(\nu)}^*(g) \\
    &= \sum_{hk = g} \sigma_c(h,k) \overline{\sigma_c(k^{-1}, k)} 1_{Z(\mu, s(\mu))}(h)1_{Z(\nu, s(\nu))}(k^{-1}).
\end{align*}
There is at most one nonzero term in the sum, and this occurs when $g = (\mu x, d(\mu) - d(\nu),
\nu x)$, in which case $h = (\mu x, d(\mu), x)$ and $k = (x, -d(\nu), \nu x)$ for some $x \in
\Lambda^\infty$. So $\pi(s_\mu s^*_\nu)$ is supported on $Z(\mu,\nu)$, and for $g = (\mu x, d(\mu)
- d(\nu), \nu x)$ we have
\[
\pi(s_\mu s^*_\nu)(g) =
        \sigma_c\big((\mu x, d(\mu),x), (x, -d(\nu), \nu x)\big)
            \sigma_c\big((\nu x, d(\nu), x), (x, -d(\nu), \nu x)\big).
\]
Since $\sigma_c$ is locally constant and $\TT$-valued by construction, it follows that $\pi (
s_{\mu} s_{\nu}^* )$ can be written as a linear combination of the desired form.

For surjectivity, observe as in the proof of \cite[Corollary~3.5(i)]{KP2000} that $C^*(\Gg_\Lambda, \sigma_c) = \clsp\{1_{Z(\mu,\nu)} : s(\mu) = s(\nu)\}$, so it suffices to show that each $1_{Z (\mu , \nu )}$ is in the image of $\pi$. Fix $(\mu, \nu) \in \Lambda \mathbin{{_s*_s}} \Lambda$ and express $\pi(s_{\mu}s_{\nu}^*) = \sum_{\tau \in F} a_\tau 1_{Z(\mu\tau, \nu\tau)}$ as above.

A routine calculation using the definitions of convolution and involution in $C_c(\Gg_\Lambda,
\sigma_c)$ shows that
\[
\pi(s_{\nu\tau}s^*_{\nu\tau})
    = 1_{Z(\nu\tau, s(\tau))} * 1^*_{Z(\nu\tau, s(\tau))}
    = 1_{Z(\nu\tau)}\quad\text{ for each $\tau \in F$.}
\]
That each $1_{Z(\nu\tau)} \in C_0(\Gg_\Lambda^{(0)})$, and that $\sigma_c$ is a normalised cocycle
imply that $f * 1_{Z(\nu\tau)}(g) = f(g) 1_{Z(\nu\tau)}(s(g))$ for all $f \in C_c(\Gg_\Lambda,
\sigma_c)$. Since the range map on $\Gg_\Lambda$ is bijective on $Z_{(\mu,\nu)}$ and the
$Z_{(\mu\tau, \nu\tau)}$ are mutually disjoint, we obtain
\[
1_{Z(\mu\tau, \nu\tau)} 1_{Z(\nu\tau')} = \delta_{\tau, \tau'} 1_{Z(\mu\tau, \nu\tau)}
    \quad\text{ for all $\tau,\tau' \in F$.}
\]
Thus
\[
 1_{Z ( \mu , \nu )} = \Big( \sum_{\tau \in F} a_\tau 1_{Z(\mu\tau, \nu\tau)}  \Big)
    \Big(\sum_{\tau \in F} \overline{a}_\tau 1_{Z ( \nu \tau)} \Big)
    = \pi ( s_\mu s_\nu^* ) \sum_{\tau \in F} \overline{a}_\tau \pi ( s_{\nu \tau} s_{\nu \tau }^* ) . \qedhere
\]
\end{proof}

\begin{cor}\label{cor:gens nonzero}
Let $\Lambda$ be a row-finite $k$-graph with no sources, and fix $c \in \Zcat{2}(\Lambda, \TT)$.
Let $\{s_\lambda : \lambda \in \Lambda\}$ be the universal generating Cuntz-Krieger $(\Lambda,
c)$-family in $C^*(\Lambda, c)$. Then each $s_\lambda \not= 0$.
\end{cor}
\begin{proof}
By Theorem~\ref{thm: C*(Lambda,c) cong C*(G,omega)}, there exist $\sigma \in
\tilde{Z}^2(\Gg_\Lambda, \TT)$ and a homomorphism $\pi : C^*(\Lambda, c) \to C^*(\Gg_\Lambda,
\sigma)$ such that $\pi(s_\lambda) = 1_{Z(\lambda, s(\lambda))}$ for all $\lambda \in \Lambda$.
Since $\Lambda$ has no sources, each $Z(\lambda, s(\lambda))$ is nonempty, and hence each
$1_{Z(\lambda, s(\lambda))}$ is nonzero. So each $s_\lambda$ is nonzero.
\end{proof}

We conjecture that for each $r \ge 0$, there is an injective homomorphism from $\Hcat{r}(\Lambda,
A)$ to $\Hgpd{r}(\Gg_\Lambda, A)$. For $r = 0$ it is clear that there is such a homomorphism
determined by the map from $\Ccat{0}(\Lambda, A)$ to $\tilde{C}^0(\Gg_\Lambda, A)$ which takes $f :
\Lambda^0 \to A$ to the function $x \mapsto f(r(x))$ from $\Gg_\Lambda^{(0)}$ to $A$. The following
remark indicates how to define such a homomorphism for $r = 1$. It is not clear to us how to
proceed for $r \ge 3$.

\begin{rmk}\label{rmk:1cocycle}
Let $\Lambda$ be a row-finite $k$-graph with no sources and let $A$ be an abelian group.  Let $c
\in \Zcat1(\Lambda, A)$. As observed in \cite[\S 5]{KP2000} the function $\tilde{c} : \Gg_\Lambda
\to A$ given by
\[
\tilde{c}(x, \ell - m, y) = c(x(0, \ell)) - c(y(0, m))
\]
defines an element of $\Zgpd1(\Gg_\Lambda, A)$. It is straightforward to check that $c \mapsto
\tilde{c}$ is a homomorphism of abelian groups and that it preserves coboundaries (indeed, if $c =
\dcat1b$, then $\tilde{c} =  \dgpd1\tilde{b}$ where $\tilde{b}(x) = b(x(0))$). Hence, the map $[c]
\mapsto [\tilde{c}]$ defines a homomorphism $\Hcat1(\Lambda, A) \to \Hgpd1(\Gg_\Lambda, A)$. We now
show by example that this map need not be surjective

Recall that $B_2$ is the path category, regarded as a $1$-graph, of the directed graph with a
single vertex $\star$ and two edges, $B_2^1 = \{ f_1, f_2 \}$.  We have $H_0(B_2) \cong \ZZ$, and
\cite[Examples~4.11(1)]{kps3} shows that $H_1(B_2) \cong \ZZ^2$.  Hence by the universal
coefficient theorem \cite[Theorem~7.3]{kps3} and Theorem~\ref{thm:cocycles} we have $\Hcat1(B_2,
\ZZ) \cong \ZZ^2$. The path space $B_2^\infty$ may be identified with the sequence space $X =
\{(x_i)^\infty_{i=1} : x_i \in \{f_1, f_2\}\text{ for all $i$}\}$.  By \cite[Theorem~2.2]{DKP01},
$\Zgpd1(\Gg_{B_2}, \ZZ) \cong H^0(X, \ZZ)$ (which may be identified with the group of continuous
maps $h : X \to \ZZ$) and
\[
\Hgpd1(\Gg_{B_2}, \ZZ) \cong \coker\big(1 - \sigma^* : H^0(X, \ZZ) \to H^0(X, \ZZ)\big)
\]
where $\sigma : X  \to X$ is the shift and $(\sigma^* h)(x) = h(\sigma x)$. Given a path $\lambda
\in B_2$, let $h_\lambda = 1_{Z(\lambda)}$ be the characteristic function of the cylinder set
$Z(\lambda)$. Then $H^0(X, \ZZ)$ is spanned by $\{h_\lambda : \lambda \in \Lambda\}$.  Observe that
$\sigma^* h_\lambda = h_{f_1\lambda} + h_{f_2\lambda}$.  It follows that the map $\varphi : H^0(X,
\ZZ) \to \ZZ[1/2]$ determined by $\varphi (h_\lambda) = 2^{-d(\lambda)}$ satisfies the condition
$\varphi(1 - \sigma^*) = 0$.  Since $\varphi$ is surjective and vanishes on coboundaries it
determines a surjective map $\tilde\varphi :  \Hgpd1(\Gg_\Lambda, A) \to \ZZ[1/2]$. Moreover, the
range of the map
\[
c \in \Zcat1(B_2, \ZZ)\mapsto
\tilde{c} \in \Zgpd1(\Gg_{B_2}, \ZZ) \cong H^0(X, \ZZ)
\]
is $\ZZ h_{f_1} + \ZZ h_{f_2}$. Hence, $\tilde\varphi([\tilde{c}]) \in \ZZ$ and the induced map $\Hcat1(\Lambda, A) \to \Hgpd1(\Gg_\Lambda, A)$ is not surjective.
\end{rmk}

\section{The gauge-invariant uniqueness theorem} \label{sec:giut}

As for untwisted $k$-graph $C^*$-algebras, each twisted $k$-graph $C^*$-algebra carries a gauge
action of the $k$-torus. We establish a gauge-invariant uniqueness theorem and deduce that the
homomorphism of Theorem~\ref{thm: C*(Lambda,c) cong C*(G,omega)} is an isomorphism. Where arguments
in this section closely parallel proofs of the corresponding results for untwisted $k$-graph
$C^*$-algebras, we have sometimes outlined proofs without giving full details.

\begin{rmk}\label{rmk:orth range proj}
Let $t$ be a Cuntz-Krieger $(\Lambda,c)$-family.
Since a sum of projections is itself a projection if and only
if the original projections are mutually orthogonal, relations
(CK1), (CK3)~and~(CK4) ensure that for fixed $n \in \NN^k$ and
distinct $\mu,\nu \in \Lambda^n$, we have $t_\mu t^*_\mu t_\nu
t^*_\nu = 0$.
\end{rmk}

\begin{lem}\label{lem:spanning}
Let $\Lambda$ be a row-finite $k$-graph with no sources, fix $c \in \Zcat2(\Lambda, \TT)$, and let
$t$ be a Cuntz-Krieger $(\Lambda,c)$-family. For $\mu,\nu, \eta,\zeta \in \Lambda$ and $n \ge
d(\nu) \vee d(\eta)$,
\begin{align}
t^*_\nu t_\eta
    &= \sum_{\nu\alpha = \eta\beta \in \Lambda^n} \overline{c(\nu,\alpha)}c(\eta,\beta) t_\alpha t^*_\beta
        \quad\text{ and}\label{eq:inside product} \\
t_\mu t^*_\nu t_\eta t^*_\zeta
    &= \sum_{\nu\alpha = \eta\beta \in \Lambda^n}
            c(\mu,\alpha)\overline{c(\nu,\alpha)}c(\eta,\beta)\overline{c(\zeta,\beta)} t_{\mu\alpha} t^*_{\zeta\beta}.
            \label{eq:stars to right}
\end{align}
Consequently, $\lsp\{ s_\mu s^*_\nu : (\mu,\nu) \in \Lambda \mathbin{{_s*_s}} \Lambda\}$ is a dense
$*$-subalgebra of $C^*(\Lambda, c)$.
\end{lem}
\begin{proof}
Using~(CK4), we have
\[
t_\nu^* t_\eta
    = \sum_{\lambda \in r(\eta)\Lambda^n} t^*_\nu t_\lambda t^*_\lambda t_\eta.
\]
Since $n \ge d(\nu) \vee d(\eta)$, given $\lambda \in \Lambda^n$, the factorisation property allows
us to factor $\lambda = \lambda^1 \lambda^2 = \lambda^3\lambda^4$ where $d(\lambda^1) = d(\nu)$ and
$d(\lambda^3) = d(\eta)$, and then by~(CK2), (CK3)~and Remark~\ref{rmk:orth range proj}, we have
\begin{align*}
t_\nu^* t_\eta
    &= \sum_{\lambda \in r(\eta) \Lambda^n}
        \overline{c(\lambda^1,\lambda^2)} c(\lambda^3,\lambda^4)
        t^*_\nu t_{\lambda^1} t_{\lambda^2} t^*_{\lambda^4} t^*_{\lambda^3} t_\eta\\
    &= \sum_{\nu\alpha=\eta\beta \in \Lambda^n}
        \overline{c(\nu,\alpha)} c(\eta,\beta)
        t_{r(\alpha)} t_{\alpha} t^*_{\beta} t_{r(\beta)}.
\end{align*}
Equation~\ref{eq:inside product} follows from~(CK2) because $c$ is normalised. Left-multiplying by
$t_\mu$, right-multiplying by $t^*_\zeta$ and applying~(CK2) establishes~\eqref{eq:stars to right}.
\end{proof}

\begin{rmk}
If we apply Lemma~\ref{lem:spanning} with $n = d(\nu) \vee d(\eta)$, then we obtain
\begin{align*}
t^*_\nu t_\eta
    &= \sum_{\nu\alpha = \eta\beta \in \MCE(\nu,\eta)} \overline{c(\nu,\alpha)}c(\eta,\beta) t_\alpha t^*_\beta.
        \quad\text{ and}\\
t_\mu t^*_\nu t_\eta t^*_\zeta
    &= \sum_{\nu\alpha = \eta\beta \in \MCE(\mu,\nu)}
        c(\mu,\alpha)\overline{c(\nu,\alpha)}c(\eta,\beta)\overline{c(\zeta,\beta)} t_{\mu\alpha} t^*_{\zeta\beta}.
\end{align*}
\end{rmk}

\begin{lem}\label{lem:fpa}
Let $\Lambda$ be a row-finite $k$-graph with no sources and let $c \in \Zcat{2}(\Lambda, \TT)$. There
is a strongly continuous action $\gamma : \TT^k \to \Aut C^* ( \Lambda , c )$ such that
\[
\gamma_z(s_\lambda) = z^{d(\lambda)} s_\lambda\quad\text{ for all $\lambda \in \Lambda$.}
\]
Moreover $C^*(\Lambda, c)^{\gamma} = \clsp\{ s_\mu s^*_\nu : (\mu,\nu) \in \Lambda
\mathbin{{_s*_s}} \Lambda, d(\mu) = d(\nu)\}$.
\end{lem}
\begin{proof}
Fix $z \in \TT^k$ and set $t^z_{\lambda} = z^{d(\lambda)}s_\lambda$ for all $\lambda \in \Lambda$.
It is routine to verify that $\lambda \mapsto t^z_\lambda$ satisfies (CK1)--(CK4) of
Definition~\ref{dfn:twisted family}, so it is a Cuntz-Krieger $(\Lambda,c)$-family. Therefore the
universal property of $C^*(\Lambda, c)$ yields an endomorphism $\gamma_z$ of $C^*(\Lambda, c)$
satisfying $\gamma_z(s_\lambda) = t^z_{\lambda}= z^{d(\lambda)}s_\lambda$ for all $\lambda \in
\Lambda$ with inverse $\gamma_{\overline{z}}$; hence, $\gamma_z$ is an automorphism. Since
$\gamma_{wz}(s_\lambda) = \gamma_w(\gamma_z(s_\lambda))$ for all $\lambda \in \Lambda$, the map
$\gamma : \TT \to \Aut(C^*(\Lambda, c))$ is a homomorphism. That $z \mapsto \gamma_z(a)$ is
continuous for each $a$ follows from an $\varepsilon/3$-argument using Lemma~\ref{lem:spanning}.
Thus, $\gamma$ defines a strongly continuous action of $\TT^k$ on $C^*(\Lambda,c)$ with the desired
property.

That $C^*(\Lambda, c)^{\gamma} = \clsp\{t_\mu t^*_\nu : (\mu,\nu) \in \Lambda \mathbin{{_s*_s}}
\Lambda, d(\mu) = d(\nu)\}$ is standard: the containment ``$\supseteq$" is clear and the reverse
containment follows from the observation that the faithful conditional expectation $\Phi^c(a) :=
\int_{\TT^k} \gamma_z(a)\,dz$ onto $C^*(\Lambda,c)^{\gamma}$ annihilates $s_\mu s^*_\nu$ whenever
$d(\mu) \not= d(\nu)$.
\end{proof}

\begin{ntn}
Let $X$ be a countable set. We write $\Kk_X$ for the universal $C^*$-algebra generated by matrix
units $\{\theta_{x,y} : x,y \in X\}$ (see
\cite[Remark~A.10]{Raeburn:Graphalgebras05}). That is the $\theta_{x,y}$ satisfy $\theta_{x,y}^* =
\theta_{y,x}$ and $\theta_{w,x}\theta_{y,z} = \delta_{x,y} \theta_{w,z}$. The algebra $\Kk_X$ is
simple, and is canonically isomorphic to $\Kk(\ell^2(X))$. If $X$ is finite, then any enumeration
of $X$ induces an isomorphism $\Kk_X \cong M_{|X|}(\CC)$.
\end{ntn}

\begin{prop}\label{prp:gauge action and fpa}
Let $\Lambda$ be a row-finite $k$-graph with no sources.  Fix $c \in \Zcat2(\Lambda, \TT)$.
\begin{enumerate}
\item\label{it:compact summands} For $n \in \NN^k$, $\clsp\{s_\mu s^*_\nu : \mu,\nu \in
    \Lambda^n\}$ is a $C^*$-subalgebra of $C^*(\Lambda,c)^\gamma$, and the assignment $s_\mu
    s^*_\nu \mapsto \theta_{\mu,\nu}$ determines an isomorphism
    \[
        \clsp\{s_\mu s^*_\nu : \mu,\nu \in \Lambda^n\} \cong \bigoplus_{v \in \Lambda^0} \Kk_{\Lambda^n v}.
    \]
\item\label{it:nested} If $m \le n \in \NN^k$, then $\clsp\{s_\mu s^*_\nu : \mu,\nu \in
    \Lambda^m\} \subseteq \clsp\{s_\mu s^*_\nu : \mu,\nu \in \Lambda^n\}$. Moreover,
    \[
    C^*(\Lambda, c)^\gamma = \overline{\bigcup_{n \in \NN^k} \lsp\{s_\mu s^*_\nu : \mu,\nu \in \Lambda^n\}} \qquad\text{is AF.}
    \]
\item\label{it:injectivity} Given a Cuntz-Krieger $(\Lambda,c)$-family $t$, the homomorphism
    $\pi_t$ induced by the universal property restricts to an injection on
    $C^*(\Lambda,c)^\gamma$ if and only if each $t_v$ is nonzero.
\item\label{it:cocycle-independent}  For any two cocycles $c_1, c_2 \in \Zcat2(\Lambda,\TT)$,
    the fixed-point algebras $C^*(\Lambda, c_1)^\gamma$ and $C^*(\Lambda, c_2)^\gamma$ are
    isomorphic.
\end{enumerate}
\end{prop}
\begin{proof} (\ref{it:compact summands}) Lemma~\ref{lem:fpa} implies that
\begin{equation}\label{eq:dir lim}
C^*(\Lambda,c)^\gamma
    = \clsp\{s_\mu s^*_\nu : d(\mu) = d(\nu)\}
    = \overline{\bigcup_{n \in \NN^k} \lsp\{s_\mu s^*_\nu : \mu,\nu \in \Lambda^n\}}.
\end{equation}
Furthermore, equation~\eqref{eq:stars to right} and that $c(\mu, s(\mu)) = 1$ implies that if
$\mu, \nu, \eta, \zeta \in \Lambda^n$ with $s(\mu) = s(\nu)$ and $s(\eta) = s(\zeta)$, then
$s_\mu s^*_\nu s_\eta s^*_\zeta = \delta_{\mu,\nu} s_\mu s^*_\zeta$. So for each $n \in \NN^k$,
the subspace $\clsp\{s_\mu s^*_\nu : \mu,\nu \in \Lambda^n\}$ is closed under multiplication.
Since it is clearly closed under involution, it is a $C^*$-subalgebra of $C^*(\Lambda,
c)^\gamma$. That $s_\mu s^*_\nu \mapsto \theta_{\mu,\nu}$ determines the desired isomorphism
with $\bigoplus_{v \in \Lambda^0} \Kk_{\Lambda^n v}$ follows from the uniqueness of the latter.

(\ref{it:nested}) Relations (CK2)~and~(CK4) give $\clsp\{s_\mu
s^*_\nu : \mu,\nu \in \Lambda^m\} \subseteq \clsp\{s_\mu s^*_\nu : \mu,\nu \in 
\Lambda^n\}$ for $m \le n$.
That $C^*(\Lambda, c)^\gamma = \overline{\bigcup_{n \in \NN^k} \lsp\{s_\mu s^*_\nu : \mu,\nu
\in \Lambda^n\}}$ follows from Lemma~\ref{lem:fpa}.  Since  $C^*(\Lambda, c)^\gamma$ is
an inductive limit of AF algebras it is also AF.

(\ref{it:injectivity}) The ``only if'' follows from Corollary~\ref{cor:gens nonzero}. For the ``if''
implication, observe that each minimal projection $s_\mu s^*_\mu$ in $\clsp\{s_\mu s^*_\nu :
\mu,\nu \in \Lambda^n\}$ is equivalent to $s^*_\mu s_\mu = s_{s(\mu)}$. So if each $s_v$ is
nonzero, then each $s_\mu s^*_\mu$ is nonzero. The result then follows from the direct-limit
decomposition~\eqref{eq:dir lim}, the simplicity of each $\Kk_{\Lambda^n v}$ and the ideal
structure of direct sums of $C^*$-algebras.

(\ref{it:cocycle-independent}) Fix $n \in \NN^k$ and $v \in \Lambda^0$, and fix cocycles $c_1,
c_2 \in \Zcat2(\Lambda,\TT)$. The assignments $s^{c_1}_\mu(s^{c_1}_\nu)^* \mapsto
\theta_{\mu,\nu}\mapsto s^{c_2}_\mu(s^{c_2}_\nu)^*$ determine isomorphisms
\[
\clsp\{s^{c_1}_\mu(s^{c_1}_\nu)^* : \mu,\nu \in \Lambda^n\}  \cong \Kk_{\Lambda^n v}
    \cong \clsp\{s^{c_2}_\mu(s^{c_2}_\nu)^* : \mu,\nu \in \Lambda^n\}.
\]

Moreover, for each pair $v,w \in \Lambda^0$ and $m \le n \in \NN^k$, the multiplicity of the
partial inclusion
\[
    \Kk_{\Lambda^m v}
        \cong \clsp\{s^{c_1}_\mu(s^{c_1}_\nu)^* : \mu,\nu \in \Lambda^m v\}
        \hookrightarrow \clsp\{s^{c_1}_\mu(s^{c_1}_\nu)^* : \mu,\nu \in \Lambda^n w\}
        \cong \Kk_{\Lambda^n w}
\]
is $|v \Lambda^{n-m} w|$ which does not depend on the cocycle $c_1$. Since AF algebras are
completely determined by the dimensions of the summands of the approximating subalgebras and by the
multiplicities of the partial inclusions, this proves the result.
\end{proof}

With the preceding analysis in hand, we can prove a version of an Huef and Raeburn's
gauge-invariant uniqueness theorem for twisted $k$-graph $C^*$-algebras.

\begin{cor}[The gauge-invariant uniqueness theorem]\label{cor:giut}
Let $\Lambda$ be a row-finite $k$-graph with no sources, and fix $c \in \Zcat2(\Lambda,\TT)$. Let
$t : \Lambda \to B$ be a Cuntz-Krieger $(\Lambda,c)$-family in a $C^*$-algebra $B$. Suppose that
there is a strongly continuous action $\beta$ of $\TT^k$ on $B$ satisfying $\beta_z(t_\lambda) =
z^{d(\lambda)} t_\lambda$ for all $\lambda \in \Lambda$ and $z \in \TT^k$. Then the induced
homomorphism $\pi_t : C^*(\Lambda,c) \to B$ is injective if and only if $t_v \not= 0$ for all $v
\in \Lambda^0$.
\end{cor}
\begin{proof}
The ``only if" direction follows from Corollary~\ref{cor:gens nonzero}.

The ``if'' direction follows from the following standard argument. Let $\Phi^\gamma :
C^*(\Lambda,c) \to C^*(\Lambda,c)^\gamma$ and $\Phi^\beta : C^*(t) \to C^*(t)^\gamma$ be the
conditional expectations obtained by averaging over $\gamma$ and $\beta$. Then $\Phi^\gamma$
is a faithful conditional expectation, and $\pi \circ \Phi^\gamma = \Phi^\beta \circ
\pi$. So for $a \in C^*(\Lambda,c)$, we have
\[
\pi(a) = 0
    \implies \pi(a^*a) = 0
    \implies \Phi^\beta(\pi_t(a^*a)) = 0
    \implies \pi(\Phi^\gamma(a^*a)) = 0.
\]
This forces $\Phi^\gamma(a^*a) = 0$ because $\pi$ restricts to an injection on
$C^*(\Lambda,c)^\gamma$ by Proposition~\ref{prp:gauge action and fpa}(\ref{it:injectivity}). Hence
$a^*a = 0$ because $\Phi^\gamma$ is faithful on positive elements, and then $a = 0$ by the
$C^*$-identity.
\end{proof}

\begin{cor}\label{cor:special groupoid algebra isomorphism}
Let $\Lambda$ be a row-finite $k$-graph with no sources. Let $\Pp$ be a subset of
$\Lambda \mathbin{{_s*_s}} \Lambda$ such that $\{(\lambda, s(\lambda)) : \lambda \in \Lambda\}
\subseteq \Pp$ and $\{Z(\mu,\nu) : (\mu,\nu) \in \Pp\}$ is a partition of $\Gg_\Lambda$. Fix $c \in
\Zcat2(\Lambda,\TT)$.  Then the homomorphism $\pi : C^*(\Lambda, c) \to C^*(\Gg_\Lambda, \sigma_c)$
of Theorem~\ref{thm: C*(Lambda,c) cong C*(G,omega)} satisfying $\pi(s_\lambda) = 1_{Z(\lambda,
s(\lambda))}$ for all $\lambda \in \Lambda$ is an isomorphism.
\end{cor}
\begin{proof}
We showed in Theorem~\ref{thm: C*(Lambda,c) cong C*(G,omega)} that $\pi$ is a surjective
homomorphism, so it remains to show that it is injective. There is a strongly continuous action
$\beta$ of $\TT^k$ on $C^*(\Gg_\Lambda, \sigma_c)$ satisfying
\[
\beta_z (f)(x, \ell - m, y) = z^{ \ell - m}f(x, \ell - m, y)
\]
for all $f \in C_c(\Gg_\Lambda, \sigma_c)$, all $z \in \TT^k$ and all
$(x, \ell - m, y) \in \Gg_\Lambda$.  Moreover, $\beta_z \circ
\pi = \pi \circ \gamma_z$ for all $z$: for $z \in \TT^k$ and $\lambda \in \Lambda$,
\[
\pi(\gamma_z(s_\lambda)) = z^{d(\lambda)}1_{Z(\lambda, s(\lambda))}
= \beta_z(1_{Z(\lambda, s(\lambda))})
=  \beta_z(\pi(s_\lambda)).
\]
Since each $Z(v) \not= \emptyset$, each $\pi(s_v) = 1_{Z(v)}$ is nonzero, so
Corollary~\ref{cor:giut} implies that $\pi$ is injective.
\end{proof}

\begin{cor}\label{cor:groupoid algebra isomorphism}
Let $\Lambda$ be a row-finite $k$-graph with no sources. Let $\Pp$ be a subset of
$\Lambda \mathbin{{_s*_s}} \Lambda$ such that $\{Z(\mu,\nu) : (\mu,\nu) \in \Pp\}$ is a partition
of $\Gg_\Lambda$. Fix $c \in \Zcat2(\Lambda,\TT)$, and let $\sigma_c \in \Zgpd2(\Gg_\Lambda, \TT)$ be
the cocycle constructed from $\Pp$ as in Lemma~\ref{lem:groupoid-2-cocycle}. Then $C^*(\Lambda,c)
\cong C^*(\Gg_\Lambda, \sigma_c) \cong C^*_r(\Gg_\Lambda, \sigma_c)$.
\end{cor}
\begin{proof}
By Lemma~\ref{lem:twist-section} there exists a set $\Qq \subseteq \Lambda
\mathbin{{_s*_s}} \Lambda$ such that $\{Z(\mu,\nu) : (\mu,\nu) \in \Qq\}$ is a partition of
$\Gg_\Lambda$ and such that $(\lambda, s(\lambda)) \in \Qq$ for all $\lambda \in \Lambda$. Let
$\sigma^\Qq_c \in \Zgpd2(\Gg_\Lambda, \TT)$ be the cocycle constructed from $\Qq$ as in
Lemma~\ref{lem:groupoid-2-cocycle}. Corollary~\ref{cor:special groupoid algebra isomorphism}
implies that $C^*(\Lambda, c) \cong C^*(\Gg_\Lambda, \sigma^\Qq_c)$. Moreover,
Theorem~\ref{thm:graph->gpd cohomology} implies that $\sigma^\Qq_c$ and $\sigma_c$ are cohomologous
in $\tilde{Z}^2(\Gg_\Lambda, \TT)$, and then \cite[Proposition~II.1.2]{Renault1980} implies that
$C^*(\Gg_\Lambda, \sigma^\Qq_c) \cong C^*(\Gg_\Lambda, \sigma_c)$.

For the assertion that $C^*(\Gg_\Lambda, \sigma_c) = C^*_r(\Gg_\Lambda, \sigma_c)$, let $\psi :
C^*(\Lambda, c) \to C^*_r(\Gg_\Lambda, \sigma^{\Qq}_c)$ be the homomorphism obtained by composing
the quotient map $q : C^*(\Gg_\Lambda, \sigma^{\Qq}_c) \to C^*_r(\Gg_\Lambda, \sigma^{\Qq}_c)$ with
the isomorphism $C^*(\Lambda, c) \cong C^*(\Gg_\Lambda, \sigma^{\Qq}_c)$ of
Corollary~\ref{cor:special groupoid algebra isomorphism}. The $\ZZ^k$-grading of $\Gg_\Lambda$
induces a strongly continuous $\TT^k$-action on $C^*_r(\Gg_\Lambda, \sigma^{\Qq}_c)$ which is
compatible under $\psi$ with the gauge action on $C^*(\Lambda, c)$. So the argument of the
preceding paragraph also applies to the reduced $C^*$-algebra, giving $C^*(\Gg_\Lambda, \sigma_c)
\cong C^*(\Lambda, c) \cong C_r^*(\Gg_\Lambda, \sigma_c)$.
\end{proof}

\section{Structure theory}\label{sec:structure}

In this section we establish some structure theorems for twisted $k$-graph $C^*$-algebras. We begin
with a version of the Cuntz-Krieger uniqueness theorem and a simplicity result that follow from
Renault's structure theory for groupoid $C^*$-algebras \cite{Renault1980} and
Corollary~\ref{cor:groupoid algebra isomorphism}.

Recall from \cite[Definition~3.1]{LS2010} that a row-finite $k$-graph with no sources is said to be
\emph{aperiodic} if for every pair $\alpha,\beta$ of distinct elements of $\Lambda$ such that
$s(\alpha) = s(\beta)$, there exists $\tau \in s(\alpha)\Lambda$ such that $\MCE(\alpha\tau,
\beta\tau) = \emptyset$.

\begin{rmk}\label{rmk:aperiodicity}
The original aperiodicity condition~(A) of \cite{KP2000} insists that for each $v \in \Lambda^0$
there exists $x \in \Lambda^\infty$ with $r(x) = v$ such that for all $p \not= q \in \NN^k$, we
have $\sigma^p x \not= \sigma^q x$. Proposition~3.6 of \cite{LS2010} implies that condition~(A) is
equivalent to aperiodicity of $\Lambda$ in the sense described above, and
\cite[Lemma~3.2]{RobertsonSims:blms07} implies that this is also equivalent to the condition of
``no local periodicity" described there.
\end{rmk}

Recall from \cite[Definition~4.7]{KP2000} that a row-finite $k$-graph $\Lambda$ with no sources is
\emph{cofinal} if for every $x \in \Lambda^\infty$ and $v \in \Lambda^0$ there exists $n \in \NN^k$
such that $v \Lambda x(n) \not= \emptyset$.

\begin{cor}
Let $\Lambda$ be a row-finite $k$-graph with no sources. Suppose that $\Lambda$ is aperiodic. Fix
$c \in \Zcat2(\Lambda, \TT)$ and a Cuntz-Krieger $(\Lambda, c)$-family $t$ in a $C^*$-algebra $B$. If each $t_v \not= 0$ then the homomorphism $\pi_t : C^*(\Lambda, c) \to B$ is injective (so
$C^*(\Lambda, c) \cong C^*(t)$).  Moreover, $C^*(\Lambda, c)$ is simple if and only if $\Lambda$ is cofinal.
\end{cor}
\begin{proof}
By Remark~\ref{rmk:aperiodicity}, $\Lambda$ satisfies Condition~(A). Hence
\cite[Proposition~4.5]{KP2000} implies that $\Gg_\Lambda$ is topologically free in the sense
that the units with trivial isotropy are dense in $\Gg_\Lambda^{(0)}$. Since $C^*(\Lambda, c)
\cong C^*(\Gg_\Lambda, \sigma_c)$, the result now follows from
\cite[Proposition~II.4.6]{Renault1980} and the arguments of \cite[Theorem~4.6 and
Proposition~4.8]{KP2000}.
\end{proof}

\begin{rmk}
Let $\Lambda$ be a row-finite $k$-graph with no sources. Combining
Remark~\ref{rmk:aperiodicity} with \cite[Theorem~3.1]{RobertsonSims:blms07} shows that the
untwisted $C^*$-algebra $C^*(\Lambda)$ is simple if \emph{and only if} $\Lambda$ is both
aperiodic and cofinal. This is not the case in general for twisted $k$-graph algebras:
\cite[Example~7.7]{kps3} shows how to recover the irrational rotation algebras, which are
simple, as twisted $C^*$-algebras of a $2$-graph which fails the aperiodicity condition quite
spectacularly. So in general, simplicity of the untwisted $C^*$-algebra $C^*(\Lambda)$ implies
simplicity of each $C^*(\Lambda, c)$ but the converse does not hold.
\end{rmk}

We show next that each $C^*(\Lambda, c)$ is nuclear. Our argument follows closely that of
\cite[Theorem~5.5]{KP2000}. We first establish a technical result.

\begin{lem}[{cf. \cite[Lemma~5.4]{KP2000}}]\label{lem:d-coboundary}
Let $\Lambda$ be a row-finite $k$-graph with no sources, and suppose that the degree map on
$\Lambda$ is a coboundary. For each $c \in \Zcat2(\Lambda, \TT)$, the twisted $C^*$-algebra
$C^*(\Lambda, c)$ is AF, and is isomorphic to $C^*(\Lambda)$.
\end{lem}
\begin{proof}
Since $d$ is a coboundary, there exists $b \in \Ccat0(\Lambda, \ZZ^k)$ such that $d(\lambda) =
(\dcat0b)(\lambda) = b(s(\lambda)) - b(r(\lambda))$ for all $\lambda \in \Lambda$.

Fix $c \in \Zcat2(\Lambda, \TT)$. For $n \in \ZZ^k$, let
\[
A_n := \clsp\{s_\mu s^*_\nu : (\mu,\nu) \in \Lambda
\mathbin{{_s*_s}} \Lambda, b(s(\mu)) = n\} \subseteq C^*(\Lambda, c),
\]
and for $v \in \Lambda^0$ with $b(v) = n$, let $ A_{n,v} := \clsp\{s_\mu s^*_\nu : s(\mu) = s(\nu)
= v\} \subseteq A_n. $ Arguing as in the proof of \cite[Lemma~5.4]{KP2000} we see: that $A_m
\subseteq A_n$ for $m \le n \in \ZZ^k$; that $C^*(\Lambda, c) = \varinjlim_{n \in \ZZ^k} A_n$; that
$A_n = \bigoplus_{b(v) = n} A_{n,v}$ for each $n$; and that $s_\mu s^*_\nu \mapsto
\theta_{\mu,\nu}$ determines isomorphisms $A_{n,v} \cong \Kk_{\Lambda v}$ for each $n,v$. So
$C^*(\Lambda, c)$ is AF.

To calculate the multiplicities of the partial inclusions $A_{m,v} \to A_{n,w}$, fix $m \le n$ and
$v \in b^{-1}(m)$, and observe that if $s(\mu) = v$ then
\[
s_\mu s^*_\mu
    = \sum_{\alpha \in v\Lambda^{n-m}} s_{\mu} s_\alpha s^*_\alpha s^*_\mu
    = \sum_{b(w) = n} \sum_{\alpha \in v\Lambda w} s_{\mu\alpha} s^*_{\mu\alpha}.
\]
So for $w \in b^{-1}(n)$, the multiplicity of the partial inclusion of $A_{m,v}$ in $A_{n,w}$ is
$|v \Lambda w|$ and in particular does not depend on the cocycle $c$. Since AF algebras are
completely determined by the dimensions of the summands of the approximating subalgebras and by the
multiplicities of the partial inclusions, the isomorphism class of  $C^*(\Lambda, c)$  is
independent of $c$. In particular, $C^*(\Lambda, c) \cong C^*(\Lambda, 1) = C^*(\Lambda)$. 
\end{proof}

Suppose that $\Lambda$ is a $k$-graph, $A$ is a discrete abelian group, and $f : \Lambda \to A$
is a functor. The \emph{skew-product} $k$-graph $\Lambda \times_f A$ is the cartesian product
$\Lambda \times A$ with operations $r(\mu, a) = (r(\mu), a)$, $s(\mu, a) = (s(\mu), a +
f(\mu))$, $(\mu, a)(\nu, a + f(\mu)) = (\mu\nu, a)$ and $d(\mu, a) = d(\mu)$ (see
\cite[Definition~5.1]{KP2000}).

\begin{lem}[cf. {\cite[Corollary~5.3]{KP2000}}]\label{lem:coaction}
Let $\Lambda$ be a row-finite $k$-graph and let $c \in \Zcat2(\Lambda, \TT)$. Let $A$ be a discrete abelian group and $f : \Lambda \to A$ a functor. There is a strongly continuous action
$\alpha^f$ of $\widehat{A}$ on $C^*(\Lambda, c)$ such that $\alpha^f_\chi(s_\lambda) =
\chi(f(\lambda)) s_\lambda$ for all $\chi \in \widehat{A}$ and $\lambda \in \Lambda$. There is a
cocycle $\tilde{c} \in \Zcat2(\Lambda \times_f A, \TT)$ given by $\tilde{c}((\mu,a),(\nu, a +
f(\mu)) = c(\mu, \nu)$, and $C^*(\Lambda \times_f A, \tilde{c})$ is isomorphic to the
crossed-product $C^*(\Lambda, c) \rtimes_{\alpha^f} \widehat{A}$.
\end{lem}
\begin{proof}
Our proof follows that of \cite[Corollary~5.3]{KP2000} except that we must take into account the
cocycles $c$ and $\tilde{c}$.

The existence of the action $\alpha^f$ follows from the universal property of $C^*(\Lambda, c)$:
for each $\chi \in \widehat{A}$, the map $t : \lambda \mapsto \chi(f(\lambda)) s_\lambda$
determines a Cuntz-Krieger $(\Lambda, c)$-family in $C^*(\Lambda, c)$. Continuity follows from an
$\varepsilon/3$-argument.

The map $\tilde{c}$ is a 2-cocycle because $c$ is a $2$-cocycle and $(\mu,a) \mapsto \mu$ is a
functor.

Let $\Pp$ be a subset of $\Lambda \mathbin{{_s*_s}}
\Lambda$ containing $\{(\lambda, s(\lambda)) : \lambda \in \Lambda\}$ and such that $\{Z(\mu,\nu) :
(\mu,\nu) \in \Pp\}$ is a partition of $\Gg_\Lambda$ as in Lemma~\ref{lem:twist-section}. Then $\Qq
:= \{((\mu, a + f(\mu)), (\nu, a + f(\nu))) : (\mu,\nu) \in \Pp, a \in A\}$ gives a partition of
$\Gg_{\Lambda \times_f A}$ with the same properties. Since $f$ is a functor there is a well-defined 
$1$-cocycle $\overline{f}$ on $\Gg_\Lambda$ given by 
$\overline{f}(\alpha x, d(\alpha) - d(\beta), \beta x) = f(\alpha) - f(\beta)$. 
Let $\Gg_\Lambda(\overline{f})$ be  the skew-product groupoid of \cite{Renault1980}, and
let $\sigma_{\tilde{c}}$ be the cocycle on $\Gg_{\Lambda \times_f A}$ obtained from
Lemma~\ref{lem:groupoid-2-cocycle} applied to $\tilde{c} \in \Zcat2(\Lambda \times_f A, \TT)$ and
$\Qq$. Let $\sigma_c$ be the cocycle on $\Gg_\Lambda$ obtained in the same way from $c \in
\Zcat2(\Lambda, \TT)$ and $\Pp$. If $q$ denotes the quotient map $\Gg_\Lambda(\overline{f}) \to
\Gg_\Lambda$ then $\sigma_c \circ q$ is a continuous $2$-cocycle on $\Gg_\Lambda(\overline{f})$. By
\cite[Theorem~5.2]{KP2000}, the groupoid $\Gg_{\Lambda \times_f A}$ is isomorphic to
$\Gg_\Lambda(\overline{f})$. Moreover, this isomorphism carries $\sigma_{\tilde{c}}$ to $\sigma_c
\circ q$. We can now apply \cite[Theorem~II.5.7]{Renault1980} as in the proof of
\cite[Corollary~5.3]{KP2000}.
\end{proof}

\begin{rmk}
Note that $\tilde{c}$ is the pull-back of $c$ under the functor $\Lambda \times_f A \to \Lambda$
(given by $(\lambda, a) \mapsto \lambda$).
\end{rmk}

\begin{cor}[cf. {\cite[Theorem~5.5]{KP2000}}]
Let $\Lambda$ be a row-finite $k$-graph and let $c \in \Zcat2(\Lambda, \TT)$. Then $C^*(\Lambda,
c)$ belongs to the bootstrap class $\mathcal{N}$, and in particular is nuclear.
\end{cor}
\begin{proof}
We follow the proof of \cite[Theorem~5.5]{KP2000}. By Takai duality, we have
\[
C^*(\Lambda, c) \sim_{Me} C^*(\Lambda, c) \times_\gamma \TT^k \times_{\hat\gamma} \ZZ^k.
\]
Lemma~\ref{lem:coaction} implies that $C^*(\Lambda, c) \times_\gamma \TT^k \cong C^*(\Lambda
\times_d \ZZ^k, \tilde{c})$. Define $b : (\Lambda \times_d \ZZ^k)^0 \to \ZZ^k$ by $b(v,m) = m$.
Then the degree map on $\Lambda \times_d \ZZ^k$ is equal to $\dcat0 b$, so
Lemma~\ref{lem:d-coboundary} implies that $C^*(\Lambda, c) \times_\gamma \TT^k$ is AF. Hence
$C^*(\Lambda, c)$ is Morita equivalent to a crossed product of an AF algebra by $\ZZ^k$, which
proves the result.
\end{proof}

Finally, we consider pullbacks and cartesian products of $k$-graphs. Recall from \cite{KP2000} that
if $\Lambda$ is a $k$-graph and $f : \NN^l \to \NN^k$ is a homomorphism, then the pullback
$l$-graph $f^*\Lambda$ is defined by $f^*\Lambda = \{(\lambda,m) \in \Lambda \times \NN^l :
d(\lambda) = f(m)\}$ with coordinatewise operations and degree map $d(\lambda, m) = m$. Recall also
that if $\Lambda_1$ is a $k_1$-graph and $\Lambda_2$ is a $k_2$-graph, then $\Lambda_1 \times
\Lambda_2$ is a $(k_1 + k_2)$-graph with coordinatewise operations and degree map $d(\lambda_1,
\lambda_2) = (d(\lambda_1), d(\lambda_2))$.

\begin{cor}
\begin{enumerate}
\item Let $\Lambda$ be a row-finite $k$-graph with no sources. Fix $c \in \Zcat2(\Lambda, \TT)$
    and a homomorphism $f : \NN^l \to \NN^k$. There is a cocycle $f^*c$ on $f^*\Lambda$ given
    by $f^*c(\lambda, m) = c(\lambda)$, and there is a homomorphism $\pi_f : C^*(f^*\Lambda,
    f^*c) \to C^*(\Lambda, c)$ given by $\pi_f(s_{\lambda,m}) = s_\lambda$. If $f$ is
    injective, so is $\pi_f$. If $f$ is surjective, then $\pi_f$ is also, and $C^*(f^*\Lambda,
    f^*c) \cong C^*(\Lambda, c) \otimes C(\TT^{l-k})$.
\item For each $i \in \{1,2\}$, let $\Lambda_i$ be a row-finite $k_i$-graph and fix $c_i \in
    \Zcat2(\Lambda_i, \TT)$. Then $(c_1 \times c_2)(\lambda_1, \lambda_2) :=
    c_1(\lambda_1)c_2(\lambda_2)$ determines an element $c_1 \times c_2 \in \Zcat2(\Lambda_1
    \times \Lambda_2)$. Moreover, the formula $(s_{\lambda_1}, s_{\lambda_2}) \mapsto
    s_{\lambda_1} \otimes s_{\lambda_2}$ determines an isomorphism
    \[
    C^*(\Lambda_1 \times \Lambda_2, c_1 \times c_2) \cong
    C^*(\Lambda_1, c_1) \otimes C^*(\Lambda_2, c_2).
    \]
\end{enumerate}
\end{cor}
\begin{proof}
The arguments are more or less identical to those of \cite[Proposition~1.11]{KP2000} and
\cite[Corollary~3.5(iii)~and~(iv)]{KP2000}.
\end{proof}

\end{document}